\documentclass[12pt,chapterheads]{ucsd}
 \usepackage{amsmath, amscd, amssymb, amsthm}
\usepackage{mathrsfs}
\usepackage{rotating}
\usepackage[T1]{fontenc}
\usepackage{aurical}
\usepackage{appendix}

\newtheorem{theorem}{Theorem}[section]

\newtheorem{proposition}[theorem]{Proposition}
\newtheorem{lemma}[theorem]{Lemma}
\newtheorem{dfn}[theorem]{Definition}
\newtheorem{corollary}[theorem]{Corollary}
\newtheorem{conjecture}[theorem]{Conjecture}
\newcommand{\C}{\mathscr{C}}
\renewcommand{\O}{\mathcal{O}}

\renewcommand{\k}{\mbox{\Fontauri k}}

\begin{document}

%% REQUIRED FIELDS -- Replace with the values appropriate to you
\title{The Bessel-Plancherel theorem and applications}
% No symbols, formulas, superscripts, or Greek letters are allowed
% in your title.

\author{Raul Gomez}
\degreeyear{2011}
\degree{Doctor of Philosophy} 
% Master's Degree theses will NOT be formatted properly with this
% file.

\field{Mathematics}
\chair{Professor Nolan Wallach}
% Uncomment the next line iff you have a Co-Chair
\cochair{Professor Wee Teck Gan} 
\othermembers{%  These must be alpha by last name.
Professor Ronald Graham\\ 
Professor Ken Intriligator\\
Professor Cristian Popescu\\
}
\numberofmembers{5} % |chair| + |cochair| + |othermembers|

\begin{frontmatter}
\makefrontmatter % The title, copyright, and signature pages.

%% DEDICATION
% You have three choices here:
%   1. Use the ``dedication'' environment.   Put in the text you want,
%   and you'll get a perfectly respectable dedication page.
%
%   2. Use the ``mydedication'' environment.  If you don't like the
%   formatting of option 1, use this environment and format things
%   however you wish.
%
%   3. If you don't want a dedication, it's not required.

\begin{dedication} % The style file will format this for you.
  To Tanya Martinez, \\
 my beloved wife,\\
  who joined my journey \\
and forever changed my life. \\ \vspace{10pt}
With the thrill of facing the unknown,\\
you took away my fears\\
 and opened my heart\\
 to a whole new world. \\ \vspace{10pt}
You made me realize \\
that  at the end of the day,\\
 as long as we are together,\\
  we will always be at home. \\ \vspace{10pt}
For all of this,\\
and for so much more,\\
I dedicate to you\\
this work.
\end{dedication}

% \begin{mydedication} % You are responsible for formatting here.
%   \vspace{1in}
%   \begin{flushleft}
% 	To me.
%   \end{flushleft}
%   
%   \vspace{2in}
%   \begin{center}
% 	And you.
%   \end{center}
% 
%   \vspace{2in}
%   \begin{flushright}
% 	Which equals us.
%   \end{flushright}
% \end{mydedication}

%% EPIGRAPH
%  The same choices that applied to the dedication apply here.

\begin{epigraph} % The style file will position the text for you.
  \emph{One cannot escape the feeling \\ that these mathematical formulas\\ have an independent existence\\ and an intelligence of their own,\\
     that they are wiser than we are, \\
     wiser even than their discoverers...  }\\
  ---Heinrich Hertz
\end{epigraph}

% \begin{myepigraph} % You position the text yourself.
%   \vfil
%   \begin{center}
%     {\bf Think! It ain't illegal yet.}
% 
% 	\emph{---George Clinton}
%   \end{center}
% \end{myepigraph}

\tableofcontents
% \listoffigures  % Uncomment if you have any figures
% \listoftables   % Uncomment if you have any tables

%% ACKNOWLEDGEMENTS
%  While technically optional, you probably have someone to thank.
%  Also, a paragraph acknowledging all coauthors and publishers (if
%  you have any) is required in the acknowledgements page and as the
%  last paragraph of text at the end of each respective chapter. See
%  the OGS Formatting Manual for more information.

\begin{acknowledgements} 
I would like to thank the Conacyt-UCMexus fellowship that allowed me to come to UCSD in the first place, and provided me financial support during all this years. In this sense, I would also like to thank Jim Lin, William Helton, Nolan Wallach and Wee Teck Gan that, in one way or another, provided me the extra support that I needed to live in San Diego with my wife and child. 

I would like to thank my mentors at the University of Guanajuato, Luis Hern\'andez, Pedro Luis del Angel, Manuel Cruz and specially my undergraduate advisor Adolfo S\'anchez-Valenzuela, that encouraged me to come to San Diego to pursue my dream.   I would also like to thank Oded Yacobi, Orest Bucicovschi, Seung Lee, Jaime Lust, Neal Harris, Jon Middleton and Mandy Cheung for their helpful conversations during all this years. Finally, I would like to thank Akshay Venkatesh for pointing out a gap in my initial calculations.

Special thanks go to my advisors Nolan Wallach and Wee Teck Gan. To Wee Teck Gan for suggesting me the problem that will eventually become my thesis, and for his constant inspiration and help that opened my mathematical horizon. I will also like to thank him for the careful reading that he gave to this thesis and the uncountable improvements that came with that. Thanks to Nolan Wallach for his helpful weekly conversations, that showed me the way into this subject, and guided me through the hardest parts of this thesis. Without his help and direction I would have never been able to complete this work.

I wouldn't be able to be here without the constant encouragement and support from my parents Ra\'ul G\'omez and Gabriela Mu\~noz, and my sisters Ana Gabriela and Cintia Noem\'i. I would also like to thank my grandparents Luis Lauro Mu\~noz and Ra\'ul G\'omez for the inspiring lesson that their lives have been.

Finally, and most important, I would like to thank my wife Tanya Mart\'inez, for her incredible support and encouragement during all this years, and my son, Diego Nicol\'as, for showing me what unconditional love looks like.

Chapter 1 is a combination of the material in the papers \emph{Holomorphic continuation of Bessel integrals for general admissible induced representations: The case of compact stabilizer,} Selecta Mathematica, 2011, coauthored with Nolan R. Wallach. I was the secondary author of this paper and made substantial contributions to the research as did my co-author. and \emph{Holomorphic continuation of Bessel integrals for general admissible induced representations: The general case,}

The material in section 2.2 is essentially a restatement of the material found in the books Real Reductive Groups, volumes I and II, authored by Nolan R. Wallach
\end{acknowledgements}

%% VITA
%  A brief vita is required in a doctoral thesis. See the OGS
%  Formatting Manual for more information.
\begin{vitapage}
\begin{vita}
  \item[2006] B.~S. in Mathematics, Universidad de Guanajuato.
Thesis: ``Aplicaciones de los Grupos de Lie a las ecuaciones diferenciales de la F\'{\i}sica y la Geometr\'{\i}a.''
Advisor:Adolfo S\'anchez Valenzuela.

\item[2008] M.~A. in Pure Mathematics, University of California, San Diego
  
%\item[2002-2007] Graduate Teaching Assistant, University of California, San Diego

\item[2011] Ph.~D. in Mathematics, University of California, San Diego 

\end{vita}
\begin{publications}
\item R. Gomez and N.R. Wallach ``Holomorphic continuation of Bessel integrals for general admissible induced representations: The case of compact stabilizer'' to appear, Selecta Mathematica.

\item R. Gomez, J.W. Helton and I. Klep ``Determinant expansions of signed matrices and of certain jacobians'' SIAM journal on matrix analysis and applications 2010, vol. 31, no. 2, pp. 732--754.
\end{publications}
\end{vitapage}

%% Abstract
% There does not seem to be a maximum length. From the OGS Formatting
% Manual: ``The abstract may continue on to a second page.''

\begin{abstract}
Let $G$ be a simple Lie Group with finite center, and let $K\subset G$ be a maximal compact subgroup. We say that $G$ is a Lie group of tube type if $G/K$ is a hermitian symmetric space of tube type. For such a Lie group $G$, we can find a parabolic subgroup $P=MAN$, with given Langlands decomposition, such that $N$ is abelian, and $N$ admits a generic character with compact stabilizer.  We will call any parabolic subgroup $P$ satisfying this properties a Siegel parabolic.

Let $(\pi,V)$ be an admissible, smooth, Fr\'echet representation of a Lie group of tube type $G$, and let $P \subset G$ be a Siegel parabolic subgroup. If $\chi$ is a generic character of $N$, let $Wh_{\chi}(V)=\{\lambda:V \longrightarrow \mathbb{C} \, | \, \lambda(\pi(n)v)=\chi(n)v\}$ be the space of Bessel models of $V$.  After describing the classification of all the simple Lie groups of tube type, we will give a characterization of the space of Bessel models of an induced representation. As a corollary of this characterization we obtain a local multiplicity one theorem for the space of Bessel models of an irreducible representation of $G$.

As an application of this results we calculate the Bessel-Plancherel measure of a Lie group of tube type, $L^2(N\backslash G;\chi)$, where $\chi$ is a generic character of $N$. Then we use Howe's theory of dual pairs
 to show that the Plancherel measure of the space $L^2(O(p-r,q-s)\backslash O(p,q))$ is the pullback, under the $\Theta$ lift, of the Bessel-Plancherel measure $L^2(N\backslash
Sp(m,\mathbb{R});\chi)$, where $m=r+s$ and $\chi$ is a generic character that depends on $r$ and $s$.
\end{abstract}
\end{frontmatter}

%% DISSERTATION

% A common strategy here is to include files for each of the chapters. I.e.,
%   \include{chapter1.tex}
%   \include{chapter2.tex}
% etc.  Of course, if you prefer, you can just start with
%   \chapter{My First Chapter Name}
% and start typing away.  

\chapter*{Introduction} \addcontentsline{toc}{chapter}{Introduction}

In the classical theory of modular forms, there is a construction that associates to every cusp form $f$ on the upper half plane $\mathcal{H}$ an $L$-function
\[
 L(s,f)=\sum_{n >0} \frac{a_{n}}{n^{s}}.
\]
This $L$-function is related to other objects of interest in number theory, like elliptic curves over number fields, and its study is of critical importance in a wide range of applications. The $L$ function $L(s,f)$ can also be constructed using a representation theoretic point of view, by considering the space of Whittaker models of a discrete series representation of $SL(2,\mathbb{R})$, associated with the modular form $f$. If we also include Mass forms, then we can extend this construction to include all types of representations of $SL(2,\mathbb{R})$. This point of view has been incredibly successful and has given rise to an intricate and beautiful theory of $L$-functions associated to automorphic representations of $GL(n)$.

Unfortunately, this theory has not been as successful with other groups like $GSp(n)$. Part of the problem is that not all automorphic representations of $GSp(n)$ admit a Whittaker model. In \cite{s:1964} Siegel developed a technique, analogous to the theory of modular forms, to construct $L$-functions associated with holomorphic representations of $Sp(n,\mathbb{R})$ \cite{w:2003}. In this construction, the space of Whittaker models is replaced by the space of generalized Bessel models
\[
 Wh_{\chi}(V)=\{\lambda:V\longrightarrow \mathbb{C} \, | \, \mbox{$\lambda(\pi(n)v)=\chi(n)\lambda(v)$, for all $n\in N$}\},
\]
where $P=MAN$ is a Siegel parabolic subgroup of $Sp(n,\mathbb{R})$, with given Langlands decomposition, and $\chi$ is a \emph{generic} character of $N$, i.e., the $P$-orbit of $\chi$ on $\hat{N}$ is open. This construction has been adapted by Novodvorsky and Piatetski-Shapiro \cite{NPS:bessel} to construct $L$-functions associated to automorphic representations of $GSp(4)$. 

The study of the space of generalized Bessel models is the subject of the first chapter of this thesis. In order to describe the results obtained in that chapter we need to introduce a little bit of notation. Let $G$ be a simple Lie group with finite center, and let $K$ be a maximal compact subgroup. We say that $G$ is a \emph{Lie group of tube type} if $K\backslash G$ is a Hermitian symmetric space of tube type. In section \ref{sec:classification} we use the correspondence between Euclidean simple Jordan algebras over $\mathbb{R}$ and simple Hermitian symmetric spaces of tube type to describe a classification of the simple Lie groups of tube type. As a consequence of this classification we have the following proposition.
\begin{proposition}
 If $G$ is a Lie group of tube type, then 
\begin{enumerate}
 \item There exists a parabolic subgroup $P=MAN,$ with given Langlands decomposition, such that $N$ is abelian. 
 \item There exists a unitary character $\chi$ on $N,$ such that its stabilizer in $M,$
\[
 M_{\chi}=\{m\in M \, | \, \chi(m^{-1}nm)=\chi(n) \quad \forall n\in N\},
\]
is compact.
\end{enumerate}
\end{proposition}
If $P\subset G $ is a parabolic subgroup satisfying  1 and 2, then we say that $P$ is a Siegel parabolic subgroup.

Let $G$ be a Lie group of tube type, and let $P=MAN$ be a Siegel parabolic subgroup, with given Langlands decomposition. Let $(\sigma,V_{\sigma})$ be an admissible, smooth, Fr\'echet representation of $M$, and let $\nu\in \mathfrak{a}_{\mathbb{C}}'$ ($\mathfrak{a}=Lie(A)$). Define
\[
 I_{\sigma,\nu}^{\infty}=\left\{\phi:G \longrightarrow V_{\sigma}\, \left| \, \begin{array}{c} \mbox{$f$ is smooth and $f(\bar{n}amk)=a^{\nu-\rho}\sigma(m)f(k)$}\\ \mbox{for all $\bar{n}\in \bar{N}$, $a\in A$ and $m\in M$}\end{array} \right\}\right. .
\]
Here $\rho$ is half the sum of the roots associated to the $p$-pair $(P,A)$ \cite{w:vol1}, and $\bar{P}=MA\bar{N}$ is the parabolic opposite to $P$.
If we set $(\pi(g)f)(x)=f(xg)$, for all $f\in I_{\sigma,\nu}^{\infty}$, $x$, $g\in G$, then $(\pi,I_{\sigma,\nu}^{\infty})$ defines and admissible, smooth, Fr\'echet representation of $G$. Let
\[
Wh_{\chi}(I_{\sigma,\nu}^{\infty})=\{\lambda:I_{\sigma,\nu}^{\infty} \longrightarrow \mathbb{C} \, | \, \mbox{$\lambda(n \cdot f)=\chi(n)\lambda(f)$, for all $n\in N$}\}
\]
be the space of generalized Bessel models for a generic character $\chi\in \hat{N}$. This space has been subject of careful study during recent years. In the real case the more general results can be found in \cite{w:hol}, where a multiplicity one result is proved in the case where $P=MAN$ is a very nice parabolic subgroup \cite{w:deg}, and $(\sigma, V_{\sigma})$ is finite dimensional representation of $M$. In this context multiplicity one means that
\begin{equation}
\dim{Wh_\chi(I_{\sigma,\nu}^{\infty})}=\dim{V_{\sigma}}. \label{eq:dimension}
\end{equation}

In 2007 Dipendra Prasad asked if a similar result was true in the case where $(\sigma,V_{\sigma})$ is an admissible, smooth, Frechet, moderate growth representation of $M$. In this case the statement about dimensions in equation (\ref{eq:dimension}) has to be replaced by an $M_{\chi}$-intertwiner isomorphism between $V_{\sigma}'$ and $Wh_\chi(I_{\sigma,\nu}^{\infty})$, where
 \[
 M_{\chi}=\{m\in M \, |\, \mbox{$\chi(mnm^{-1})=\chi(n)$, for all $n \in N$}\}.
 \]
 
Let $I_{\sigma}$ be the representation smoothly induced from $K_{M}=K\cap M$ to $K$. Given $f\in I_{\sigma}^{\infty}$ define
\[
 f_{\nu}(namk)=a^{\nu-\rho}\sigma(m)f(k).
\]
The map $f\mapsto f_{\nu}$ defines a $K$-equivariant linear isomorphism from $I_{\sigma}^{\infty}$ to $I_{\sigma,\nu}^{\infty}$. Consider the integrals
\[
 J_{\sigma,\nu}^{\chi}(f)=\int_{N}\chi(n)^{-1}f_{\nu}(n)\, dn.
\]
These integrals are called generalized Jacquet integrals and converge absolutely and uniformly on compacta for $\operatorname{Re} \nu \ll 0$ \cite{w:hol}.
Let $\mu \in V_{\sigma}'$ and define $\gamma_{\mu}(\nu)=\mu\circ J_{\sigma,\nu}^{\chi}$. Observe that if $\operatorname{Re} \nu \ll 0$ then $\gamma_{\mu}$ defines a weakly holomorphic map into $(I_{\sigma}^{\infty})'$.
\begin{theorem}\label{thm:compact}
 Assume that $M_{\chi}$ is compact.
\begin{description}
 \item[i)] $\gamma_{\mu}$ extends to a weakly holomorphic map from $\mathfrak{a}_{\mathbb{C}}'$ to $(I_{\sigma}^{\infty})'$
\item[ii)] Given $\nu\in\mathfrak{a}_{\mathbb{C}}'$ define
\[
 \lambda_{\mu}(f_{\nu})=\gamma_{\mu}(\nu)(f), \qquad f\in I_{\sigma}^{\infty}.
\]
Then $\lambda_{\mu}\in Wh_{\chi}(I_{\sigma,\nu}^{\infty})$ and the map $\mu \mapsto \lambda_{\mu}$ defines an $M_{\chi}$-equivariant isomorphism between $V_{\sigma}'$ and $Wh_{\chi}(I_{\sigma,\nu}^{\infty})$.
 \end{description}
\end{theorem}

When $M_{\chi}$ is not compact, the above theorem as it is stated is false. This is mainly due to the fact that the orbits of the symmetric space $X:=M_{\chi}\backslash M$ under the action of a minimal parabolic subgroup of $M$ are much more complicated than in the case where $M_{\chi}$ is compact \cite{m:orbits}. However something can still be said about $Wh_{\chi}(I_{\sigma,\nu}^{\infty})$. Assume that the center of $M_{\chi}$ is compact, and let $(\tau,V_{\tau})$ be an irreducible, admissible, tempered, infinite dimensional representation of $M_{\chi}$. As in the case were $M_{\chi}$ is compact, define $\gamma_{\mu}(\nu)=\mu\circ J_{\sigma,\nu}^{\chi}$ and observe that if $\operatorname{Re} \nu \ll 0$ then $\gamma_{\mu}$ defines a weakly holomorphic map into $Hom(I_{\sigma}^{\infty},V_{\tau})$. Let
\[
 Wh_{\chi,\tau}(I_{\sigma,\nu}^{\infty})=\{\lambda:I_{\sigma,\nu}^{\infty}\longrightarrow V_{\tau} \, | \, \mbox{$\lambda(\pi(mn)f)=\chi(n)\tau(m)\lambda(f)$,  $\forall m\in M_{\chi}$, $n\in N$}\}.
\]
\begin{theorem}\label{thm:noncompact} With assumptions as above.
 \begin{description}
 \item[i)] $\gamma_{\mu}$ extends to a weakly holomorphic map from $\mathfrak{a}_{\mathbb{C}}'$ to $Hom(I_{\sigma}^{\infty},V_{\tau})$.
\item[ii)] Given $\nu\in\mathfrak{a}_{\mathbb{C}}'$ define
\[
 \lambda_{\mu}(f_{\nu})=\gamma_{\mu}(\nu)(f), \qquad f\in I_{\sigma}^{\infty}.
\]
Then $\lambda_{\mu}\in Wh_{\chi,\tau}(I_{\sigma,\nu}^{\infty})$ and the map $\mu \mapsto \lambda_{\mu}$ defines a linear isomorphism between $Hom_{M_{\chi}}(V_{\sigma},V_{\tau})$ and $Wh_{\chi,\tau}(I_{\sigma,\nu}^{\infty})$.
 \end{description}
\end{theorem}

%In \cite{JSZ:bess} Jian, Sun and Zhu prove a uniqueness result for the space of Bessel Models of a large class of groups. However their definition of the space of Bessel models is different than the one defined here, except for the special case of $SO(n,2)_{\circ}$. In \cite{JSZ:bess} the parabolic subgroups considered are very similar in their structure to the minimal parabolic subgroups. In the extreme cases their result reduces to the uniqueness of Whittaker models for minimal parabolic subgroups on one side, and to the multiplicity free formula developed by Sun and Zhu \cite{SZ:mult} and independently by Aizenbud, Gourevitch, et. al, \cite{AG:mult,AGRS:mult,AGS:gelf} on the other.

We will now describe an application of the results given so far. Let
\[
L^2(N\backslash G;\chi)=\left\lbrace f:G \longrightarrow \mathbb{C} \, \left| \,\begin{array}{c}
				\mbox{$f(ng)=\chi(n)f(g)$ and } \\
                                \mbox{$\int_{N\backslash G} |f(g)|^2 \, dNg < \infty$}
                              \end{array}  \right\rbrace\right..
\]
We will call this the space of generalized Bessel functions. Observe that there is a natural action of $M_{\chi}\times G$ on this space with $G$ acting on the right, and $M_{\chi}$ acting on the left. In chapter \ref{chapter:Bessel-Plancherel} we compute the ``Bessel-Plancherel'' measure, i.e., the spectral decomposition of the space of generalized Bessel functions with respect to this action. The calculations are based on the work of Wallach described in \cite{w:vol2} and depend on theorem \ref{thm:compact} and \ref{thm:noncompact}. The main result of that chapter is:
\begin{theorem}\label{thm:bessel-plancherelintroduction}
 Let $G$ be a Lie group of tube type, and let $P=MAN$ be a Siegel parabolic subgroup of $G$, with given Langlands decomposition. Let $\chi$ be a generic unitary character of $N$, and let $M_{\chi}$ be its stabilizer in $M$. Then the spectral decomposition of $L^2(N\backslash G;\chi)$, with respect to the action of $M_{\chi}\times G$, is given by
\[
 L^2(N\backslash G;\chi) \cong  \int_{\hat{G}}  \int_{\hat{M}_{\chi}} W_{\chi,\tau}(\pi)\otimes \tau^{\ast} \otimes\pi \,d\nu(\tau)\, d\mu(\pi),
\]
where $W_{\chi,\tau}$ is some multiplicity space, $\mu$ is the usual Plancherel measure of $G$, and $\nu$ is the Plancherel measure of $M_{\chi}$. Furthermore, if $M_{\chi}$ is compact, then $W_{\chi,\tau}(\pi)\cong Wh_{\chi,\tau}(\pi)$ is finite dimensional.
\end{theorem}
Given an irreducible unitary representation $(\pi, H_{\pi})$, let $(\pi^{\ast}, H_{\pi}^{\ast})$ be the associated contragradient representation. Then, in chapter \ref{chapter:Bessel-Plancherel}, we also have the following theorem
\begin{theorem}\label{thm:restrictiontoparabolicintroduction}
 Let $G$ be a Lie group of tube type, and let $P=MAN$ be a Siegel parabolic subgroup with given Langlands decomposition. Let $\Omega$ be the set of open $P$-orbits in $\hat{N}$. If $\chi$ is a generic character of $N$, let $\O_{\chi}$ be its associated $P$-orbit in $\hat{N}$. Then, for $\mu$-almost all irreducible tempered representations $(\pi, V_{\pi})$ of $G$,
\[
 \pi^{\ast}|_{P}\cong \bigoplus_{\O_{\chi}\in \Omega}\int_{M_{\chi}} W_{\chi,\tau}(\pi)\otimes \operatorname{Ind}_{M_{\chi}N}^{P}\tau^{\ast}\chi^{\ast}\, d\nu(\tau).
\]
Here the spaces $W_{\chi,\tau}(\pi)$ are the same as the ones appearing in the spectral decomposition of $L^2(N\backslash G;\chi)$.
\end{theorem}

Let $G$ be a reductive group and let  $X$ be a $G$-spherical variety. In \cite{yv:sph} Sakellaridis and Venkatesh give a conjecture describing the spectral decomposition of $L^2(X)$ in terms of the representation theory of another group. This conjecture generalizes the results of Harish-Chandra for $L^2(G)$ and of Delorme, Schlichtkrull and Van Den Ban for $L^2(X)$, where $X$ is a symmetric space \cite{VS:harmonic,D:formula,D:harmonic,D:schwartz}. More precisely the Sakellaridis-Venkatesh conjecture postulates the existence of a group $G_{X}$ and a correspondence
\[
\Theta:\hat{G}_{X} \longrightarrow \hat{G} ,
\]
between the unitary duals of $G_{X}$ and $G$, such that
\[
L^2(X) = \int_{\hat{G}_{X}} m(\pi)\otimes \Theta(\pi)\, d\mu(\pi),
\]
where $\mu$ is the Plancherel measure of $G_{X}$, and $m(\pi)$ is some multiplicity space whose dimension is finite, and typically $\leq 1$. If $X$ satisfies some technical hypothesis, then the group $G_{X}$ has the property that its dual group is $\check{G}_{X}$, the dual group associated to $X$ by Gaitsgory and Nadler \cite{gn:2010}, and hence the conjecture fits nicely into the setting of a proposed ``relative Langlands program'' \cite{y:func}.

Using the theory of dual pairs of the oscillator representation to construct the map $\Theta$, Howe \cite{h:someresult}, parameterized the spectral decompositions of the space $L^2(O(p-1,q)\backslash O(p,q))$ in terms of the unitary dual of $SL(2,\mathbb{R})$. Following the same ideas it's possible to obtain examples in the spirit of the Sakellaridis-Venkatesh conjecture, but that lie beyond the spherical variety case. For example, consider the dual pair $(Sp(m,\mathbb{R})\times O(p,q)) \subset Sp(mn,\mathbb{R})$, $p+q=n$, and assume that $p\geq q > m$. The last condition states that we are in the stable case. To simplify the exposition, assume also that $n$ is even (the $n$ odd case is very similar, but involves a double cover of $Sp(m,\mathbb{R})$). Let $P=MAN$ be the Siegel parabolic subgroup of $Sp(m,\mathbb{R})$, with given Langlands decomposition, and let $\chi_{r,s}$, $r+s=m$, be the character of $N$ given by
\[
 \chi_{r,s}\left(\left[\begin{array}{cc} I_m & X \\ & I_m \end{array}\right]\right) = \chi(\operatorname{tr} I_{r,s}X),
\]
where
\[
I_{r,s}=\left[\begin{array}{cc} I_r &  \\ & -I_s \end{array}\right]
\]
and $\chi$ is some fixed nontrivial unitary character of $\mathbb{R}$. Let
\[
L^2(N\backslash Sp(m,\mathbb{R});\chi_{r,s})=\left\lbrace f:Sp(m,\mathbb{R}) \longrightarrow \mathbb{C} \, \left| \,\begin{array}{c}
				\mbox{$f(ng)=\chi_{r,s}(n)f(g)$ and } \\
                                \mbox{$\int_{N\backslash Sp(m,\mathbb{R})} |f(g)|^2 \, dNg < \infty$}
                              \end{array}  \right\rbrace\right. .
\]
Observe that $MA\cong GL(m,\mathbb{R})$ in a natural way, and $M_{\chi_{r,s}}\cong O(r,s)$. In this setting theorem \ref{thm:bessel-plancherelintroduction} says that
\begin{equation}
L^2(N\backslash Sp(m,\mathbb{R});\chi_{r,s}) \cong  \int_{Sp(m,\mathbb{R})^{\wedge}}  \int_{O(r,s)^{\wedge}} W_{\chi_{r,s},\tau}(\pi)\otimes \tau^{\ast} \otimes\pi \, d\eta(\tau) \, d\mu(\pi), \label{eq:L2whittaker}
\end{equation}
where $\eta$ is the Plancherel measure of $O(r,s)$ and $\mu$ is the Plancherel measure of $Sp(m,\mathbb{R})$. Moreover, by theorem \ref{thm:restrictiontoparabolicintroduction}, we have that for $\mu$-almost all tempered representation $\pi$ of $Sp(m,\mathbb{R})$
\begin{equation}
\pi^{\ast}|_{P} \cong \bigoplus_{r+s=m} \int_{O(r,s)^{\wedge}} W_{\chi_{r,s},\tau}(\pi)\otimes \operatorname{Ind}_{O(r,s)  N}^{P} \tau^{\ast} {\chi_{r,s}^{\ast}} \, d\eta(\tau). \label{eq:restriction}
\end{equation}
On the other hand, Howe showed that in the stable range the Oscillator representation $(\xi,L^2(\mathbb{R}^{mn}))$ of $Sp(mn,\mathbb{R})$ decomposes in the following way when restricted to $Sp(m,\mathbb{R})\times O(p,q)$:
\begin{equation}
L^2(\mathbb{R}^{mn})\cong \int_{Sp(m,\mathbb{R})^{\wedge}} \pi\otimes \Theta(\pi) \, d\mu(\pi). \label{eq:howeduality}
\end{equation}
 where $\mu$ is the Plancherel measure of $Sp(m,\mathbb{R})$ and $\Theta(\pi)$ is a representation of $O(p,q)$ called the $\Theta$-lift of $\pi$. A lot of work has been done to describe the explicit $\Theta$-correspondence, and in the stable range this correspondence can be described using the work of Jian-Shu Li \cite{jsl:stiefel} among others. Using equations (\ref{eq:L2whittaker}), (\ref{eq:restriction}), (\ref{eq:howeduality}) and the explicit formulas for the action of $(Sp(m,\mathbb{R})\times O(p,q))$ on $L^2(\mathbb{R}^{mn})$ given in \cite{A:theta,R:weil,R:witt,R:explicit} we obtain the following description of $L^2(O(p-r,q-s)\backslash O(p,q))$ 
 \begin{theorem}
 As an $O(r,s)\times O(p,q)$-module with  $O(r,s)$ acting on the left, and $O(p,q)$ acting on the right
 \[
L^2(O(p-r,q-s)\backslash O(p,q))\cong \int_{Sp(m,\mathbb{R})^{\wedge}}  \int_{O(r,s)^{\wedge}} W_{\chi_{r,s},\tau}(\pi)\otimes  \tau^{\ast} \otimes \Theta(\pi^{\ast}) \, d\eta(\tau) \, d\mu(\pi)
\]
where $\eta$ and $\mu$ are the Plancherel measures of $O(r,s)$ and $Sp(m,\mathbb{R})$ respectively, and $W_{\chi_{r,s},\tau}(\pi)$ are the multiplicity spaces appearing in equation (\ref{eq:restriction}).
\end{theorem}
Observe that when $m=1$ we regain Howe's result \cite{h:someresult}. Also observe that in this case the decomposition given in equation (\ref{eq:L2whittaker}) is contained in Wallach's work on the Plancherel-Whittaker measure for minimal parabolic subgroups \cite{w:vol2}.

%Chapter Bessel Models

\chapter{Bessel models for representations of Lie groups of tube type}\label{chapter:besselmodels}

\section{Siegel modular forms on the upper half plane} \label{sec:siegelmodularforms}

Let $\mathcal{H}=\{z=x+iy \in \mathbb{C}\, |\, y>0\}$ denote the complex upper half plane. For each integer $k > 0$, we will consider the space of holomorphic functions $f:\mathcal{H}\longrightarrow \mathbb{C}$ such that
\begin{equation}
 f\bigg(\frac{az+b}{cz+d}\bigg)=(cz+d)^{k}f(z), \label{eq:modularcondition}
\end{equation}
for all integers $a$, $b$, $c$, $d$, such that $ad-bc=1$. Observe that this condition implies that $f(z+1)=f(z)$, for all $z\in \mathcal{H}$, and hence we have a Fourier series expansion
\begin{equation}
 f(z)=\sum_{n\in \mathbb{Z}} a_{n}q^{n} \label{eq:modularfourierexpansion}
\end{equation}
where $q=e^{2\pi iz}$. We say that $f$ is a \emph{modular form of weight $k$}, if $a_{n}=0$ for all $n< 0$. If, in addition to this conditions, we have that $a_{0}=0$, then we say that $f$ is a \emph{cusp form}.

Given a cusp form $f$ of weight $k$, we can define a Dirichlet series
\[
 L(s,f)=\sum_{n >0} \frac{a_{n}}{n^{s}},
\]
where the $a_{n}$ are the Fourier coefficients appearing in the expansion of $f$ (\ref{eq:modularfourierexpansion}). Observe that this Dirichlet series defines a holomorphic function for $\operatorname{Re} s \gg 0$. Moreover, it can be shown that $L(s,f)$ has meromorphic continuation to all of $\mathbb{C}$, and that it satisfies a functional equation. The $L$-functions constructed this way are related to other objects of interest in number theory, like elliptic curves over the rational numbers via the modularity theorem, and its study is of central importance in a wide range of applications.

The numbers $a_{n}$, appearing in this construction, have a
beautiful representation theoretic interpretation, that we will now describe. Given a holomorphic function $f$ on $\mathcal{H}$, and an element $g\in SL(2,\mathbb{R})$, define
\[
 (f|_{k}g)(z)=(cz+d)^{-k} f\bigg(\frac{az+b}{cz+d}\bigg),  \qquad g=\left[\begin{array}{cc} a & b \\ c & d \end{array}\right].
\]
This equation defines a right action of $SL(2,\mathbb{R})$ on the space of holomorphic functions on $\mathcal{H}$. Given a modular form $f$ of weight $k$, define a function $\phi$ on $SL(2,\mathbb{R})$ by
\[
 \phi(g)=(f|_{k}g)(i).
\]
It is then clear, using equation (\ref{eq:modularcondition}), that $\phi \in C^{\infty}(\Gamma\backslash SL(2,\mathbb{R}))$, where $\Gamma=SL(2,\mathbb{Z})$. Furthermore it can be shown that if
\[
k(\theta)=\left[
\begin{array}
[c]{cc}%
\cos\theta & \sin\theta\\
-\sin\theta & \cos\theta
\end{array}
\right]  \in SO(2),
\]
then%
\[
\phi(gk(\theta))=e^{ik\theta} \phi(g).
\]
Let $U(\mathfrak{g})$ be the universal enveloping algebra of $\mathfrak{g}=Lie(SL(2,\mathbb{R}))$, and set $V_{K}=U(\mathfrak{g})\phi$. Then, it can be shown that $V_{K}$ is isomorphic to the space of $SO(2,\mathbb{R})$-finite vectors of a discrete series representation $(\pi,H)$ of $SL(2,\mathbb{R})$, with lowest weight $k$. Let $V=H^{\infty}$ be the space of smooth vectors of $H$. Then for each $n\in \mathbb{Z}$ there is a natural choice of a linear functional
$\lambda_{n}:V \longrightarrow \mathbb{C}$ such that 
$$
\lambda_{n}\left(\pi\left(\left[\begin{array}{cc} 1 & x \\ & 1 \end{array}\right]\right)v \right)=e^{2\pi inx}\lambda(v).
$$
This family of linear functionals has the property that $\lambda_{n}(\phi)=a_{n}$ for all $n\in \mathbb{Z}$.

The theory of Maass forms extends this theory to include all
types of unitary representations of $SL(2,\mathbb{R})$. The analogous functionals $\lambda_{n}$ are called, following Jacquet,  Whittaker
functionals corresponding to the unitary character $\chi_{n}$ of $N$ given by
\[
\chi_{n}\left(  \left[
\begin{array}
[c]{cc}
1 & x\\
0 & 1
\end{array}
\right]  \right)  =e^{2\pi inx}.
\]

We will now move to the Siegel upper half plane,
$\mathcal{H}_{m}$, consisting of elements $Z=X+iY$ with $X$ and $Y$ symmetric $m\times m$ matrices over $\mathbb{R}$ and $Y$ positive definite. Let $G=Sp(m,\mathbb{R})$ realized as the set of $2m\times 2m$ matrices with block form
\[
g=\left[
\begin{array}
[c]{cc}
A & B\\
C & D
\end{array}
\right],
\]
with $A,B,C,D$, $m\times m$ matrices such that if
\[
J=\left[
\begin{array}
[c]{cc}%
0 & I\\
-I & 0
\end{array}
\right]  ,
\]
with $I$ the $m\times m$ identity matrix, then
\[
gJg^{T}=J.
\]
Observe that we can define an action of $G$ on $\mathcal{H}_{m}$ by ``linear fractional transformations''
\[
g\cdot Z=(AZ+B)(CZ+D)^{-1}.
\]
In this case C.L.Siegel \cite{s:1964} considered subgroups $\Gamma$ of finite index in
$G_{\mathbb{Z}}=Sp(m,\mathbb{Z})$, and holomorphic functions $f$ on $\mathcal{H}_{m}$, such that (with $g$ in
block form as above)
\[
f(gZ)=\det\left(  CZ+D\right)  ^{k}f(Z),
\]
for $g\in\Gamma$ and a growth condition at $\infty$.  As above one has the subgroup $N$ consisting of the elements of the form%
\[
n(L)=\left[
\begin{array}
[c]{cc}%
I & L\\
0 & I
\end{array}
\right]
\]
with $L$ an $m\times m$ symmetric matrix over $
\mathbb{R}$. Observe that $\Gamma\cap N$ contains a subgroup of finite index in $N_{\mathbb{Z}}=G_{\mathbb{Z}
}\cap N$. We will assume, just as we did in the classical case, that $\Gamma\cap N$ is actually equal to $N_{
\mathbb{Z}
}$. Then,
\[
f(Z+L)=f(Z)
\]
for $L$ an $m\times m$ symmetric matrix with entries in $\mathbb{Z}$. 
We can thus expand this Siegel modular form in a Fourier series.%
\[
\sum a_{S}e^{2\pi iTr(SZ)},
\]
where the sum runs over the $m\times m$ symmetric matrices $S$ over $\mathbb{Z}$. 
One finds that if $a_{S}\neq0$ then $S$ must be positive semi-definite.

Once again this coefficients have a beautiful representation theoretic
interpretation. We can
consider
\[
\chi_{S}(n(X))=e^{2\pi iTr(SX)},
\]
with $S$ a symmetric $m\times m$ matrix over $\mathbb{R}$. Let $M$ be the image of $GL(m,\mathbb{R})$ in $G$ via the embedding
\[
g\longmapsto\left[
\begin{array}
[c]{cc}
g & 0\\
0 & g^{T}
\end{array}
\right].
\]
Then $P=MN$ is the Siegel
parabolic subgroup of $G$, and we have an action of $M$ on $N$ (and hence in its space of characters $\hat{N}$) by conjugation. One finds that if a character is \emph{generic} (that
is, the $M$-orbit in the character group is open) then the character must be
given by $\chi_{S}$ with $\det S\neq0$. The stabilizer of the character is
compact if, only if, $S$ is positive or negative definite. If $(\pi, V_{\pi})$ is an irreducible representation of $Sp(m,\mathbb{R})$, we will set
\[
 Wh_{\chi_{S}}(V_{\pi})=\{\lambda: V_{\pi} \longrightarrow \mathbb{C} \, | \, \lambda(n(L)v)=\chi_{S}(n(L))\lambda(v)\}.
\]
If the stabilizer of the character $\chi_{S}$ is compact, we will call any element $\lambda \in Wh_{\chi_{S}}(V_{\pi})$ a \emph{Bessel model}. If $\chi_{S}$ is just a generic character, then we will call any $\lambda \in Wh_{\chi_{S}}(V_{\pi})$ a \emph{generalized Bessel model}. One can show (c.f. W[5])
that the only such models for holomorphic (resp. antiholomorphic)
representations are those corresponding to positive definite (resp.
negative definite) such $S$. Thus the only generic characters that can appear
if we consider holomorphic or anti-holomorphic Siegel modular forms are the
ones with compact stabilizer, and the relevant space here is the space of Bessel models of holomorphic or antiholomorphic representations of $G$. As in the case of $SL(2,\mathbb{R})$
 ($m=1$) this theory can be extended to include other types of representation of $Sp(m,\mathbb{R})$. In this case we need to broaden or scope to include the space of generalized Bessel models for a generic character $\chi_{S}$.

\section{Classification of Lie groups of tube type and its generalized Bessel characters}\label{sec:classification}

Let $G$ be a connected simple Lie group
with finite center and let $K$ be a maximal compact subgroup. We assume that
$G/K$ is Hermitian symmetric of tube type. This can be interpreted as follows.
There exists a group homomorphism, $\phi$, of a finite covering, $S$ of $PSL(2,\mathbb{R})$
into $G$ such that if $H=\phi(S)$
then $H\cap K$ is the center of $K$. We take a standard basis $h,e,f$ of
$Lie(H)$ over
$\mathbb{R}$
 with the standard TDS (three dimensional simple) commutation relations
($[e,f]=h,[h,e]=2e,[h,f]=-2f$). If $\mathfrak{g}=Lie(G)$, then we have
$\mathfrak{g}=\mathfrak{\bar{n}\oplus m \oplus a \oplus n}$ with $\mathfrak{\bar{n}%
,m\oplus a,n}$ respectively the $-2,0,2$ eigenspace of $ad(h)$, $\mathfrak{a}=\mathbb{R} h$, and $\mathfrak{m}$ the orthogonal complement of $\mathfrak{a}$ in $\mathfrak{m\oplus a}$ with respect to the Cartan-Killing form. In particular,
$\mathfrak{\bar{n}}$ and $\mathfrak{n}$ are commutative and $e\in\mathfrak{n}%
$, $f\in\mathfrak{\bar{n}}$. Let $\theta$ be the Cartan involution of $G$
corresponding to the choice of $K$, then we may assume, $\theta\mathfrak{n=\bar
{n}}$, $f=-\theta e$ and
$\mathbb{R}(e-f)=Lie(H\cap K)$. Set $\mathfrak{p=}$ $\mathfrak{m\oplus a \oplus n}$ and let
$P=\{g\in G|Ad(g)\mathfrak{p}=\mathfrak{p\}}$. Then $P$ is a parabolic
subgroup of $G$, and if we take its Langlands decomposition $P=MAN$, then $\mathfrak{m}=Lie(M)$, $\mathfrak{a}=Lie(A)$ and $\mathfrak{n}=Lie(N)$. Let $\chi$ be a generic character of $N$, and let
\[
M_{\chi}=\{m\in M \, | \, \chi\circ Ad(m)=\chi\}.
\]
In this section we will describe representatives for all the equivalence classes of generic characters on $N$. For the rest of the section we will fix a unitary character $\chi_{\mathbb{R}}$ of $\mathbb{R}$. This is the list of examples.

1. $G=Sp(n,\mathbb{R})$ realized as $2n\times2n$ matrices such that $gJ_{n}g^{T}=J_{n}$ with%
\[
J_{n}=\left[
\begin{array}
[c]{cc}%
0 & I_{n}\\
-I_{n} & 0
\end{array}
\right]
\]
with $I_{n}$ the $n\times n$ identity matrix (upper $T$ means transpose).
$\theta(g)=(g^{-1})^{T}$. With this description
\[
MA=\left\{  \left[
\begin{array}
[c]{cc}%
g & 0\\
0 & (g^{-1})^{T}%
\end{array}
\right]  |g\in GL(n,%
%TCIMACRO{\U{211d} }%
%BeginExpansion
\mathbb{R}
%EndExpansion
)\right\}  ,
\]%
\[
N=\left\{  \left[
\begin{array}
[c]{cc}%
I & X\\
0 & I
\end{array}
\right]  |X\in M(n,\mathbb{R}),X^{T}=X\right\}  .
\]
The list of generic characters is described as follows: let $p$, $q$ be two positive integers such that $p+q=n$. Define a character $\chi_{p,q}$ as follows
\[
\chi_{p,q}\left(\left[
\begin{array}
[c]{cc}%
I & X\\
0 & I
\end{array}
\right]\right)=\chi_{\mathbb{R}}(\operatorname{tr}{I_{p,q}X}).
\]
From this definitions it is clear that
\[
M_{\chi_{p,q}} = O(p,q)
\]
using the natural identification $M=GL(n,\mathbb{R})$.

2. $G=SU(n,n)$ realized as the $2n\times2n$ complex matrices, $g$, such that
$gL_{n}g^{\ast}=L_{n}$ with%
\[
L_{n}=\left[
\begin{array}
[c]{cc}%
0 & iI_{n}\\
-iI_{n} & 0
\end{array}
\right]  .
\]
In this case the centralizer, $MA$, of $h$ in $G$ is the set of all%
\[
\left[
\begin{array}
[c]{cc}%
g & 0\\
0 & (g^{\ast})^{-1}%
\end{array}
\right]
\]
with $g\in GL(n,\mathbb{C})$, and
\[
N=\left\{  \left[
\begin{array}
[c]{cc}%
I & X\\
0 & I
\end{array}
\right]  |X\in M(n,\mathbb{R}),X^{\ast}=X\right\}  .
\]
Representatives of equivalence classes of characters are again parameterized by positive integers $p$, $q$ such that $p+q=n$ and we can define
\[
\chi_{p,q}\left(\left[
\begin{array}
[c]{cc}%
I & X\\
0 & I
\end{array}
\right]\right)=\chi_{\mathbb{R}}(\operatorname{tr}{I_{p,q}X}).
\]
In this case
\[
M_{\chi_{p,q}}=U(p,q).
\]

3. $G=SO^{\ast}(4n)$ realized as the group of all $g\in SO(4n,\mathbb{C})$ such that $gJ_{2n}g^{\ast}=J_{2n}$. We can describe $\mathfrak{g}=Lie(G)$ as a Lie subalgebra of
$M_{2n}(\mathbb{H})$ as the matrices in block form%
\[
\left[
\begin{array}
[c]{cc}%
A & X\\
Y & -A^{\ast}%
\end{array}
\right]
\]
with $A,X,Y\in M_{n}(\mathbb{H})$ and $X^{\ast}=X,Y^{\ast}=Y$. In this form
$\mathfrak{g}\cap M_{2n}(\mathbb{R})=Lie(Sp(n,\mathbb{R}))$. \ We take $e,f,h$ as above and note that $MA \cong GL(n,\mathbb{H})$. If we define $\chi_{p,q}$ as before it is then easy to check that
\[
M_{\chi_{p,q}}=Sp(p,q).
\]

4. $G$ the Hermitian symmetric real form of $E_{7}$. In this case we will
emphasize a decomposition of $Lie(G)$ which makes it look exactly like those
examples 1.,2., and 3.. In each of those cases we have
\[
Lie(G)=\left[
\begin{array}
[c]{cc}%
A & X\\
Y & -A^{\ast}%
\end{array}
\right]
\]
with $A$ an element of $M_{n}(F)$ and $F=\mathbb{R},\mathbb{C}$ or $\mathbb{H}$ the upper * is the conjugate (of the field) transposed.
Furthermore, $X,Y$ are elements of $M_{n}(F)$ that are self adjoint. Example 4
corresponds to the octonions, $\mathbb{O}$. Here we replace $M_{3}%
(\mathbb{O)}$ by $\mathfrak{m}\oplus \mathfrak{a}=\mathbb{R}\oplus E_{6,2}$ (the real form of real rank 2 with maximal compact of type
$F_{4}$). We take for $X,Y$ elements of the exceptional Euclidean Jordan
algebra (the $3\times3$ conjugate adjoint matrices over $\mathbb{O}$ with
multiplication $A\cdot B=\frac{1}{2}(AB+BA)$ thus in this case the $X$'s and
$Y$'s are defined in the same way for the octonions as for the other fields).
Here $\mathfrak{m}$ acts by operators that are a sum of Jordan multiplication
and a derivation of the Jordan algebra (the derivations defining the Lie
algebra of compact $F_{4}$). With this notation our choice of $e,f,h$ are
exactly the same as the examples for $\mathbb{R},\mathbb{C}$ or $\mathbb{H}$. In this case we can define the characters
\[
\chi_{3,0}(\left(\left[
\begin{array}
[c]{cc}%
I & X\\
0 & I
\end{array}
\right]\right)=\chi_{\mathbb{R}}(\operatorname{tr}{X}).
\]
in which case $M_{\chi_{3,0}}$ is isomorphic to compact $F_4$, or the character
\[
\chi_{2,1}(\left(\left[
\begin{array}
[c]{cc}%
I & X\\
0 & I
\end{array}
\right]\right)=\chi_{\mathbb{R}}(\operatorname{tr}{I_{2,1}X}).
\]
in which case $M_{\chi_{2,1}}$ is isomorphic to $F_{4,1}$, the real form of $F_4$ of real rank 1. The stabilizer of the characters $\chi_{1,2}$ and $\chi_{0,3}$ are the same as the stabilizers for $\chi_{2,1}$ and $\chi_{3,0}$ respectively.

There is one more example (that doesn't fit this beautiful picture).

5. $G=SO(n,2)$ realized as the group of $n+2$ by
$n+2$ matrices of determinant 1 that leave invariant the form
\[
\left[
\begin{array}
[c]{ccc}
0 & 0 & 1 \\
0 & I_{n-1,1} & 0 \\
1 & 0 & 0
\end{array}
\right].
\]
Here
\[
 MA=\left\{\left[\begin{array}{ccc} a & 0 & 0 \\ 0 & m & 0 \\ 0 & 0 & a^{-1} \end{array}\right] \, | \, \mbox{$a\in \mathbb{R}^{\ast},$ $m\in SO(n-1,1)$} \right\}
\]
and
\[
N=\left\{\left[\begin{array}{ccc} 1 & -v^{t} & -\frac{\langle v,v \rangle}{2} \\ 0 & I & v \\ 0 & 0 & 1 \end{array}\right] \, | \, \mbox{$v\in \mathbb{R}^{n-1,1}$}\right\}.
\]
Representatives of the orbits of generic characters are given by
\[
 \chi_{    k}\left(\left[\begin{array}{ccc} 1 & -v^{t} & \frac{\langle v,v \rangle}{2} \\ 0 & I & v \\ 0 & 0 & 1 \end{array}\right]\right)=\chi_{\mathbb{R}}(    v_{k}),
\]
where $v_{k}$ is the $k$-th component of $v$. Observe that
\[
M_{\chi_{n}}\cong SO(n-1,\mathbb{R}).
\]
and
\[
M_{\chi_{k}}\cong SO(n-2,1), \qquad\mbox{if $k\neq n$.}
\]

\section{Jacquet integrals and Bessel models}
Let $G$ be one of the simple Lie groups of tube type we just described, and let $P=MAN,$ $\chi$ and $M_{\chi}$ be as before. Let $P_{\circ}=M_{\circ}A_{\circ}N_{\circ}$ be a minimal parabolic sugroup such that
\[
 P_{\circ}\subset P, \qquad N\subset N_{\circ}, \qquad A\subset A_{\circ}, \qquad M_{\circ}\subset M.
\]
Let $\Phi^{+}$ be the system of positive roots of $G$ relative to $P_{\circ},$ and let $\Phi_{M}^{+}$ be the system of positive roots of $MA$ induced by $\Phi^{+}.$ Let $W=W(G,A_{\circ}),$ $W_{M}=W(MA,A_{\circ})$ and set
\[
W^{M}=\{w\in W\, | \, w\Phi_{M}^{+}\subset \Phi^{+}\}.
\] 
Then we have the following classical result.
\begin{lemma}[Bruhat decomposition] 
With notation and assumptions as above.
\begin{enumerate}
\item If $v\in W$, then $v$ can be expressed in a unique way as a product of an element in $W_{M}$ and an element in $W^{M}$.
\item Given $v\in W,$ fix $w_{v}\in N_{K}(A_{\circ})$ such that $M_{\circ}w_{v}=v$. Then
 \[
  G=\bigcup_{v\in W^{M}} P_{\circ}w_{v}P.
 \]
 
\item Let $v_{G}$ be the longest element of $W,$ $v_{M}$ the longest element of $W_{M},$ and set $v^{M}=v_{G}v_{M}$. If we set $w_{G}=w_{v_{G}}$, $w_{M}=w_{v_{M}}$ and $w^{M}=w_{v^{m}}$,  then $$P_{\circ}w^{M}P=Pw^{M}N$$ and if $v\neq v^{M}$ then
\[
\dim P_{\circ}w_{v}P < \dim Pw^{M}N.
\]
\end{enumerate}
\end{lemma}

\begin{corollary}\label{lemma:bruhat}
Assume that the character $\chi$ has compact stabilizer. Then
\[
G=\bigcup_{v\in W^{M}}P_{\circ}w_{v}M_{\chi}N.
\]
Furthermore, if $v\in W^{M}$ and $v\neq v^{M}$,
then
\[
\dim P_{\circ}w_{v}M_{\chi}N<\dim P_{\circ}w^{M}M_{\chi}N.
\]
\end{corollary}

\begin{proof}
From the classification of simple Lie groups of tube type, we see that if the stabilizer of $\chi$, $M_{\chi}$, is compact, then there is a maximal compact subgroup $K$ such that, if we set $K_{M}=M\cap K$, then $M_{\chi}=K_{M}$ is a maximal compact subgroup of $M$.
We note that $W^{M}=\{v\in W|v\cdot \Phi_{M}^{+}\subset\Phi^{+}\}$, hence
$w_{v}(P_{\circ}\cap M)w_{v}^{-1}\subset P_{\circ}$. The Iwasawa decomposition
implies that%
\[
M=(P_{\circ}\cap M)K_{M}.
\]
Since $w_{v}A_{\circ}w_{v}^{-1}\subset A_{\circ}$ for all $v\in W$ we see that%
\begin{eqnarray*}
G & = &\bigcup_{v\in W^{M}} P_{\circ} w_{v} (P_{\circ}\cap M)K_{M}N =\bigcup_{v\in W^{M}%
}P_{\circ}w_{v}(P_{\circ}\cap M)w_{v}^{-1}w_{v}K_{M}N \\
&=& \bigcup_{v\in W^{M}}P_{\circ}w_{v}K_{M}N.
\end{eqnarray*}
If we now use that $M_{\chi}=K_{M}$ we obtain the decomposition we wanted. The dimension assertion follows from the fact that $v^{M}=v_{G}v_{M}$ is the
unique element of $W^{M}$ such that $w_{v}N w_{v}^{-1}\cap N_{\circ}=\{1\}$.
\end{proof}

\begin{lemma}\label{lemma:compactvanishing}
Assume that the character $\chi$ has compact stabilizer. If $v\in W^{M}$ is not $v^{M},$ then the restriction of $\chi$ to $w_{v}^{-1}N_{\circ}w_{v} \cap N $ is non-trivial.
\end{lemma}

\begin{proof}
The tube type assumption implies that $\Phi$ is
a root system of type $C_{n}$ with $n=\dim A_{\circ}$. Hence, there
exist linear functionals $\varepsilon_{1},...,\varepsilon_{n}$ on
$\mathfrak{a}_{\circ}=Lie(A_{\circ})$ such that
\[
\Phi^{+}=\{\varepsilon_{i}\pm\varepsilon_{j}|1\leq i<j\leq n\}\cup
\{2\varepsilon_{1},...,2\varepsilon_{n}\}
\]
and
\[
\Phi_{M}^{+}=\{\varepsilon_{i}-\varepsilon_{j}|1\leq i<j\leq n\}.
\]

Let $X \in Lie(N)$ be such that $[H,X]=2\varepsilon_{i}(H)X,$ for all $H \in Lie(A_{\circ})$. For such an $X$ it can be checked that $d\chi(X)\neq 0$. Hence, if $v\in
W^{M}$ and $\chi$ restricted to $w_{v}^{-1}N_{\circ}w_{v} \cap N $ is trivial, we must have
\[
v^{-1}\cdot(2\varepsilon_{i}) \in -\Phi^{+},\qquad i=1,...,n.
\]
Therefore
$v^{-1}\cdot(\varepsilon_{i}+\varepsilon_{j})\in -\Phi^{+}$ for all $i\leq j,$ which
implies that $v=v^{M}$.
\end{proof}

Given an admissible, smooth, Fr\'echet representation, $(\pi,V_{\pi})$, of $G$, and a generic character, $\chi$, of $N$, define the space of Bessel models of $V_{\pi}$ to be
\[
 Wh_{\chi}(V_{\pi})=\{\lambda \in V_{\pi}' \, | \, \mbox{$\lambda(\pi(n)v)=\chi(n)v$ for all $n\in N$ }\}.
\]
Here $V_{\pi}'$ is the continuous dual of $V_{\pi}$. Observe that we can define an action of $M_{\chi}$ on $Wh_{\chi}(V)$ by setting $(m\cdot \lambda)(v)=\lambda(\pi(m)^{-1}v)$, for $m\in M_{\chi}$, $\lambda \in Wh_{\chi}(V)$. Effectively, if $n\in N$,
\begin{eqnarray*}
 (m\cdot \lambda)(\pi(n)v) & = & \lambda(\pi(m)^{-1}\pi(n)v)=\lambda(\pi(m^{-1}nm)\pi(m)^{-1}v)\\
 & = & \chi(m^{-1}nm)\lambda(\pi(m)^{-1}v) = \chi(n) (m\cdot \lambda)(v),
\end{eqnarray*}
 where the last equality follows from the fact that $\chi\circ Ad(m)=\chi$, for all $m\in M_{\chi}$. If $(\tau, V_{\tau})$ is an admissible, smooth, Fr\'echet representation of $M_{\chi}$ we will set 
\[
 Wh_{\chi,\tau}(V_{\pi})=\{\lambda \in \mbox{Hom}(V_{\pi},V_{\tau}) \, | \, \lambda(\pi(mn)v)=\chi(n)\tau(m)v\},
\]
where $\mbox{Hom}(V_{\pi},V_{\tau})$ is the set of continuous linear maps between $V_{\pi}$ and $V_{\tau}$. Observe that if $(\tau,V_{\tau})$ is irreducible, then there is a natural $M_{\chi}$-equivariant embedding, $V_{\tau}'\otimes Wh_{\chi,\tau}(V_{\pi})\longrightarrow Wh_{\chi}(V_{\pi})$, given by $(\mu\otimes \lambda)(v)=\mu(\lambda(v))$ for $\mu \in V_{\tau}'$, $\lambda \in Wh_{\chi,\tau}(V_{\pi})$, and $v\in V_{\pi}$.

Let $(\sigma, V_{\sigma})$ be an admissible, smooth, Fr\'echet representation of $M$ of moderate growth. Let $\mathfrak{a}=Lie(A)$, and let
$\mathfrak{a}_{\mathbb{C}}'$ be the set of complex valued linear functionals on $\mathfrak{a}$. Let $\rho$ be half the sum of the positive roots of $P$ relative to $A$. Given an element $\nu \in \mathfrak{a}_{\mathbb{C}}'$, set 
\[
 \sigma_{\nu}(nam)=a^{\nu+\rho}\sigma(m). \qquad \mbox{for all $n\in N$, $a\in A$, $m\in M$.}
\]
Let 
\[
 I_{\sigma,\nu}^{\infty}=\{\phi:G\longrightarrow V_{\sigma} \, | \, \mbox{$\phi$ is smooth, and $\phi(namk)=a^{\nu+\rho}\phi(k)$} \},
\]
and define an action of $G$ on $I_{\sigma,\nu}^{\infty}$ by setting $\pi_{\sigma,\nu}(g)\phi(x)=\phi(xg)$ for $x,g\in G$. If we give to $I_{\sigma,\nu}^{\infty}$ the usual $C^{\infty}$ topology, then $(\pi_{\sigma,\nu},I_{\sigma,\nu}^{\infty})$ is an admissible, smooth, Fr\'echet representation of $G$ of moderate growth. Let
$$
I_{\sigma}^{\infty}=\{\phi:K \longrightarrow V_{\sigma} \, | \, \mbox{$\phi$ is smooth and $\phi(mk)=\sigma(m)\phi(k)$, for all $m\in K_{M}$}\}
$$
Observe that this space has a natural $K$-action.
Given $\phi \in I_{\sigma}^{\infty}$, define $\phi_{\sigma,\nu}\in I_{\sigma,\nu}^{\infty}$ by
\[
 \phi_{\sigma,\nu}(namk)=a^{\nu+\rho}\sigma(m)\phi(k).
\]
Observe that the map $\phi \mapsto \phi_{\sigma,\nu}$ defines a $K$-equivariant isomorphism between $I_{\sigma}^{\infty}$ and $I_{\sigma,\nu}^{\infty}$ for all $\nu \in \mathfrak{a}_{\mathbb{C}}'$.

Given an element $\phi\in I^{\infty}_{\sigma}$, and a generic character $\chi$ of $N$ we define its generalized Jacquet integral to be the integral
\[
 J_{\sigma,\nu}^{\chi}(\phi)=\int_{N} \chi(n)^{-1}\phi_{\sigma,\nu}(w^{M}n)\, dn.
\]
Observe that if $G$ is a simple Lie group of tube type, then $\dim \mathfrak{a}_{\mathbb{C}}'=1$, and we can use $\rho$ to identify $\mathfrak{a}_{\mathbb{C}}'$ with $\mathbb{C}$. We will use this identification during the rest of the chapter.
  
\begin{lemma}\label{lemma:hyperplane}
 There exists a constant $q_{\sigma}$ such that
\[
 J_{\sigma,\nu}^{\chi}(\phi)=\int_{N}\chi(n)^{-1}\phi_{\sigma,\nu}(w^{M} n)\, dn
\]
converges absolutely and uniformly in compacta of $\{\nu \in \mathfrak{a}_{\mathbb{C}}'\, | \, \operatorname{Re} \nu > q_{\sigma}\}$ for all $\phi\in I^{\infty}_{\sigma,\nu}$. 
\end{lemma}
\begin{proof}
From the Iwasawa decomposition of $G$, we have that for any $n\in N$, there exists $n(w^{M}n) \in N$, $a(w^{M}n)\in A$, $m(w^{M}n) \in M$ and $k(w^{M}n) \in K$ such that 
\begin{equation}
 \|\phi_{\sigma,\nu}(w^{M}n)\|= a(w^{M}n)^{\operatorname{Re} \nu +\rho}\|\sigma(m(w^{M}n))\phi_{\sigma,\nu}(k(w^{M}n))\|. \label{eq:sigmaabsolutevalueestimate}
\end{equation}
On the other hand, in [W3 proof of 4.5.6] it is proved that there exists some constants $C_{1}$ and $r$, depending on $\sigma$, such that 
\[
 \|\sigma(m(w^{M}n))\| \leq C_{1}\|n\|^{r}.
\]
Now in [W3 lemma 4.A.2.3] it's shown that there exists a constant $C_{2}$ and linear functional $\mu >0$ such that
\[
 \|n\|\leq C_{2}a(w^{M}n)^{-\mu}.
\]
Therefore
\[
 \|\sigma(m(w^{M}n))\| \leq Ca(w^{M}n)^{rd\mu},
\]
for some constant $C>0$.
From this equation and (\ref{eq:sigmaabsolutevalueestimate}) we conclude that
\[
  \|\phi_{\sigma,\nu}(w^{M}n)\| \leq Ca(w^{M}n)^{-r\mu + \operatorname{Re} \nu +\rho}\|\phi_{\sigma,\nu}(k(w^{M}n))\|.
\]
Therefore, it is enough to show that there exists a constant $q_{\sigma}$ such that, if $\nu \in \mathfrak{a}_{\mathbb{C}}'$ and $\operatorname{Re} \nu > q_{\sigma}$, then
\[
 \int_{N}  a(w^{M}n)^{-r\mu+\operatorname{Re} \nu +\rho} \, dn < \infty.
\]
But this follows directly from [W3, thm 4.5.4].
\end{proof}

Let $\nu \in \mathfrak{a}_{\mathbb{C}}'$ be such that $\operatorname{Re} \nu > q_{\sigma}$, and let $\mu\in V_{\sigma}'$. Define a linear functional $\lambda_{\mu}$ of $I_{\sigma,\nu}$ by
\[
 \lambda_{\mu}(\phi_{\sigma,\nu})=\mu\circ J_{\sigma,\nu}^{\chi}(\phi)=\mu\left(\int_{N}\chi(n)^{-1}\phi_{P,\sigma,\nu}(w^{M}n)\, dn\right).
\]
It's easy to check that $\lambda_{\mu}\in Wh_{\chi}(I_{\sigma,\nu}^{\infty})$. Furthermore, given $v\in V_{\sigma}$ we can find an element $\phi \in I_{\sigma,\nu}$ with support on $Pw^{M}N$ such that $J_{\sigma,\nu}^{\chi}(\phi)=v$. This means that if $\lambda_{\mu}=\lambda_{\mu'}$, then $\mu=\mu'$. In the next section we will describe how we can use the theory of the transverse symbol of Kolk-Varadarajan to define a map
\[
 \Phi_{\sigma,\nu}:Wh_{\chi}(I_{\sigma,\nu}^{\infty}) \longrightarrow V_{\sigma}'
\]
for all $\nu\in \mathfrak{a}_{\mathbb{C}}'$ such that if $\operatorname{Re} \nu > q_{\sigma}$, then $\Phi_{\sigma, \nu}(\lambda_{\mu})=\mu$. Furthermore, we will show how we can make use of lemma \ref{lemma:compactvanishing} to show that this map $\Phi_{\sigma,\nu}$ is injective in the case where $\chi$ has compact stabilizer.

%\section{Kolk-Varadarajan theory of the Transverse symbol}
\section{The transverse symbol of Kolk-Varadarajan}
Let $H$ be a Lie group, and let $X$ be a $C^{\infty}$ manifold with a left $H$ action. Given a Fr\'echet space $E,$ let $C_{c}^{\infty}(X:E)$ be the space of smooth, compactly supported functions on $X$ with values in $E$. We will denote by 
\[
D'(X:E):= (C_{c}^{\infty}(X:E))'
\]
its dual space, and we will make the identification 
\[
D'(X:E) \longleftrightarrow Hom(C_{c}^{\infty}(X),E').  
\]
We will call any element in this space an $E$-distribution on $X$.

Fix an $H$-orbit $\mathcal{O}\subset X$. Let $\operatorname{Diff}^{(r)}$ be the sheaf of differential operators of order $\leq r$ on $X$. For any $x\in X$ let $V_{x}^{(r)}$ be the subspace of $\operatorname{Diff}_{x}^{(r)}$ generated by germs of $r$-tuples $v_{1}\cdots v_{r}$ of vector fields around $x$ for which at least one of the $v_{i}$ is tangent to $\mathcal{O}$. Let 
\[ 
I_{x}^{(r)}=\operatorname{Diff}_{x}^{(r-1)}+V_{x}^{(r)}.
\]
Choosing local coordinates at $x$ it can be seen that $I_{x}^{(r)}$ actually is the stalk at $x$ of a subsheaf $I^{(r)} \subset \operatorname{Diff}^{(r)}$ \cite{kv:transverse}. Hence we have a well-defined quotient sheaf
\[
M^{(r)}=\operatorname{Diff}^{(r)}/I^{(r)}.
\]
with stalk at $x$ equal to $M_{x}^{(r)}=\operatorname{Diff}_{x}^{(r)}/I_{x}^{(r)}$. It can be checked that $M^{(r)}$ is a vector bundle over $O$ of finite rank \cite{kv:transverse}. This is the $r$-th graded part of the transverse jet bundle on $\mathcal{O}$. Observe that $M^{(r)}$ is the $r$-th symmetric power of $M^{(1)}$. 

We say that $T\in D'(X:E)$ has transverse order $\leq r$ at $x\in \mathcal{O},$ if there exists an open neighborhood $U$ of $x$ in $X,$ such that for all $f\in C_{c}^{\infty}(U:E),$ with the property that $Df|_{O\cap U}=0$ for all $D \in \operatorname{Diff}^{(r)}(U),$ $T(f)=0$. Let $D_{O}^{\prime (r)}(X:E)$  be the linear subspace of elements in $D'(X:E)$ which have transverse order $\leq r$ at all points of $O$. Observe that if $T\in D_{O}^{\prime (r)}(X:E)$, then $\operatorname{supp} T \subset \mathcal{O}$, which justifies the notation.

Given a normal subgroup $H'\subset H$, and a point $y\in \mathcal{O}$, define a character $\chi_{y}$ of $H'_{y}=\{h\in H'\, | \, h\cdot y=y\}$ by
\[
 \chi_{y}(h)=\frac{\delta_{H'}(h)}{\delta_{H'_{y}}(h)},
\]
where $\delta_{H'}$ is the modular function of $H_{y}$ and $\delta_{H'_{y}}$ is the modular function of $H'_{y}$.
The following theorem is a restatement of theorems 3.9, 3.11 and 3.15 of \cite{kv:transverse}.

\begin{theorem}[Kolk-Varadarajan] \label{thm:kolk-varadarajan}
Let $X$ be a $C^{\infty}$ manifold with a left action of $H,$ let $(\pi,E)$ be a smooth Fr\'echet representation of a normal subgroup $H'$ of $H,$ and let $O\subset X$ be an $H$-orbit of $X$.  
\begin{enumerate}
\item Assume that the action of $H'$ can be extended to an action of $H$. If there exists $y\in O,$ such that
\[
(M_{y}^{(r)}\otimes E'\otimes \mathbb{C}_{y}')^{H_{y}'}=(0),
\]
for all $r\in\mathbb{Z}_{\geq 0},$ then 
\[
D_{O}'(X:E)^{H'}=(0).
\]
\item Assume that $H=H'$. Then for any $$T\in D_{O}^{\prime (r)}(X:E)/D_{O}^{\prime (r-1)}(X:E),$$ there exists $\mu_{y} \in (M_{y}^{(r)}\otimes E'\otimes \mathbb{C}_{y}')^{H_{y}}$ such that
\[
T(f)=\int_{H/H_{y}} (h\cdot \mu_{y})(f) \, dh
\]
\item Assume that $E$ is finite dimensional, and assume that for all $y\in O$
\[
(M_{y}^{(r)}\otimes E'\otimes \mathbb{C}_{y}')^{H_{y}'}=(0),
\]
for all $r\in\mathbb{Z}_{\geq 0},$ then 
\[
D_{O}'(X:E)^{H'}=(0).
\]
\end{enumerate}
\end{theorem}

We will now show how we can use this result to define a linear map
\[
 \Phi_{\sigma,\nu}:Wh_{\chi}(I_{\sigma,\nu}^{\infty}) \longrightarrow V_{\sigma}'.
\]

Given $f\in C^{\infty}_{c}(G),$ and $v\in V_{\sigma},$ set
\[
 f_{P,\sigma,\nu,v}(g)=\int_{P}f(pg)\sigma_{\nu}(p)^{-1}v \, d_{r}p.
\]
Then
\[
 f_{P,\sigma,\nu,v}(pg)=\sigma_{\nu}(p)f(g), \qquad \mbox{i.e} \qquad f_{P,\sigma,\nu,v} \in I^{\infty}_{P,\sigma_{\nu}}.
\]
 Let
\[
U_{P,\sigma_{\nu}}=\{f\in I^{\infty}_{P,\sigma_{\nu}} \, | \, \operatorname{supp} f\subset P(w^{M})^{\ast}N\}.
\]
Then, given $f\in C^{\infty}_{c}(G)$ such that $\operatorname{supp}f \subset P(w^{M})^{\ast}N,$ $f_{P,\sigma,\nu,v} \in U_{P,\sigma_{\nu}}$.    Furthermore, the span of the $f_{P,\sigma,\nu,v}$'s constructed this way is dense in $U_{P,\sigma_{\nu}}$.

Let 
\[
 D'(P(w^{M})^{\ast}N:V_{\sigma})= \{T:C^{\infty}_{c}(P(w^{M})^{\ast}N) \longrightarrow V_{\sigma}'\}
\]
 be the space of $V_{\sigma}$ distributions on $P(w^{M})^{\ast}N$. Given $\lambda \in Wh_{\chi}(I_{P,\sigma_{\nu}}),$ define $\bar{\lambda} \in D'(P(w^{M})^{\ast}N:V_{\sigma})$ by
\[
 \bar{\lambda}(f)(v)=\lambda(f_{P,\sigma,\nu,v}).
\]
It's easy to check that actually 
$$\bar{\lambda}\in D'(P(w^{M})^{\ast}N:V_{\sigma_{\nu-2\rho}}\otimes \mathbb{C}_{\chi})^{P\times N}.$$
Hence, according to part 2 of Kolk-Varadarajan theorem,  there exist $\mu_{\lambda}\in V_{\sigma}'$ such that
\begin{eqnarray*}
 \bar{\lambda}(f)(v)& = & \mu_{\lambda}\left(\int_{N}\int_{P}\chi(n)^{-1}f(pw^{M}n)\sigma_{\nu}(p)^{-1}v \, d_{r}p\, dn \right)\\
 \lambda(f_{P,\sigma,\nu,v})& = & \mu_{\lambda}\left(\int_{N}\chi(n)^{-1} f_{P,\sigma,\nu,v}(w^{M}n) \, dn\right)\\
 & = & \mu_{\lambda}\circ J_{P,\sigma_{\nu}}^{\chi}(f_{P,\sigma,\nu,v}|_{K}).
\end{eqnarray*}  
We will denote the map $\lambda \mapsto \mu_{\lambda}$ by $\Phi_{P,\sigma_{\nu}}$. Observe that if $\operatorname{Re} \nu > q_{\sigma}$ then $\Phi_{\sigma,\nu}$ is surjective, since given $\mu\in V_{\sigma}'$ we can define $\lambda_{\mu}=\mu\circ J_{\sigma,\nu}$ and it's then clear that $\Phi_{\sigma,\nu}(\lambda_{\mu})=\mu$.

\begin{proposition}\label{prop:vanishing}
Assume that the character $\chi$ has compact stabilizer. If $\lambda \in Wh_{\chi}(I_{P,\sigma_{\nu}}^{\infty})$ and $\lambda_{|U_{P,\sigma_{\nu}}}=0$
then $\lambda=0$.
\end{proposition}

\begin{corollary}
If $\chi$ has compact stabilizer, then the map
 \[
\Phi_{P,\sigma_{\nu}}:Wh_{\chi}(I_{P,\sigma_{\nu}}) \longrightarrow V_{\sigma}'
\] is injective.
\end{corollary}

\begin{proof}[Proof (of proposition)]
We will first reduce the problem to the case where $\sigma$ is an induced representation. If $(\eta,V)$ is a finite dimensional representation of $P_{\circ}$, we define $I^{\infty}_{P_{\circ},\eta}$ to be the space of smooth $\phi:G \longrightarrow V$ such that $\phi(pg)=\eta(p)\phi(g)$ for $p\in P_{\circ}$ and $g\in G$. Set $\pi_{\eta}(g)\phi(x)=\phi(xg)$, for $x,g\in G$. If we endow $I^{\infty}_{P_{\circ},\eta}$ with the $C^{\infty}$ topology, then $(\pi_{\eta},I^{\infty}_{P_{\circ},\eta})$ is an admissible, smooth, Fr\'echet representation of moderate growth. Let $(\xi,F)$ be a finite dimensional representation of $P_{M}:=P_{\circ}\cap M$ and let $(\pi_{\xi},I^{\infty}_{P_{M},\xi})$ be the corresponding representation of $M$ (Observe that $P_{\circ}\cap M$ is a minimal parabolic subgroup of $M$). The Casselman-Wallach theorem implies that there exists a surjective, continuous, $M$-intertwining operator $L:I^{\infty}_{P_{M},\xi} \longrightarrow V_{\sigma}$ for some finite dimensional representation $(\xi, F)$ of $P_{M}$. This map lifts to a surjective $G$ interwining map $\tilde{L}:I^{\infty}_{P,\pi_{\xi},\nu} \longrightarrow I^{\infty}_{P,\sigma,\nu}$, given by $\tilde{L}(\phi)(g)=L(\phi(g))$ for $\phi \in I^{\infty}_{P,\pi_{\xi},\nu}$, $g\in G$. The representation  $I^{\infty}_{P,\pi_{\xi},\nu}$ is equivalent to the representation smoothly induced from $P_{\circ}$ to $G$ by the representation $\xi_{\nu}$ of $P_{\circ}$ with values on $F$ defined as follows:
\[
 \xi_{\nu}(nap)=a^{\nu+\rho}\xi(p) \qquad \mbox{for $p\in P_{M}$, $a\in A$, $n\in N$.}
\]
Setting $\eta=\xi_{\nu}$ we can identify the map $\tilde{L}$ with a surjective $G$-equivariant map $\tilde{L}:I^{\infty}_{P_{\circ},\eta} \longrightarrow I^{\infty}_{P,\sigma,\nu}$. Set $U_{P_{\circ},\eta}=\{\phi\in I^{\infty}_{P_{\circ},\eta} \, | \, \mbox{$\operatorname{supp} \phi \subset Pw^{M}N = P_{\circ} w^M P$}\}$, and define $Wh_{\chi}(I^{\infty}_{P_{\circ},\eta})$ in the same way as above. Assume that we have proved the proposition for $I^{\infty}_{P_{\circ},\eta}$, i.e., assume that if $\lambda \in Wh_{\chi}(I^{\infty}_{P_{\circ},\eta})$ and $\lambda|_{U_{P_{\circ},\eta}}=0$, then $\lambda=0$. Let $\lambda\in Wh_{\chi}(I^{\infty}_{P,\sigma,\nu})$ be such that $\lambda|_{U_{P,\sigma,\nu}}=0$ and let $\tilde{\lambda}=\tilde{L}^{\ast}\lambda$ be the pullback of $\lambda$ to $Wh_{\chi}(I^{\infty}_{P_{\circ},\eta})$ by $\tilde{L}$. It's easy to check, using the definition of $\tilde{L}$, that $\tilde{\lambda}|_{U_{P_{\circ},\eta}}=0$ and hence, by our assumptions, $\tilde{\lambda}=\tilde{L}^{\ast}\lambda=0$, but $\tilde{L}$ is surjective, therefore $\lambda=0$. We will now prove the proposition for $Wh_{\chi}(I^{\infty}_{P_{\circ},\eta})$.

Let $\lambda \in Wh_{\chi}(I^{\infty}_{P_{\circ},\eta})$ be such that $\lambda|_{U_{P_{\circ},\eta}}=0$. Proceeding as before, we can define a distribution 
$$\bar{\lambda}\in D'(G:F \otimes \mathbb{C}_{\chi})^{N_{\circ}\times N}$$
that vanishes on the big Bruhat cell. Now, if we can prove that
\[
D_{P_{\circ}w_{v}K_{M}N}'(G:F\otimes \mathbb{C}_{\chi})^{N_{\circ}\times N}=(0) \qquad \forall v\in W^{M}, \quad v \neq v^{M},
\]
then, the standard Bruhat theoretic argument shows that $\bar{\lambda}$, and hence $\lambda,$ is equal to $0$.

Now, since $K_{M}=M_{\chi}$, we can extend the action of $N_{\circ}\times N$ on $F\otimes \mathbb{C}_{\chi}$ to an action of $P_{\circ}\times K_{M}N$. Therefore, from part 1 of theorem \ref{thm:kolk-varadarajan}, we just need to show that
\[
  (M_{w_{v}}^{(r)}\otimes (F\otimes \mathbb{C}_{\chi})')^{(N_{\circ}\times N)_{w_{v}}}=(0), \qquad \forall r\geq 0.
\]
But this follows from the fact $N_{\circ}$ acts unipotently on $M_{w_{v}}^{(r)}\otimes F'$ and that the restriction of $\chi$ to $w_{v}^{-1}N_{\circ}w_{v} \cap N $ is non-trivial, according to lemma \ref{lemma:compactvanishing}.
\end{proof}

Although proposition \ref{prop:vanishing} is false in the case where the stabilizer of $\chi$ is non-compact, we will show later how we can obtain a similar result in the general case. But before we are able to prove this result we will need dig a little bit more into the structure of Lie groups of tube type.

\section{The Bruhat-Matsuki decomposition of a Lie group of tube type}

An affine symmetric space is a triple $(G, H, \sigma)$ consisting of a connected Lie group $G$, a closed subgroup $H$ of $G$ and an involutive automorphism $\sigma$ of $G$ such that $H$ lies between $G_{\sigma}$ and the identity component of $G_{\sigma}$, where $G_{\sigma}$ denotes the closed subgroup of $G$ consisting of all the elements fixed by $\sigma$. If $G$ is real semi-simple, Matsuki \cite{m:orbits} has given an explicit double coset decomposition of the space $H\backslash G /P$, where $P$ is a minimal parabolic subgroup of $G$. His construction goes as follows.

Let
$(G, H, \sigma)$ be an affine symmetric space such that $G$ is real semi-simple, and $(\mathfrak{g},\mathfrak{h},\sigma)$ the corresponding symmetric algebra. Let $\theta$ be a Cartan involution commutative with $\sigma$, and $\mathfrak{g}=\mathfrak{k}\oplus\mathfrak{p}$ the corresponding Cartan decomposition. Since the factor space $G/P$ is identified with the set of all minimal parabolic subalgebras of $\mathfrak{g}$, the following theorem and corollary give a complete characterization of the $H$-orbits on $G/P$.
\begin{theorem}
\begin{enumerate}
 \item Let $P_{\circ}$ be a minimal parabolic subalgebra of $\mathfrak{g}$. Then there exists a $\sigma$-stable maximal abelian subspace $\mathfrak{a_{\mathfrak{p}}}$ of $\mathfrak{p}$ and a positive system $\Phi^{+}$ of the root system $\Phi$ of the pair $(\mathfrak{g},\mathfrak{a}_{\mathfrak{p}})$ such that $P_{\circ}$ is $H_{\circ}$-conjugate to $\mathcal{P}(\mathfrak{a}_{\mathfrak{p}},\Phi^{+})$, where $H_{\circ}$ is the identity component of $H$, $\mathcal{P}(\mathfrak{a}_{\mathfrak{p}},\Phi^{+})=\mathfrak{m}\oplus \mathfrak{a}_{\mathfrak{p}}\oplus \mathfrak{n}$, $\mathfrak{m}=z(\mathfrak{a}_{\mathfrak{p}})$, $\mathfrak{n}=\sum_{\alpha\in {\Phi^{+}}}\mathfrak{g}_{\alpha},$ and $g_{\alpha}=\{X\in \mathfrak{g}\, | \, [Y,X]=\alpha(Y)X \mbox{for all $Y \in P_{\circ}$}\}$.

\item Let $\mathfrak{a}_{\mathfrak{p}}$ and $\mathfrak{a}_{\mathfrak{p}}'$ be $\sigma$-stable maximal abelian subspaces of $\mathfrak{p}$, and $\Phi^{+}$, $(\Phi^{+})'$ be positive systems of root systems $\sigma(\mathcal{P}(\mathfrak{a}_{\mathfrak{p}})$ and $\Phi(\mathcal{P}(\mathfrak{a}_{\mathfrak{p}}')$ respectively. If $\mathcal{P}(\mathfrak{a}_{\mathfrak{p}},\Phi^{+})$ and $\mathcal{P}(\mathfrak{a}_{\mathfrak{p}}',(\Phi^{+})')$ are $H$-conjugate, then $\mathfrak{a}_{\mathfrak{p}}$ and $\mathfrak{a}_{\mathfrak{p}}'$ are $K_{+}$-conjugate $(K_{+}=H\cap K)$.
\end{enumerate}
\end{theorem}

If $\mathfrak{a}_{\mathfrak{p}}$ is a $\sigma$-stable maximal abelian subspace of $\mathfrak{p}$, we can define a subgroup $W(\mathfrak{a}_{\mathfrak{p}},K_{+})$ of the Wey group $W(\mathfrak{a}_{\mathfrak{p}})$ by $W(\mathfrak{a}_{\mathfrak{p}},K_{+}) = (M^{\ast}(\mathfrak{a}_{\mathfrak{p}})\cap K_{+}/(M(\mathfrak{a}_{\mathfrak{p}}\cap K_{+})$, where $M^{\ast}(\mathfrak{a}_{\mathfrak{p}})=N_{K}(\mathfrak{a}_{\mathfrak{p}})$, and $M(\mathfrak{a}_{\mathfrak{p}})=Z_{K}(\mathfrak{a}_{\mathfrak{p}})$.

\begin{corollary}
 Let $\{\mathfrak{a}_{\mathfrak{p}_{i}} \, | \, i\in I\}$ be representatives of the $K_{+}$-conjugacy classes of $\sigma$-stable maximal abelian subspaces of $\mathfrak{p}$. Then there exists a one-to-one correspondence between the $H$-conjugacy classes of minimal parabolic subalgebras of $\mathfrak{g}$ and
\[
 \cup_{i\in I} W(\mathfrak{a}_{\mathfrak{p}},K_{+}) \backslash W(\mathfrak{a}_{\mathfrak{p}_{i}})
\]
where the union is disjoint. The correspondence is given as follows. Fix a positive system $\Phi^{+_{i}}$ of $\Phi(\mathfrak{a}_{\mathfrak{p}_{i}})$ for each $i\in I$. Then $W(\mathfrak{a}_{\mathfrak{p}},K_{+})w\in \cup_{i\in I} W(\mathfrak{a}_{\mathfrak{p}},K_{+}) \backslash W(\mathfrak{a}_{\mathfrak{p}_{i}})$ corresponds to the $H$-conjugacy class of minimal parabolic subalgebras of $\mathfrak{g}$ containing $\mathcal{P}(\mathfrak{a}_{\mathfrak{p}},w\Phi^{+}_{i})$.
\end{corollary}

Let $\mathfrak{a_{p}}$ be a $\sigma$-stable maximal abelian subspaces of $\mathfrak{p}$ such that $\mathfrak{a}_{\mathfrak{p}^{+}}=\mathfrak{a}_{\mathfrak{p}}\cap \mathfrak{p}$ is maximal abelian in $\mathfrak{p}_{+}=\mathfrak{p}\cap \mathfrak{h}$. put $q=\{X\in \mathfrak{g}\, | \, \sigma(X)=-X\}$, and $\Phi(\mathfrak{a}_{\mathfrak{p}^{+}})=\{\alpha \in \Phi(\mathfrak{a}_{\mathfrak{p}})\, | \, H_{\alpha}\in \mathfrak{a}_{\mathfrak{p}^{+}}\}$, where $H_{\alpha}$ is the unique element in $\mathfrak{a}_{\mathfrak{p}}$ such that $B(H_{\alpha},H)=\alpha(H)$ for all $H \in \mathfrak{a}_{\mathfrak{p}}$ ($B$ is the Killing form of $\mathfrak{g}$). Let $\alpha_{i}$, $i=1,\ldots, k$ be elements of $\Phi(\mathfrak{a}_{\mathfrak{p}^{+}})$ and $X_{\alpha_{i}}$, $i=1,\ldots, k$ be on-zero elements of $\mathfrak{g}_{\alpha_{i}}$. Then $\{X_{\alpha_{1}},\ldots,X_{\alpha_{k}}\}$ is said to be a $q$-orthogonal system of $\Phi(\mathfrak{a}_{\mathfrak{p}^{+}})$ if the following two conditions are satisfied:
(i) $X_{\alpha_{i}} \in \mathfrak{q}$ for $i=1,\dots, k,$,
(ii) $[X_{\alpha_{i}},X_{\alpha_{j}}]=0$ and $[X_{\alpha_{i}},\theta(X_{\alpha_{j}})]=0$ for $i,j=1,\dots, k,$ $i\neq j$.

Two $\mathfrak{q}$-orthogonal systems $\{X_{\alpha_{1}},\ldots,X_{\alpha_{k}}\}$ and $\{Y_{\beta_{1}},\ldots,Y_{\beta_{k}}\}$ are said to be conjugate under $W(\mathfrak{a}_{\mathfrak{p}},K_{+})$ if there is a $w\in W(\mathfrak{a}_{\mathfrak{p}},K_{+})$ such that $w(\sum_{i=1}^{k}\mathbb{R}H_{\alpha_{i}})=\sum_{i=1}^{k}\mathbb{R}H_{\beta_{i}}$. Then the following theorem gives a complete characterization of the $K_{+}$-conjugacy classes of $\sigma$-stable maximal abelian subspaces of $\mathfrak{p}$.

\begin{theorem}
Let $(G,H,\sigma)$ be an affine symmetric spaces such that $g$ is real semi-simple, $\theta$ a Cartan involution of $\mathfrak{g}$ commutative with $\sigma$, and $\mathfrak{g}=\mathfrak{k}\oplus \mathfrak{p}$ the corresponding Cartan decomposition of $\mathfrak{g}$. Let $\mathfrak{a}_{\mathfrak{p}^{+}}$ be a maximal abelian subspace of $\mathfrak{p}^{+}$ and $\mathfrak{a}_{\mathfrak{p}}$ a maximal abelian subspace of $\mathfrak{p}$ containing $\mathfrak{a}_{\mathfrak{p}^{+}}$. Then there exists a one-to-one correspondence between the $K_{+}$-conjugacy classes of $\sigma$-stable maximal abelian subspaces of $\mathfrak{p}$ and the $W(\mathfrak{a}_{\mathfrak{p}},K_{+})$-conjugacy classes of $q$ orthogonal systems of $\Phi(\mathfrak{a}_{\mathfrak{p}^{+}})$. The correspondence is given as follows: Let $Q=\{X_{\alpha_{1}},\ldots,X_{\alpha_{k}}\}$ be a $\mathfrak{q}$-orthogonal system of $\Phi(\mathfrak{a}_{\mathfrak{p}^{+}})$. Put $\mathfrak{r}=\sum_{i=1}^{k}\mathbb{R}H_{\alpha_{i}}$, $\mathfrak{a}_{\mathfrak{p}^{+}}'=\{H\in \mathfrak{a}_{\mathfrak{p}^{+}} \, | \, B(H,\mathfrak{r})=0\}$, $\mathfrak{a}_{\mathfrak{p}^{-}}'=\mathfrak{a}_{\mathfrak{p}_{-}}+\sum_{i=1}^{k}\mathbb{R}(X_{\alpha_{i}}-X_{-\alpha_{i}})$, where $\mathfrak{a}_{\mathfrak{p}_{-}}=\mathfrak{a}_{\mathfrak{p}}\cap \mathfrak{q}$, and $\mathfrak{a}_{\mathfrak{p}}'=\mathfrak{a}_{\mathfrak{p}^{+}}'+\mathfrak{a}_{\mathfrak{p}^{-}}'$. Then the $W(\mathfrak{a}_{\mathfrak{p}},K_{+})$-conjugacy class of $\mathfrak{q}$-orthogonal system of $\Phi(\mathfrak{a}_{\mathfrak{p}^{+}})$ containing $Q$ corresponds to the $K_{+}$-conjugacy class of $\sigma$-stable maximal abelian subspace of $\mathfrak{p}$ containing $\mathfrak{a}_{\mathfrak{p}}'$. Moreover if $X_{\alpha_{i}}$, $-i=1,\dots, k,$ is normalized such that $2\alpha_{i}(H_{\alpha_{i}})B(X_{\alpha_{i}},X_{-\alpha_{i}})=-1$, then 
\[
 \mathfrak{a}_{\mathfrak{p}}'=Ad(\exp{(\pi/2)(X_{\alpha_{1}}+X_{-\alpha_{1}})})\cdots Ad(\exp{(\pi/2)(X_{\alpha_{k}}+X_{-\alpha_{k}})})\mathfrak{a}_{\mathfrak{p}}, 
\]
where $X_{-\alpha_{i}}=\theta(X_{\alpha_{i}})$.
\end{theorem}

\begin{theorem}\label{thm:matsukidecomposition}
 Let $(G,H,\sigma)$ be an affine symmetric space such that $G$ is real semi-simple, $\theta$ a Cartan involution commutative with $\sigma$, and $\mathfrak{g}=\mathfrak{k}\oplus\mathfrak{p}$ the corresponding Cartan decomposition. Let $\mathfrak{a}_{\mathfrak{p}}$ be a maximal abelian subspace of $\mathfrak{p}$ such that $\mathfrak{a}_{\mathfrak{p}^{+}}$ is maximal abelian in $\mathfrak{p}_{+}$, and $\{Q_{1},\ldots, Q_{m}\}$ be representatives of $W(\mathfrak{a}_{\mathfrak{p}},K_{+})$-conjugacy classes of $\mathfrak{q}$-orthogonal systems of $\Phi(\mathfrak{a}_{\mathfrak{p}^{+}})$. Suppose that $Q_{j}=\{X_{\alpha_{1}},\ldots ,X_{\alpha_{k}}\}$ is normalized such that $2\alpha_{i}(H_{\alpha_{i}})B(X_{\alpha_{i}},X_{-\alpha_{i}})=-1$, $i=1,\ldots, k$ for each $j=1,\ldots,m$. Put $c(Q_{j})=\exp{(\phi/2)(X_{\alpha_{1}}+X_{-\alpha_{1}})\cdots(\phi/2)(X_{\alpha_{k}}+X_{-\alpha_{k}})}$. Then
\begin{enumerate}
 \item We have the following decomposition of $G$.
\[
 G=\cup_{i=1}^{m}\cup_{v\in W(\mathfrak{a}_{\mathfrak{p}},K_{+})}\backslash W(\mathfrak{a}_{\mathfrak{p_{i}}}) Hw_{v}c(Q_{i})P \qquad \mbox{(disjoint union)}
\]
where $P=P(\mathfrak{a}_{\mathfrak{p}},\Phi^{+})$, $\Phi^{+}$ is a positive system of $\Phi(\mathfrak{a}_{\mathfrak{p}})$, $\mathfrak{a}_{\mathfrak{p}_{i}}=Ad(c(Q_{i}))\mathfrak{a}_{\mathfrak{p}}$, and $w_{v}$is an element of $M^{\ast}(\mathfrak{a}_{\mathfrak{p}_{i}})$ that represents an element of the left coset $v\subset W(\mathfrak{a}_{\mathfrak{p_{i}}})$.
\item Put $P_{i,w_{v}}=w_{v}c(Q_{i})Pc(Q_{i})^{-1}w_{v}^{-1}$. Let $h_{1}, h_{2}\in H$ and $p_{1},p_{2}\in P$. Then $h_{1}w_{v}c(Q_{i})p_{1}=h_{2}w_{v}c(Q_{i})p_{2}$ if and only if there exists an $x\in H\cap P_{i,w_{v}}$ such that $h_{2}=h_{1}x$ and that $p_{2}=c(Q_{i})^{-1}w_{v}^{-1}x^{-1}w_{v}c(Q_{i})p_{1}.$
\item Let $P=P(\mathfrak{a}_{\mathfrak{p}}',\Phi^{+})=MA_{\mathfrak{p}}'N^{+}$ be a minimal parabolic subgroup of $G$ such that $\mathfrak{a}_{\mathfrak{p}}'$ is $\sigma$-stable. Then
\[
 H\cap P = (K_{+}\cap M)A_{\mathfrak{p}}'\exp{(\mathfrak{h}\cap\mathfrak{n}^{+}\cap \sigma \mathfrak{n}^{+})}.
\]

\end{enumerate}

\end{theorem}

In this section we will record the decomposition of some relevant symmetric spaces with respect to the action of a minimal parabolic subgroups. Let $G=GL(n,F)$, $H=U(p,q,F)$, for $F=\mathbb{R}$, $\mathbb{C}$, $\mathbb{H}$ or $\mathbb{O}$.
Let
\[
B^{i}_{p,q}=\left[\begin{array}{ccccc} 0_{i} & & & I_{i}& \\ & I_{q-i} & & & 0_{q-i} \\ & & I_{p-q} & & \\ I_{i} & & & 0_{i} & \\ & 0_{q-i} & & & I_{q-i}\end{array} \right]
\]
for $n=p+q$, $p\geq q \geq i$, and define $\sigma^{i}: G \longrightarrow G$ by
\[
\sigma^{i}(g)=B^{i}_{p,q}\theta(g)B^{i}_{p,q}.
\]
Then $\sigma^{i}$ is an involution, and $\sigma^{i}\theta=\theta\sigma^{i}$ for all $i$. Let $H^{i}=\{g\in G\, | \, \sigma^{i}(g)=g\}$. Then $(G,H^{i},\sigma^{i})$ is a symmetric space. Let $\mathfrak{h}^{i}=Lie(H^{i})$, then 
\[
\mathfrak{h}^{i}=\left.\left\lbrace \left[\begin{array}{ccc} A & -T^{\ast} & X \\ -S^{\ast} & Z & T\\ Y & S & -A^{\ast} \end{array}\right] \,\right| \, \begin{array}{c} \mbox{$A\in \mathfrak{gl}(i,F)$, $Z\in \mathfrak{u}(p-i,q-i,F)$},\\ \mbox{ $S,T \in End_{F}(F^{q-i},F^{p-i,q-i})$},\\ \mbox{ $X^{\ast}=-X$, $Y^{\ast}=-Y$ } \end{array}\right\rbrace 
\]
Let
\[
\mathfrak{a}=\left\lbrace \left. \left[\begin{array}{ccc} \lambda_{1} & & \\ & \ddots & \\ & & \lambda_{n}\end{array}\right]\, \right| \, \lambda_{i}\in \mathbb{R}\right\rbrace,
\]
and observe that $\mathfrak{a}\cap \mathfrak{h}^{q}$ is maximal abelian in $\mathfrak{p}\cap\mathfrak{h}^{q}$, where
\[
\mathfrak{g}=\mathfrak{k}\oplus \mathfrak{p}
\]
is the Cartan decomposition of $\mathfrak{g}$. Let
\[
\mathfrak{q}^{q}=\{X\in \mathfrak{g} \, | \, \sigma^{q}(X)=-X \},
\]
then
\[
\mathfrak{q}^{q}=\left.\left\lbrace \left[\begin{array}{ccc} A & T^{\ast} & X \\ S^{\ast} & Z & T\\ Y & S & A^{\ast} \end{array}\right] \,\right| \, \begin{array}{c} \mbox{$A\in \mathfrak{gl}(i,F)$, $Z^{\ast}=Z$},\\ \mbox{ $S,T \in End_{F}(F^{q-i},F^{p-i,q-i})$},\\ \mbox{ $X^{\ast}=X$, $Y^{\ast}=Y$ } \end{array} \right\rbrace.
\]
Let 
\[
X_{i}=\left[\begin{array}{ccc} 0 & & E_{ii} \\ & 0 & \\ & & 0 \end{array} \right],
\]
then $\{X_{1},\ldots,X_{q}\}$, is a representative of the unique conjugacy class of maximal $\mathfrak{q}$-orthogonal systems. Let
\[
c_{i}=\left[\begin{array}{ccccc} \frac{\sqrt{2}}{2}I_{i} & & & \frac{\sqrt{2}}{2}I_{i}& \\ & I_{q-i} & & & 0_{q-i} \\ & & I_{p-q} & & \\ -\frac{\sqrt{2}}{2}I_{i} & & & \frac{\sqrt{2}}{2}I_{i} & \\ & 0_{q-i} & & & I_{q-i}\end{array} \right]
\]
then $c_{i}B^{0}_{p,q}c_{i}^{-1}=B^{i}_{p,q}$, $c_{i}H^{0}c_{i}^{-1}=H^{i}$, and
\begin{eqnarray*}
G& = &\bigcup_{i=0}^{q}\bigcup_{v\in W(\mathfrak{a},K^{i}_{+})\backslash W(\mathfrak{a})} P_{\circ}w_{v}c_{i}H^{0}, \qquad \mbox{with $K^{i}_{+}=H^{i}\cap K$} \\
 & = & \bigcup_{i=0}^{q}\bigcup_{v\in W(\mathfrak{a},K^{i}_{+})\backslash W(\mathfrak{a})} w_{v}P^{v}_{\circ}H^{i}c_{i} \qquad \mbox{with $P^{v}_{\circ}=w_{v}^{-1} P_{\circ} w_{v}$.}
\end{eqnarray*}
Observe that
\[
\mathfrak{h}^{i}\cap \mathfrak{a}=\left\lbrace\left.\left[\begin{array}{ccccccc} \lambda_{1} & & & & & & \\ & \ddots & & & & & \\ & & \lambda{i} & & & & \\ & & & 0_{p+q-2i}& & & \\ & & & & -\lambda_{1} & & \\ & & & & & \ddots \\ & & & & & & -\lambda_{i}\end{array} \right] \, \right| \, \lambda{i}\in \mathbb{R} \right\rbrace.
\]

%\section{Bruhat Decomposition for $F=\mathbb{R}$, $\mathbb{C}$, $\mathbb{H}$, $\mathbb{O}$}
Let $\chi$ be a unitary character of $\mathbb{R}$, and let $B^{i}_{p,q}$ be as before. Define a character $\chi^{i}_{p,q}$ on $N$ by
\[
\chi^{i}_{p,q}\left(\left[\begin{array}{cc} I & X \\ & I \end{array}\right]\right)= \chi(\operatorname{tr}{(B^{i}_{p,q}X)})
\]
Let $M_{\chi^{i}_{p,q}}$ be the stabilizer of $\chi^{i}_{p,q}$ in $M$. According to the Bruhat decomposition.
\begin{eqnarray}
G & = & \bigcup_{w\in W^{M}} P_{\circ}w^{\ast}P=\bigcup{w\in W^{M}}P_{\circ} w^{\ast} MN \nonumber \\
 & = & \bigcup_{w\in W^{M}} \bigcup_{i=0}^{q}\bigcup_{v\in W(\mathfrak{a},K^{i}_{+})\backslash W(\mathfrak{a})} P_{\circ}w^{\ast}P_{M}v^{\ast}c_{i}M_{\chi^{0}_{p,q}}N \nonumber \\
  & = & \bigcup_{w\in W^{M}} \bigcup_{i=0}^{q}\bigcup_{v\in W(\mathfrak{a},K^{i}_{+})\backslash W(\mathfrak{a})} P_{\circ}w^{\ast}v^{\ast}c_{i}M_{\chi^{0}_{p,q}}N \label{eq:bruhatmatsukidecomposition}\\
   & = &\bigcup_{i=0}^{q} \bigcup_{w\in W^{M}} \bigcup_{v\in W(\mathfrak{a},K^{i}_{+})\backslash W(\mathfrak{a})} w^{\ast}v^{\ast}P^{wv}_{\circ}M_{\chi^{i}_{p,q}}Nc_{i}.
\end{eqnarray}
Observe that
\[
\mathfrak{a}\cap \mathfrak{m}_{\chi^{i}_{p,q}}=\left\lbrace\left[\begin{array}{ccccccc}\lambda_{1} & & & & & & \\ & \ddots & & & & & \\ & & \lambda{i} & & & & \\ & & & 0_{p+q-2i}& & & \\ & & & & -\lambda_{1} & & \\ & & & & & \ddots \\ & & & & & & -\lambda_{i}\end{array} \right]\right\rbrace
\]
and hence $\dim{(A\cap M_{\chi^{i}_{p,q}})}=i$.

\section{The vanishing of certain invariant distributions}

\begin{dfn}
 Let $(\pi, V)$ be a smooth,  Fr\'echet, moderate growth representation of a semi-simple Lie group $G$ with finite center. We say that $V$ has a \emph{split eigenvector} if the following holds: There exists an Iwasawa decomposition $\mathfrak{g=k\oplus a_{\circ} \oplus n_{\circ}}$, where $\mathfrak{g}=Lie(G)$, and an element $H\in \mathfrak{a_{\circ}}$ such that 

\begin{enumerate}
                                                                                                                                                                                                                                                                                             \item The projection of $H$ to any simple factor of $\mathfrak{g}$ is non-zero.

\item There exists $v\in V$ such that $H\cdot v = \lambda v$ for some $\lambda \in \mathbb{C}$.
                                                                                                                                                                                                                                                                                            \end{enumerate}

\end{dfn}

\begin{lemma}
 Let $G$ be a connected semi-simple Lie group with finite center. Let $(\pi, V)$ be a smooth, irreducible, admissible  Fr\'echet representation of moderate growth with a split eigenvector. Then $\dim V < \infty$.
\end{lemma}

\begin{proof}
Let $K$ be a maximal compact subgroup of $G$ associated with a Cartan involution $\theta$.
Let $V_{K}$ be the space of $K$-finite vectors of $V$. Then the Casselman-Wallach theorem implies that if $\lambda \in (V_{K}/\mathfrak{n}_{\circ}V_{K})'$ then $\lambda$ extends to a continuous map of $V$ to $\mathbb{C}$. This implies that $V_{k}/n_{\circ}V_{K}\cong V/\overline{\mathfrak{n}_{\circ}V}$ as an $\mathfrak{a}_{\circ}$-module.

 Let $v\neq 0$ be such that $Hv=\lambda v$ with $\lambda \in \mathbb{C}$. Let $W$ be the restricted Weyl group of $G$ acting on $\mathfrak{a}_{\circ}$. If $\mathfrak{g}$ is simple over $\mathbb{R}$, then the action of $W$ on $\mathfrak{a}_{\circ}$ would be irreducible. Hence our assumption on $H$ implies that $WH$ spans $\mathfrak{a}_{\circ}$. If $s\in W$, then there exists $s^{\ast}\in K$ such that $Ad(s^{\ast})H=sH$, and $d(s^{\ast})Hs^{\ast}v=\lambda s^{\ast}v$. This implies that there exists $u\in V$ and $\xi\in (\mathfrak{a}_{\circ})_{\mathbb{C}}'$ such that if $H\in \mathfrak{a}_{\circ}$ then $Hu=\xi(H)u$. Now since the action of $\mathfrak{a}_{\circ}$ diagonalizes on $U(\mathfrak{g})$ and $U(\mathfrak{g})u=V$, the action of $\mathfrak{a}_{\circ}$ diagonalizes on $V$. Let $\xi_{1},\ldots, \xi_{m}\in \mathfrak{a}_{\mathbb{C}}'$ be the weights of $\mathfrak{a}_{\circ}$ on $V/\overline{\mathfrak{n}_{\circ}V} $. Then we see that the set of weights of $\mathfrak{a}_{\circ}$ on $V$ is contained in
\[
 \cup_{i=1}^{m}\{\xi_{i}+\alpha \, | \, \mbox{$\alpha$ a non-zero sum of positive roots of $\mathfrak{a}_{\circ}$ in $\mathfrak{n}_{\circ}$}\}.
\]
Note that $m<\infty$. Let $\bar{\mathfrak{n}}_{\circ}=\theta\mathfrak{n}_{\circ}$. These observations imply that $\bar{\mathfrak{n}}_{\circ}$ acts locally nilpotently on $V$. Now let $u$ be a weight vector for $\mathfrak{a}_{\circ}$ on $V_{K}$. (Since $V/\overline{\mathfrak{n}_{\circ}V} \cong V_{k}/n_{\circ}V_{K}$ the above argument implies that such a $u$ exists). Then we have
\[
 U(\bar{\mathfrak{n}}_{\circ})U(\mathfrak{a}_{\circ})u=U(\bar{\mathfrak{n}}_{\circ})u
\]
which is finite dimensional. Thus $V_{K}=U(\mathfrak{k})U(\bar{\mathfrak{n}}_{\circ})U(\mathfrak{a}_{\circ})u$ is finite dimensional.
\end{proof}

Observe that this lemma implies, in particular, that if $G$ is as above, and $(\pi, V)$ is an irreducible, tempered representation of $G$, then $V$ has no split eigenvectors.

\begin{proposition}\label{prop:vanishinggeneral}
Let $(\eta, V_{\eta})$ be an admissible representation  of $P$ such that $N$ acts locally unipotently, an $ V_{\eta}$ has finite length. Let
\[
I_{P,\eta}=\{f:g\longrightarrow  V_{\eta} \, |\, \mbox{$f$ is $C^{\infty}$, and $f(pg)=\eta(p)f(g)$ for all $p\in P$}\}
\]
and let
\[
U_{P,\eta}=\{f\in I_{P,\eta} \, | \, \operatorname{supp} f \subset Pw^{\ast}_{M} N\}.
\]
Let $V_{\tau}$ be a representation of $M_{\chi}$ with no split eigenvectors, and let
\[
Wh_{\chi,\tau}(I_{P,\eta}):= \{\lambda: I_{P,\eta} \longrightarrow V_{\tau}\, |\, \mbox{$\lambda(nm \cdot f)=\chi(n)\tau(m)\lambda(f)$, $\forall n\in N$, $m\in M_{\chi}$}\}.
\]
Let $\lambda \in Wh_{\chi,\tau}(I_{P,\eta})$ be such that $\lambda|_{U_{P,\eta}}=0$, then $\lambda=0$
\end{proposition}

\begin{proof}
Given $f\in C_{c}^{\infty}(G)$ and $v\in V_{\sigma}$, define $f_{v}\in I_{p,\eta}$ by
\[
f_{v}(g)=\int_{P} \eta(p)^{-1}f(pg)v \,d_{l}p
\]
It's easy to check that $f\otimes v \mapsto f_{v}$ defines a surjective continuous map from $C_{c}^{\infty}(G)\bar{\otimes} V_{\sigma}$ to $I_{P,\eta}$. Let $\lambda \in Wh_{\chi,\tau}(I_{P,\eta})$ be such that $\lambda|_{U_{P,\eta}}=0$, and define a $Hom(V_{\sigma},V_{\tau})$-valued distribution $\tilde{\lambda}$ on $G$ by
\[
\tilde{\lambda}(f)(v)=\lambda(f_{v}).
\]
To prove the proposition it suffices to show that $\tilde{\lambda}=0$.

By the definition of $\tilde{\lambda}$
\begin{eqnarray}
\tilde{\lambda}(L_{p}f)(v) & = &\tilde{\lambda}(\eta(p)^{-1}v)  \qquad \forall p \in P_{\circ}  \label{eq:equivariantP} \\
\tilde{\lambda}(R_{nm}f)(v) & = &\chi(n)\tau(m) \tilde{\lambda} \qquad \forall n\in N, \, \forall m\in M_{\chi} \label{eq:equivariantN} 
\end{eqnarray}
and $\operatorname{supp} \tilde{\lambda}\subset G\backslash Pw_M^{\ast}N$. Recall that, according to equation (\ref{eq:bruhatmatsukidecomposition})
\[
G=\bigcup_{i=0}^{q}\bigcup_{w\in W^{M}}\bigcup_{w_{v}\in W(\mathfrak{a},K_{+}^{i})\backslash W(\mathfrak{a})} P_{\circ}w^{\ast}w_{v}c_{i}M_{\chi}N,
\]
So the condition $\operatorname{supp} \tilde{\lambda}\subset G\backslash Pw_M^{\ast}N$ says that $\operatorname{supp}\tilde{\lambda}$ is contained outside the union of the open orbits. We will show that all the other orbits can't support distributions satisfying conditions (\ref{eq:equivariantP}) and (\ref{eq:equivariantN}). Then the usual Bruhat theoretic argument show that $\tilde{\lambda}=0$.

Let $\O_{w,v,i}=P_{\circ}w^{\ast}w_{v}c_{i}M_{\chi}N $ be an orbit that is not an open orbit. Then either $w^{\ast}\neq W_{M}^{\ast}$, $c_{i}\neq e$ or both. Let
\[
H_{(w,v,i)}=\{mn\,|\,\mbox{$m\in M_{\chi}$ and $w^{-1}w_{v}c{i}m^{-1}c_{i}^{-1}w_{v}w^{-1}\in P_{\circ}$}\}
\]
and let
\[
\tilde{H}_{(w,v,i)}=\{(ww_{v}c{i}m^{-1}c_{i}^{-1}w_{v}^{-1}w^{-1},m) \,| \, m\in H_{(w,v,i)}\}.
\]
Observe that $\tilde{H}_{(w,v,i)}$ is the stabilizer of $ww_{v}c_{i}$ on $P_{\circ}\times M_{\chi}N$. Define a representation $(\eta_{(w,v,i)},V_{\eta})$ of $H_{(w,v,i)}$ by
\[
h\cdot v = (ww_{v}c_{i}mc{i}^{-1}w_{v}^{-1}w^{-1})\cdot v.
\]
Let $M_{(w,v,i)}^{(r)}$ be the $r$-th transverse bundle on $\O_{(w,v,i)}$ defined by Kolk-Varadarajan \cite{kv:transverse}. Then theorem 3.9 of the aforementioned paper says that if
\begin{equation}
Hom_{H_{(w,v,i)}}(V_{\eta}\otimes M_{(w,v,i)}^{(r)},V_{t\chi})=0, \qquad \forall r, \label{eq:vectorinvariant}
\end{equation}
then
\begin{equation}
D_{\omega}(G;Hom(V_{\eta},V_{\tau}))^{P_{\circ}\times M_{\chi}N}=0.             \label{eq:distributioninvariant}
\end{equation}
Assume that $c_{i}\neq e$. Let $\mu\in Hom_{H_{(w,v,i)}}(V_{\eta}\otimes M_{(w,v,i)}^{(r)},V_{t\chi})$. For each $t\in \mathbb{R}$ define
\[
h_{t}=\left[\begin{array}{ccc} \cosh t & & \sinh t \\ & I & \\ \sinh t & & \cosh t \end{array}\right].
\]
Observe that the $h_{t}$ define a one parameter subgroup of $M_{\chi}$. Observe also that
\[
ww_{v}c_{i} \left[\begin{array}{ccc} \cosh t & & \sinh t \\ & I & \\ \sinh t & & \cosh t \end{array}\right] c_{i}^{-1}w_{v}^{-1}w^{-1}= ww_{v} \left[\begin{array}{ccc} e^{-t} & &  \\ & I & \\  & & e^{t} \end{array}\right]w_{v}^{-1}w^{-1} \in A_{\circ}
\]
Since the action of $A_{\circ}$ on $V_{\eta}\otimes M^{(r)}_{(w,v,i)}$ is semi-simple, there exists a nonzero vector $v\in V_{\eta} \otimes  M^{(r)}_{(w,v,i)}$, such that
\[
 ww_{v} \left[\begin{array}{ccc} e^{-t} & &  \\ & I & \\  & & e^{t} \end{array}\right]w_{v}^{-1}w^{-1} \cdot v = e^{\alpha t} v
 \]
 for some $\alpha \in \mathbb{C}$. On the other hand
\begin{eqnarray*}
\left[\begin{array}{ccc} \cosh t & & \sinh t \\ & I & \\ \sinh t & & \cosh t \end{array}\right] \mu(v) & = & \mu(ww_{v} \left[\begin{array}{ccc} e^{-t} & &  \\ & I & \\  & & e^{t} \end{array}\right]w_{v}^{-1}w^{-1} \cdot v) \\
                                                                                                        & = & \mu(e^{\alpha t}v)=e^{\alpha t}\mu(v).
\end{eqnarray*}
which contradicts our assumption on $V_{\tau}$. Hence if $i\neq 0$
\[
Hom_{H_{(w,v,i)}}(V_{\eta}\otimes M_{(w,v,i)}^{(r)},V_{t\chi})=0 \qquad \forall r.
\]
Now assume that $i=0$.Then
\[
\O_{(w,v,0)}=P_{\circ}ww_{v}M_{\chi}N.
\]
Since $w\neq w_{M}$, there exists $1\leq i,j,\leq n$ such that $ww_{v}\left[\begin{array}{cc} I & xE_{i,i} \\ & I\end{array}\right]w_{v}^{-1}w^{-1}= \left[\begin{array}{cc} I & xE_{j,j} \\ & I\end{array}\right]\in N_{\circ}$. Now since $N_{\circ}$ acts locally unipotently on $V_{\eta}\otimes M_{(w,v,i)}^{(r)}$ there is a $v\in V_{\eta}\otimes M_{(w,v,i)}^{(r)}$ such that
\[
\left[\begin{array}{cc} I & xE_{j,j} \\ & I\end{array}\right]\cdot v=v+\tilde{v}
\]
with $\tilde{v} \in \operatorname{Ker} \mu$, and $\mu(v)\neq 0$. On the other hand
\begin{eqnarray*}
\mu(\left[\begin{array}{cc} I & xE_{j,j} \\ & I\end{array}\right] v) & = & \mu(ww_{v}\left[\begin{array}{cc} I & xE_{i,i} \\ & I\end{array}\right]w_{v}^{-1}w^{-1}) \\
\mu(v) & = & \chi(\left[\begin{array}{cc} I & xE_{i,i} \\ & I\end{array}\right])\mu(v)
\end{eqnarray*}
for all $x\in \mathbb{R}$, but this is only possible if $\mu(v)=0$ which is a contradiction. From all this we see that if $\O_{(w,v,i)}$ is not an open orbit, then
\[
D_{\O_{(w,v,i)}}(G;Hom(V_{\eta},V_{\tau}))^{P_{\circ}\times M_{\chi}N}=0.
\]

Now the standard Bruhat theoretic argument shows that $\tilde{\lambda} = 0$ which implies that $\lambda =0$.
\end{proof}

We are now ready to state the main result of this chapter.

\begin{theorem}
Let $(\sigma,V_{\sigma})$ be an admissible, smooth, Fr\'echet representation of $M$, and let $(\tau,V_{\tau})$ be an smooth, irreducible, tempered representation of $M_{\chi}$.
\begin{enumerate}
 \item The map
\[
\Phi_{\sigma,\nu}^{\tau}:Wh_{\chi,\tau}(I_{P,\sigma_{\nu}}^{\infty}) \longrightarrow Hom_{M_{\chi}}(V_{\sigma},V_{\tau})
\]
defines a linear isomorphism for all $\nu \in \mathfrak{a}'_{\mathbb{C}}$.
\item For all $\mu\in Hom_{M_{\chi}}(V_{\sigma},V_{\tau})$, the map $\nu\mapsto \mu\circ J_{\sigma,\nu}^{\chi}$ extends to a weakly holomorphic map of $\mathfrak{a}'_{\mathbb{C}}$ into $Hom(I^{\infty}_{\sigma},V_{\tau})$.
\end{enumerate}
\end{theorem}

Before starting the proof of this theorem, we will need some technical results that we will develop in the next section.

\section{Tensoring with finite dimensional representation}

Let $G$ be a simple Lie group of tube type, and let $P=MAN$ be a Siegel parabolic. Observe that $\dim A=1$, hence, we can choose $H\in Lie(A)$ such that, if $(\eta,F)$ is a finite dimensional representation of $G$, then there exists an integer $r$ such that $F=\oplus_{j=0}^{r}F_{2j-r}$ where
\[
 F_{j}=\{v \in F | H \cdot v = jv\}.
\]
Given such a finite dimensional representation $(\eta,F)$, set $X_{j}= \oplus_{k=j}^{r}F_{2k-r}$, then
\[
 F=X_{0}\supset X_{1}\supset \cdots \supset X_{r}\supset X_{r+1}=(0) \label{eq:filtration}
\]
is a $P$-invariant filtration. 
If we now set $Y^{j}=\{\phi \in F'|\phi|_{X_{j}=0}\}$, then
\[
 F'=Y^{r+1}\supset Y^{r}\supset \cdots \supset Y^{0}=(0)
\]
is the corresponding dual filtration.

We will identify $I_{P,\sigma_{\nu}}^{\infty}\otimes F$ with the space
\[
 \{\phi:G \longrightarrow V_{\sigma_{\nu}}\otimes F |\phi(pg)=(\sigma_{\nu}(p)\otimes I)\phi(g) \}
\]
and observe that
\[
 I_{P,\sigma_{\nu}\otimes\eta}^{\infty}=\{\phi:G \longrightarrow V_{\sigma_{\nu}}\otimes F |\phi(pg)=(\sigma_{\nu}(p)\otimes \eta(p))\phi(g)\}.
\]
With this conventions there is an isomorphism of $G$-modules
\begin{eqnarray*}
 I_{P,\sigma_{\nu}}^{\infty}\otimes F &\cong & I_{P,\sigma_{\nu}\otimes\eta}^{\infty}  \\
                                   \phi & \rightarrow & \hat{\phi}  \\
                                   \check{\phi} & \leftarrow & \phi,
\end{eqnarray*}
where
\[
 \hat{\phi}(g)=(I\otimes\eta(g))\phi(g), \qquad \mbox{and} \qquad \check{\phi}(g)=(I\otimes\eta(g)^{-1})\phi(g).
\]

Let $(\eta_{j},X_{j})$ be the restriction of $\eta$ to $P$ acting on $X_{j}$, and let $(\bar{\eta}_{j},X_{j}/X_{j+1})$ be the representation induced on the quotient. 

Then we have the following $G$-invariant filtration
\[
 I_{P,\sigma_{\nu}\otimes\eta}^{\infty}=I_{P,\sigma_{\nu}\otimes\eta_{0}}^{\infty}\supset\ldots\supset I_{P,\sigma_{\nu}\otimes\eta_{r+1}}^{\infty}=(0).
\]

Moreover, it can be checked that
\[
 I_{P,\sigma_{\nu}\otimes\eta_{j}}^{\infty}/I_{P,\sigma_{\nu}\otimes\eta_{j+1}}^{\infty} \cong I_{P,\sigma_{\nu}\otimes\bar{\eta}_{j}}^{\infty}.
\]

The next theorem is a restatement of some of the results given in \cite{w:hol} for the case at hand.

\begin{theorem}
Let $G$ be a Lie group of tube type, and let $P=MAN$ be a Siegel parabolic subgroup with given Langlands decomposition. Let $\chi$ be a generic character of $N$, and set $\mathfrak{g}=Lie(G)$. There exists an element $\Gamma \in U(\mathfrak{g})^{M_{\chi}}$, depending only on $F$, such that
\begin{enumerate}
 \item The map 
\[
\Gamma: Wh_{\chi}(I_{P,\sigma_{\nu}}^{\infty}) \otimes F' \longrightarrow   Wh_{\chi}(I_{P,\sigma_{\nu}}^{\infty} \otimes F)
\]
is an isomorphism.
\item If $\lambda\in Wh_{\chi}(I_{P,\sigma_{\nu}}^{\infty}) \otimes Y^{j}$, then $\Gamma(\lambda)=\lambda+\tilde{\lambda}$ with $\tilde{\lambda}\in (I_{P,\sigma_{\nu}}^{\infty})'\otimes Y^{j-1}$.
\end{enumerate}
\end{theorem}

Define
\[
 \check{\Gamma}:Wh_{\chi}(I_{P,\sigma_{\nu}}^{\infty}) \otimes F' \longrightarrow Wh_{\chi}(I_{P,\sigma_{\nu}\otimes \eta}^{\infty})
\]
by $\check{\Gamma}(\lambda)(\phi)=\Gamma(\lambda)(\check{\phi})$.

Then it's clear, from the above lemma, that $\check{\Gamma}$ defines an $M_{\chi}$-equivariant isomorphism.

 If $(\xi,V)$ is a representation of $P$, let
\begin{eqnarray*}
 I_{P,\xi}^{\infty} & = & \{\phi:G\longrightarrow V\, | \, \phi(pg)=\delta_{P}(p)^{\frac{1}{2}}\xi(p)\phi(g)\} \\
          U_{P,\xi} & = & \{\phi \in I_{P,\xi}^{\infty}\, | \, \operatorname{supp} \phi \subset P(w^{M})^{\ast}N\}.
\end{eqnarray*}
Observe that if $\phi\in U_{P,\xi}$ then $\phi$ has compact support modulo $P$.

If $(\tau,V_{\tau})$ is a representation of $M_{\chi}$, define
\[
 \Phi_{\xi}^{\tau}: Wh_{\chi,\tau}(I_{P,\xi}^{\infty}) \longrightarrow Hom_{M_{\chi}}(V,V_{\tau})
\]
in the same way we defined $\Phi_{\sigma,\nu}^{\tau}$.

Now observe that, since $\tilde{\Gamma}$ is an $M_{\chi}$-equivariant isomorphism, $\tilde{\Gamma}$ induces an isomorphism, which we wil also denote by $\tilde{\Gamma}$, 
\[
 \begin{array}{ccc}
  \tilde{\Gamma}: (Wh_{\chi}(I_{\sigma_{\nu}}\otimes F'\otimes V_{\tau}))^{M_{\chi}} & \longrightarrow & (Wh_{\chi}(I_{\sigma_{\nu}\otimes\eta})\otimes V_{\tau})^{M_{\chi}}  \\
  \mbox{ \begin{sideways}$ \cong $ \end{sideways}} & & \mbox{ \begin{sideways}$ \cong $ \end{sideways}}\\
Wh_{\chi,\eta'\otimes\tau}(I_{\sigma_{\nu}}) & \longrightarrow & Wh_{\chi,\tau}(I_{\sigma_{\nu}\otimes\eta})
 \end{array}
\]
Here we are using the fact that $Wh_{\chi,\tau}(V_{\pi}) \cong (Wh_{\chi}(V_{\pi}\otimes V_{\tau})^{M_{\chi}}$ for all representations $(\pi, V_{\pi})$ of $G$, and $(\tau,V_{\tau})$ of $M_{\chi}$.

We will identify
\[
 Wh_{\chi,\eta'\otimes \tau}(I_{\sigma_{\nu}}) \cong \left\lbrace\lambda:I_{\sigma_{\nu}}\otimes F \longrightarrow V_{\sigma}\, \left| \, \begin{array}{c} \lambda(\pi(n)\otimes I)\phi)=\chi(n) \lambda(\phi), \\ \lambda(\pi(m)\otimes \eta(m)\phi)=\tau(m)\lambda(\phi) \end{array} \right\rbrace\right. .
\]
Then we can identify $\Phi_{\sigma_{\nu}}^{\eta'\otimes\tau}$ with a map
\[
\Phi_{\sigma_{\nu}}^{\eta'\otimes\tau} : Wh_{\chi,\eta'\otimes \tau}(I_{\sigma_{\nu}}) \longrightarrow Hom_{M_{\chi}}(V_{\sigma}\otimes F, V_{\tau}),
\]
such that, if $\lambda \in Wh_{\chi,\eta'\otimes\tau}(I_{\sigma_{\nu}})$, $\phi\in U_{\sigma_{\nu}\otimes\eta}$, and we set $\mu_{\lambda}=\Phi_{\sigma_{\nu}}^{\eta'\otimes\tau}(\lambda)$, then
\[
 \lambda(\phi)=\mu_{\lambda}(\int_{N} \chi(n)^{-1}\phi(w_{M}n)\, dn).
\]
Let $\nu \in \mathfrak{a}_{\mathbb{C}}'$ be such that $\Phi_{\sigma_{\nu}}^{\tilde{\tau}}$
is an isomorphism for every representation $(\tilde{\tau},V_{\tilde{\tau}})$ of $M_{\chi}$. Then $\Phi_{\sigma_{\nu}}^{\eta'\otimes\tau}$ is an isomorphism, and if we set $\tilde{\Gamma}=\Phi_{\sigma_{\nu}\otimes\eta}^{\tau} \circ \check{\Gamma} \circ (\Phi_{\sigma_{\nu}}^{\eta'\otimes\tau})^{-1}$, then the following diagram commutes
\begin{equation}\label{diag1}
 \begin{picture}(250,155)%(-50,0)
 \put(-10,10){$Hom_{M_{\chi}}(V_{\sigma}\otimes F, V_{\tau})$}
 %\put(170,65){$(V_{\sigma}\otimes F)'$}
\put(0,120){$Wh_{\chi,\eta'\otimes\tau}(I_{\sigma_{\nu}}^{\infty})$}
\put(170,120){$Wh_{\chi,\tau}(I_{\sigma_{\nu}\otimes\eta}^{\infty})$}
\put(160,10){$Hom_{M_{\chi}}(V_{\sigma}\otimes F, V_{\tau})$}
\put(195,113){\vector(0,-1){90}}
%\put(195,58){\vector(0,-1){35}}
\put(75,125){\vector(1,0){85}}
\put(95,15){\vector(1,0){60}}
\put(40,113){\vector(0,-1){90}}
\put(120,130){$\check{\Gamma}$}
\put(120,20){$\tilde{\Gamma}$}
\put(45,65){$\Phi_{\sigma_{\nu}}^{\eta'\otimes\tau}$}
\put(200,65){$\Phi_{\sigma_{\nu}\otimes\eta}^{\tau}$}
\end{picture}
\end{equation}

\begin{lemma}
 $\tilde{\Gamma}$ is an isomorphism.
\end{lemma}

Observe that this lemma immediately implies that $\Phi_{\sigma_{\nu}\otimes\eta}^{\tau}$ is an isomorphism. 

\begin{proof}
Let $\mu\in Hom_{M_{\chi}}(V_{\sigma}\otimes F, V_{\tau})$. Since $\Phi_{\sigma_{\nu}}^{\eta'\otimes\tau}$ is an isomorphism, there exists $\lambda \in Wh_{\chi,\eta'\otimes\tau}(I_{\sigma_{\nu}}^{\infty})$ such that $\mu=\Phi_{\sigma_{\nu}}^{\eta'\otimes\tau}(\lambda)=:\mu_{\lambda}$, i.e., if $\check{\phi}\in U_{P,\sigma_{\nu}\otimes\eta}$, then
\[
 \lambda(\phi)=\mu_{\lambda}(\int_{N}\chi(n)^{-1}\phi(w_{M}n)\, dn ).
\]

To prove the lemma, we will show that if $\mu_{\lambda}\in Hom_{M_{\chi}}(V_{\sigma}\otimes F, V_{\tau})$, is such that $\mu_{\lambda}(v)=0$ for all $v\in V_{\sigma}\otimes X_{j}$, then 
\[
(\mu_{\lambda}-(I\otimes\eta(w_{M})')\tilde{\Gamma}(\mu_{\lambda}))(v)=0 \qquad \forall v\in V_{\sigma}\otimes X_{j-1}
\]
For any such $\mu_{\lambda}$, $\lambda \in (Wh_{\chi}(I_{\sigma_{\nu}}^{\infty})\otimes Y^{j} \otimes V_{\tau})^{M_{\chi}}$, and hence $\Gamma(\lambda)=\lambda + \tilde{\lambda}$, with $\tilde{\lambda}\in (I_{P,\sigma_{\nu}}^{\infty})'\otimes Y^{j-1}\otimes V_{\tau}$. Therefore, if $\phi\in I_{\sigma_{\nu}}^{\infty} \otimes X_{j-1}$, then $\tilde{\lambda}(\phi)=0$ and
\[
 \Gamma(\lambda)(\phi)=\lambda(\phi)+\tilde{\lambda}(\phi)=\lambda(\phi).
\]
Now by definition
\[
 \check{\Gamma}(\lambda)(\hat{\phi})=\Gamma(\lambda)(\phi)=\mu_{\lambda}(\int_{N}\chi(n)^{-1}\phi(w_{M}n)\, dn ).
\]
But on the other hand since $\hat{\phi}\in U_{P,\sigma_{\nu}\otimes\eta}$
\begin{eqnarray*}
 \check{\Gamma}(\lambda)(\hat{\phi})& = &\mu_{\check{\Gamma}(\lambda)}(\int_{N}\chi(n)^{-1}\hat{\phi}(w_{M}n)\, dn ) \\
                                 & = & \tilde{\Gamma}(\mu_{\lambda})(\int_{N}\chi(n)^{-1}(I\otimes\eta(w_{M}n))\phi(w_{M}n)\, dn ). 
\end{eqnarray*}
Given $v_{j-1}\in V_{\sigma}\otimes X_{j-1}$ choose a $\phi\in I^{\infty}_{\sigma_{\nu}}\otimes X_{j-1}$ such that $\phi(w_{M})=v_{j-1}$. Let $\{u^{k}\}_{k}$ be an approximate identity on $N$, and define $\phi^{k}\in U_{P,\sigma_{\nu}}$ by
\[
 \phi^{k}(w_{M}n)=u^{k}(n)\chi(n)\phi(w_{M}n).
\]
Then
\begin{eqnarray*}
 \lim_{k\rightarrow \infty} \mu_{\lambda}(\int_{N}\chi(n)^{-1}\phi^{k}(w_{M}n)\, dn ) & = & \lim_{k\rightarrow \infty} \tilde{\Gamma}(\mu_{\lambda})(\int_{N}\chi(n)^{-1}I\otimes\eta(w_{M}))\phi^{k}(w_{M})\, dn ) \\
 \mu_{\lambda}(\lim_{k\rightarrow \infty} \int_{N}u^{k}(n)\phi(w_{M}n)\, dn ) & = & \tilde{\Gamma}(\mu_{\lambda})(\lim_{k\rightarrow \infty} \int_{N}u^{k}(n)I\otimes\eta(w_{M}))\phi(w_{M})\, dn ) \\
\mu_{\lambda}(\phi(w_{M}))& = & \tilde{\Gamma}(\mu_{\lambda})(I\otimes\eta(w_{M}))\phi(w_{M})) \\
\mu_{\lambda}(v_{j-1}) & = & [(I\otimes\eta(w_{M}^{-1})')\tilde{\Gamma}(\mu_{\lambda})](v_{j-1}). 
\end{eqnarray*}
Since this holds for all $v_{j-1}\in V_{\sigma}\otimes X_{j-1}$, we conclude that $\mu_{\lambda}-(I\otimes\eta(w_{M}^{-1})')\tilde{\Gamma}(\mu_{\lambda}) \in V_{\sigma}' \otimes Y^{j-1}\otimes V_{\tau}$ as we wanted to prove.
\end{proof}

We will now choose a representation $(\eta,F)$, such that the action of $M$ on $F_{r}$ is trivial, then $\sigma_{\nu}\otimes\bar{\eta}_{0} \cong \sigma_{\nu-r}$, and hence
\[
 I_{\sigma_{\nu}\otimes\bar{\eta}_{0}}^{\infty} \cong I_{\sigma_{\nu-r}}^{\infty}.
\]
Let
\[
 W^{j}=\{\lambda \in Wh_{\chi,\tau}(I_{\sigma_{\nu}\otimes\eta}^{\infty}) | \lambda|_{I_{\sigma_{\nu}\otimes \eta_{j}}^{\infty}}=0\}.
\]
Observe that if $\lambda \in W^{j+1},$ then $ \lambda|_{I_{\sigma_{\nu}\otimes \eta_{j}}^{\infty}}$ defines an element in $Wh_{\chi,\tau}(I_{\sigma_{\nu}\otimes \bar{\eta}_{j}}^{\infty})$. 

\begin{lemma}
There exists and isomorphism
\[
 S:Wh_{\chi,\tau}(I_{\sigma_{\nu}\otimes\eta}^{\infty})\longrightarrow \bigoplus_{j=0}^{r} W^{j+1}|_{I_{\sigma_{\nu}\otimes\eta_{j}}^{\infty}}
\]
such that the following diagram is commutative:
\begin{equation}\label{diag2}
 \begin{picture}(300,220)
\put(10,200){$Wh_{\chi,\tau}(I_{\sigma_{\nu}\otimes\eta}^{\infty})$}
\put(180,200){$\oplus_{j=0}^{r}W^{j+1}|_{I_{\sigma_{\nu}\otimes\eta_{j}}^{\infty}}$}
\put(180,140){$\oplus_{j=0}^{r}Wh_{\chi,\tau}(I_{P,\sigma_{\nu}\otimes \bar{\eta}_{j}}^{\infty})$}
\put(180,80){$\oplus_{j=0}^{r} Hom_{M_{\chi}}(V_{\sigma}\otimes X_{j}/X_{j+1}, V_{\tau})$}
\put(180,20){$Hom_{M_{\chi}}(V_{\sigma}\otimes F, V_{\tau})$}
\put(85,203){\vector(1,0){90}}
\put(210,190){\vector(0,-1){35}}
\put(210,130){\vector(0,-1){35}}
\put(210,70){\vector(0,-1){35}}
\put(40,190){\vector(1,-1){155}}
\put(130,210){$S$}
\put(120,120){$\Phi_{\sigma_{\nu}\otimes\eta}^{\tau}$}
\put(220,110){$\Phi_{\sigma_{\nu}\otimes\bar{\eta}_{j}}^{\tau}$}
\put(220,50){$X_{j}/X_{j+1}\cong F_{r-2j}$}
\end{picture}
\end{equation}
\end{lemma}

\begin{corollary}
 $\Phi_{\sigma_{\nu}}^{\tau}$ is an isomorphism for all $\nu \in \mathfrak{a}_{\mathbb{C}}'$
\end{corollary}

\begin{proof}[Proof (of corollary)]
Let $\nu \in \mathfrak{a}_{\mathbb{C}}'$ be such that $\Phi_{\sigma_{\nu}}^{\tau}$ is an isomorphism. Then we know that $\Phi_{\sigma_{\nu}\otimes \eta}^{\tau}$ is an isomorphism, and from the above diagram $\Phi_{\sigma_{\nu}\otimes\bar{\eta}_{j}}$ is an isomorphism for all $j$. In particular, if $(\eta,F)$ is as before, then  $\Phi_{\sigma_{\nu-r}}^{\tau}$ is an isomorphism. Proceeding by induction, it can now be shown that $\Phi_{\sigma_{\nu}}^{\tau}$ is an isomorphism for all $\nu \in \mathfrak{a}_{\mathbb{C}}'$. 
\end{proof}

\begin{proof}[Proof (of lemma)]
Let $\lambda\in Wh_{\chi,\tau}(I_{\sigma_{\nu}\otimes\eta}^{\infty})$. Since $\Phi_{\sigma_{\nu}\otimes\eta}^{\tau}$ is an isomorphism, there exists $\mu_{\lambda}\in Hom_{M_{\chi}}(V_{\sigma}\otimes F, V_{\tau})$ such that if $\phi\in U_{\sigma_{\nu}\otimes\eta}$, then
\[
 \lambda(\phi)=\mu_{\lambda}(\int_{N}\chi(n)^{-1}\phi(w_{M}n)\, dn).
\]
Let $p_{j}$ be the natural projection of $V_{\sigma}\otimes F$ onto $V_{\sigma}\otimes F_{r-2j}$. Since $p_{j}$ is $M_{\chi}$-equivariant, $\mu_{\lambda}\circ p_{j}\in Hom_{M_{\chi}}(V_{\sigma}\otimes F, V_{\tau})$. Let $\lambda_{j}=(\Phi_{\sigma_{\nu}\otimes\eta}^{\tau})^{-1}(\mu_{\lambda}\circ p_{j})$. Then it is clear that $\lambda=\sum \lambda_{j}$. Furthermore if $\phi \in U_{\sigma_{\nu}\otimes \eta_{j+1}}$, then
\[
 \lambda_{j}(\phi)=\mu_{\lambda}\circ p_{j}(\int_{N}\chi(n)^{-1}\phi(w_{M}n)\, dn)=0,
\]
 since all the values of $\phi$ are in $X_{j+1}$. Now since $\lambda_{j}|_{U_{\sigma_{\nu}\otimes \eta_{j+1}}}=0$ proposition \ref{prop:vanishinggeneral} implies that $\lambda_{j}|_{I_{\sigma_{\nu}\otimes\eta_{j+1}}^{\infty}}=0$, i.e., $\lambda_{j}\in W^{j+1}$. Define
\[
 S(\lambda)=(\tilde{\lambda}_{0},\ldots,\tilde{\lambda}_{r}), \qquad \mbox{with $\tilde{\lambda}_{j}=\lambda_{j}|_{I_{\sigma_{\nu}\otimes\eta_{j+1}}^{\infty}}$}.
\]
It's then clear, from the above observations, that $S$ is an isomorphism that makes diagram \ref{diag2} commute.
\end{proof}

\section{Holomorphic Continuation of Certain Jacquet Integrals} \label{sec:holomorphiccontjacquet}
\begin{theorem}
Let $(\sigma,V_{\sigma})$ be an admissible, smooth, Fr\'echet representation of $M$, and let $(\tau,V_{\tau})$ be an smooth, irreducible, tempered representation of $M_{\chi}$.
\begin{enumerate}
 \item The map
\[
\Phi_{\sigma,\nu}^{\tau}:Wh_{\chi,\tau}(I_{P,\sigma_{\nu}}^{\infty}) \longrightarrow Hom_{M_{\chi}}(V_{\sigma},V_{\tau})
\]
defines a linear isomorphism for all $\nu \in \mathfrak{a}'_{\mathbb{C}}$.
\item For all $\mu\in Hom_{M_{\chi}}(V_{\sigma},V_{\tau})$, the map $\nu\mapsto \mu\circ J_{\sigma,\nu}^{\chi}$ extends to a weakly holomorphic map of $\mathfrak{a}'_{\mathbb{C}}$ into $Hom(I^{\infty}_{\sigma},V_{\tau})$.
\end{enumerate}
\end{theorem}
\begin{proof}
We have already seen that $\Phi_{\sigma,\nu}^{\tau}$ is an isomorphism for all $\nu \in \mathfrak{a}'_{\mathbb{C}}$. Given  $\mu\in Hom_{M_{\chi}}(V_{\sigma}, V_{\tau})$, and $\phi\in I_{\sigma}^{\infty}$ define
\[
 \gamma_{\mu}^{\tau}(\nu)(\phi):=(\Phi_{\sigma,\nu}^{\tau})^{-1}(\mu)(\phi_{\sigma,\nu}).
\]
Observe that, if $\operatorname{Re} \nu \gg 0$, then $\gamma_{\mu}(\nu)(\phi)=\mu\circ J_{\sigma,\nu}^{\chi}(\phi)$. We will show that $\gamma$ has holomorphic continuation to all $\nu \in \mathfrak{a}_{\mathbb{C}}'$, by showing that it satisfies a shift equation.

Let $\nu\in \mathfrak{a}'_{\mathbb{C}}$ and $\phi\in I^{\infty}_{\sigma}$ be arbitrary. By definition
\[
 \gamma_{\mu}(\nu-r)(\phi)=(\Phi_{\sigma,\nu-r}^{\tau})^{-1}(\mu)(\phi_{\sigma,\nu-r})=\lambda(\phi_{\sigma,\nu-r})
\]
for some $\lambda \in Wh_{\chi,\tau}(I_{\sigma,\nu-r}^{\infty})$. But now, according to (\ref{diag1}) and (\ref{diag2}), there exists $\delta\in Wh_{\chi,\eta'\otimes\tau}(I_{\sigma_{\nu}})$, and $\psi\in I_{\sigma_{\nu}\otimes \eta}^{\infty}$, such that
\begin{eqnarray*}
 \gamma_{\mu}(\nu-r)(\phi)& = & \lambda(\phi_{\sigma,\nu-r})=\check{\Gamma}(\delta)(\psi) \\
 & = & \delta(\Gamma^{T}\check{\psi}) = \gamma_{\mu_{\delta}}^{\eta'\otimes\tau}(\nu)(\Gamma^{T}\check{\psi}),
\end{eqnarray*}
where $\mu_{\delta}=\Phi_{\sigma_{\nu}}^{\eta'\otimes\tau}(\delta).$
%Let $\{v_{j}\}_{j=1}^{m}$ be a basis of $F,$ and let $\{l_{j}\}_{j=1}^{m}$ be its dual basis. Then we can find $\delta_{j}\in Wh_{\chi,\tau}(I_{\sigma_{\nu}}^{\infty}),$ $j=1,\ldots,m,$ and $\psi \in I_{\sigma_{\nu}\otimes\eta}^{\infty}$ such that
%\begin{eqnarray*}
% \gamma_{\mu}(\nu-r)(\phi)& = &\lambda(\phi_{\sigma,\nu-r})=\check{\Gamma}(\sum \delta_{j}\otimes l_{j})(\psi) \\
%                       & = & \Gamma(\sum \delta_{j} \otimes l_{j})(\check{\psi})  \\
%                       & = & (\sum \delta_{j} \otimes l_{j})(\Gamma^{T}\check{\psi}). 
%\end{eqnarray*}
% Now since $\Gamma^{T}\check{\psi} \in I_{\sigma,\nu}^{\infty}\otimes F,$ we can find $\phi_{j}\in I^{\infty}_{\sigma},$ $j=1,\ldots,m,$ such that
%\[
% \Gamma^{T}\check{\psi}=\sum (\phi_{j})_{\sigma_{\nu}} \otimes v_{j}.
%\]
%Hence
%\[
%  \gamma_{\mu}(\nu-r)(\phi)= \sum \delta_{j}((\phi_{j})_{P,\sigma_{\nu}})=\sum \gamma_{\delta_{j}}(\nu)(\phi_{j}).
%\]
This is the desired shift equation which shows that $\gamma_{\mu}$ is weakly holomorphic everywhere.
\end{proof}

\chapter{The Bessel-Plancherel theorem} \label{chapter:Bessel-Plancherel}

\section{The Bessel-Plancherel theorem for rank 1 Lie groups of tube type}

Let $G$ be a Lie group of tube type and let $P=MAN$ be a Siegel parabolic subgroup with given Langlands decomposition. Given a character $\chi$ of $N$, we will set $C_{c}^{\infty}(N\backslash G ; \chi)$ to be equal to the set of smooth functions $f:G \longrightarrow \mathbb{C}$ such that $f(ng)=\chi(n)f(g)$, for all $n\in N$, $g\in G$, and such that $|f|$ has compact support modulo $N$. We will also set 
\[
 L^2(N\backslash G;\chi)=\left\lbrace f:G \longrightarrow \mathbb{C} \, \left| \,\begin{array}{c}
				\mbox{$f(ng)=\chi(n)f(g)$, for $n\in N$, $g\in G$, } \\
                                \mbox{ and $\int_{N\backslash G} |f(g)|^2 \, dNg < \infty$}
                              \end{array}  \right\rbrace\right. .
\]
Here $dNg$ is the measure on $N\backslash G$ such that, if $f\in C_{c}^{\infty}(G)$, then
\begin{equation}
 \int_{G} f(g) \, dg = \int_{N\backslash G}\int_{N}f(ng) \, dn \, dNg. \label{eq:quotient}
\end{equation}
Observe that $L^2(N\backslash G;\chi)$ has a natural inner product, defined by
\[
 \langle f, h \rangle := \int_{N\backslash G} \overline{f(g)} h(g) \,dNg, \qquad \mbox{for $f,h \in L^2(N\backslash G;\chi)$.}
\]
Let $\lambda$ be the measure on $\hat{N}$ such that, if $f\in C_{c}^{\infty}(N)$, then
\begin{equation}
 f(e)=\int_{\hat{N}} \hat{f}(\chi) d\lambda(\chi), \label{eq:PlancheremeasureN}
\end{equation}
where $\hat{f}$ is the Fourier transform of $f$.

 Given $f\in C_{c}^{\infty}(G)\subset L^{2}(G)$, define 
$s_{f}(\chi) \in C_{c}^{\infty}(N\backslash G ; \chi)\subset L^{2}(N\backslash G;\chi)$, by
\begin{equation}
 s_{f}(\chi)(g)= \int_{N} \chi(n)^{-1}f(ng) \, dn. \label{eq:Fouriertransform}
\end{equation}
We would like to use this equation to identify $f$ with a ``section'' $s_{f}$ on a ``vector bundle'' $E$ over $\hat{N}$ with fibers $L^{2}(N\backslash G;\chi)$. Unfortunately, this ``bundle'' fails to satisfy the local triviality property. We will work around this problem in the following way: Given $f\in C_{c}^{\infty}(N\backslash G;\chi)$ and $p \in P$ set $L_{p}f(g)=f(p^{-1}g)$. Then 
\[
L_{p}f(ng)=f(p^{-1}ng)=f((p^{-1}np)p^{-1}g)=\chi(p^{-1}np)f(p^{-1}g)=(p\cdot\chi)(n)L_{p}f(g),
\]
i.e, $L_{p}f \in C_{c}^{\infty}(N\backslash G ; p \cdot \chi)$. Moreover, according to section $\ref{sec:preliminary}$,
\begin{eqnarray*}
 \langle L_{p}f, L_{p}f \rangle & = & \int_{N\backslash G} \overline{L_{p}f(g)}L_{p}f(g) \, dNg = \int_{N\backslash G} \overline{f(p^{-1}g)}f(p^{-1}g) \, dNg\\
 & = &  \int_{N\backslash G} \overline{f(g)}f(g)\delta_{P}(p)^{-1} \, dNg = \delta_{P}(p)^{-1} \langle f, f \rangle,
\end{eqnarray*}
where $\delta_{P}$ is the modular function of $P$. What this couple of equations say is that we can extend $L_{p}$ to a conformal transformation from $L^{2}(N\backslash G;\chi)$ to $L^{2}(N\backslash G;p\cdot \chi)$. Now from equation (\ref{eq:Fouriertransform})
\begin{eqnarray}
 s_{R_{g_{1}}f}(\chi)(g)&=&\int_{N}\chi(n)^{-1}R_{g_{1}}f(ng) \, dn=\int_{N}\chi(n)^{-1}f(ngg_{1}) \, dn \nonumber \\ &=& s_{f}(\chi)(gg_{1})=:(R_{g_{1}}s_{f}(\chi))(g) \label{eq:leftbundleaction}
\end{eqnarray}
and
\begin{eqnarray}
 s_{L_{p}f}(\chi)(g)& = & \int_{N}\chi(n)^{-1}L_{p}f(ng) \, dn=\int_{N}\chi(n)^{-1}f(p^{-1}npp^{-1}g) \, dn \nonumber \\
                    & = & \int_{N}\chi(pnp^{-1})^{-1}\delta_{P}(p)f(np^{-1}g) \, dn \nonumber \\ 
&=& \int_{N}(p^{-1}\chi)(n)^{-1}\delta_{P}(p)f(np^{-1}g) \, dn  \nonumber \\
                    & = & \delta_{P}(p)s_{f}(p^{-1}\chi)(p^{-1}g). \label{eq:rightbundleaction}
\end{eqnarray}
Let $\Omega$ be the set of open orbits for the action of $P$ on $\hat{N}$. Then we can use equations (\ref{eq:leftbundleaction}) and (\ref{eq:rightbundleaction}) to define a $P\times G$-vector bundle
\[
 \begin{array}{c}
  E \\ \downarrow \\ \omega,
 \end{array}
\]
with fibers $E_{\chi}=L^{2}(N\backslash G;\chi)$. Now observe that, if $f\in C_{c}^{\infty}(G)$, then, according to equations (\ref{eq:quotient}), (\ref{eq:PlancheremeasureN}) and (\ref{eq:Fouriertransform}),
\begin{eqnarray}
 \langle f, f \rangle & = & \int_{N\backslash G} \int_{N} \overline{f(ng)}f(ng) \, dn \, dNg \nonumber \\
 & = & \int_{N\backslash G} \int_{\hat{N}} \overline{s_{f}(\chi)(g)}s_{f}(\chi)(g) \,d\lambda(\chi) \, dNg \nonumber \\
& = & \int_{\hat{N}} \langle s_{f}(\chi) , s_{f}(\chi) \rangle d\lambda(\chi) \nonumber \\
&  =  &\bigoplus_{\omega \in \Omega} \int_{\omega} \langle s_{f}(\chi) , s_{f}(\chi) \rangle d\lambda(\chi), \label{eq:isometry}
\end{eqnarray}
 where the last equality follows from the fact that the complement of the union of the open orbits of $P$ in  $\hat{N}$ has measure zero. Therefore, if we set
\[
 L^{2}(\omega,E, \lambda)=\{s:\omega\longrightarrow E \, | \, \mbox{$s(\chi)\in E_{\chi}$, and $\int_{\omega} \|s(\chi)\|^{2}\, d\lambda(\chi) < \infty$}\},
\]
then the map $f\mapsto s_{f}$ extends to a $P\times G$-equivariant isometry between $L^{2}(G)$ and $\oplus_{\omega \in \Omega}L^{2}(\omega,E, \lambda)$. This is the identification we were after when we defined $s_{f}$.

We will now use the material developed in section \ref{sec:representationparabolics} to continue the study of the decomposition of $L^{2}(G)$ with respect to the action of $P\times G$. For every $\omega \in \Omega$, we will fix a character $\chi_{\omega}\in \omega$. Then, according to the material in section \ref{sec:representationparabolics}, there is a natural $P\times G$-equivariant isomorphism
\[
 \bigoplus_{\omega \in \Omega}L^{2}(\omega,E, \lambda) \cong \bigoplus_{\omega \in \Omega} \operatorname{Ind}_{M_{\chi}N\times G}^{P\times G}L^{2}(N\backslash G ; \chi_{\omega}).
\]
Therefore
\begin{eqnarray}
 L^{2}(G)  & \cong & \bigoplus_{\omega \in \Omega}L^{2}(\omega,E, \lambda)  
           \cong \bigoplus_{\omega \in \Omega} \operatorname{Ind}_{M_{\chi}N\times G}^{P\times G}L^{2}(N\backslash G ; \chi_{\omega}) \nonumber\\
           & \cong & \bigoplus_{\omega \in \Omega} \int_{\hat{M}_{\chi_{\omega}}}\operatorname{Ind}_{M_{\chi_{\omega}}N}^{P}(\overline{\chi}_{\omega}\otimes \tau^{\ast})\otimes L^2(M_{\chi_{\omega}}N \backslash G ;\tau\otimes\chi_{\omega}) d\eta(\tau) \nonumber\\
           & \cong & \bigoplus_{\omega \in \Omega} \int_{\hat{M}_{\chi_{\omega}}} \operatorname{Ind}_{M_{\chi_{\omega}}N}^{P}(\overline{\chi}_{\omega}\otimes \tau^{\ast})\otimes \int_{\hat{G}} W_{\chi_{\omega},\tau}(\pi)\otimes \pi\, d\mu_{\omega,\tau}(\pi)d\eta(\tau)\nonumber \\
	  & \cong & \bigoplus_{\omega \in \Omega} \int_{\hat{M}_{\chi_{\omega}}}\int_{\hat{G}} W_{\chi_{\omega},\tau}(\pi)\otimes \operatorname{Ind}_{M_{\chi_{\omega}}N}^{P}(\overline{\chi}_{\omega}\otimes \tau^{\ast})\otimes   \pi\, d\mu_{\omega,\tau}(\pi)d\eta(\tau). \label{eq:PtimesGgeometrical}
\end{eqnarray}
Here $\eta$ is the Plancherel measure of $M_{\chi}$, whereas $\mu_{\omega,\tau}$ is a measure on $\hat{G}$ that depends on $\omega$ and $\tau$, and $W_{\chi_{\omega},\tau}(\pi)$ is some multiplicity space, that also depends on $\omega$ and $\tau$. On the other hand, Harish-Chandra's Plancherel theorem says that
\begin{equation}
 L^2(G)\cong \int_{\hat{G}} \pi^{\ast}|_{P} \otimes \pi \,d\mu(\pi), \label{eq:Plancherelrestricted}
\end{equation}
where $\mu$ is the Plancherel measure on $G$. 

We will now restrict to the case where $G$ has rank 1. In this case if $\chi$ is generic, then $M_{\chi}$ is compact, and hence equation (\ref{eq:PtimesGgeometrical}) reads
\[
 L^{2}(G) = \bigoplus_{\omega \in \Omega} \bigoplus_{\hat{M}_{\chi_{\omega}}}\int_{\hat{G}} W_{\chi_{\omega},\tau}(\pi)\otimes \operatorname{Ind}_{M_{\chi_{\omega}}N}^{P}(\overline{\chi}_{\omega}\otimes \tau^{\ast})\otimes   \pi\, d\mu_{\omega,\tau}(\pi).
\]
From this and equation (\ref{eq:Plancherelrestricted}) we conclude that, for $\mu$-almost all $\pi$,
\begin{equation}
 \pi^{\ast}|_{P}  \cong  \bigoplus_{\omega \in \Omega} \bigoplus_{\tau \in \hat{M}_{\chi_{\omega}}} W_{\chi_{\omega},\tau}(\pi)\otimes \operatorname{Ind}_{M_{\chi_{\omega}}N}^{P}\overline{\chi}_{\omega}\otimes \tau^{\ast},         \label{eq:rank1restrictiontoparabolic}
\end{equation}
and that $\mu_{\omega,\tau}$ is absolutely continuous with respect to the Plancherel measure $\mu$, for all $\omega$, $\tau$, i.e.,
\begin{equation}
 L^2(M_{\chi}N \backslash G ;\tau\otimes\chi) \cong \ \int_{\hat{G}} W_{\chi,\tau}(\pi) \otimes \pi \,d\mu(\pi). \label{eq:rank1besselplancherel}
\end{equation}
Observe that in the rank 1 case $P$ is both, a maximal and a minimal parabolic subgroup, and hence this case is contained in the calculation of the Whittaker-Plancherel measure given in \cite{w:vol2}. There the multiplicity spaces $W_{\chi,\tau}(\pi)$ are explicitly computed to be isomorphic to the space of Whittaker models $Wh_{\chi,\tau}(\pi)$.

Looking at equations (\ref{eq:PtimesGgeometrical}), (\ref{eq:Plancherelrestricted}), (\ref{eq:rank1restrictiontoparabolic}) and (\ref{eq:rank1besselplancherel}) it is natural to state the following conjecture:
\begin{conjecture} \label{conj:mainconjecture}
 For $\mu$-almost all $\pi$,
\begin{equation}
 \pi^{\ast}|_{P}  \cong  \bigoplus_{\omega \in \Omega} \int_{\tau \in \hat{M}_{\chi_{\omega}}} W_{\chi_{\omega},\tau}(\pi)\otimes \operatorname{Ind}_{M_{\chi_{\omega}}N}^{P}\overline{\chi}_{\omega}\otimes \tau^{\ast}d\eta(\tau).         \label{eq:restrictiontoparabolic}
\end{equation}
Furthermore, if $\chi$ is a generic character of $N$,  then
\begin{equation}
 L^2(N \backslash G;\chi) \cong \ \int_{\hat{G}}\int_{\hat{M}_{\chi}}  W_{\chi,\tau}(\pi)\otimes \tau^{\ast} \otimes \pi \,d\eta(\tau)d\mu(\pi), \label{eq:besselplancherel}
\end{equation}
where $\eta$ is the Plancherel measure of $M_{\chi}$, and $\mu$ is the Plancherel measure of $G$.
\end{conjecture}

Proving this conjecture will be the goal of this chapter. We will start in section \ref{sec:asymptotic} where we will restate the asymptotic expansion of certain matrix coefficients developed by Wallach \cite{w:vol1}, \cite{w:vol2}. Then, in section \ref{sec:Schwartz}, we will define the Schwartz space, $\C(N\backslash G;\chi)$, of $L^{2}(N\backslash G;\chi)$, analogous to the Schwartz space, $\C(G)$, of $L^{2}(G)$, and define a ``Fourier transform'' map between $\C(G)$ and $\C(N\backslash G;\chi)$. In section \ref{sec:besselplancherel} we will use the results developed in the previous two sections to prove equations (\ref{eq:restrictiontoparabolic}) and (\ref{eq:besselplancherel}). This is the \emph{generalized Bessel-Plancherel theorem}, generalized in the sense that the group $M_{\chi}$ may not be compact.

In the case where $M_{\chi}$ is compact, we can give a more explicit description of the isomorphism appearing in equation (\ref{eq:besselplancherel}). In section \ref{sec:fouriertransform} we will use the holomorphic continuation of the generalized Jacquet integrals proved in section \ref{sec:holomorphiccontjacquet} to calculate the Fourier transform of a wave packet, in the case where $M_{\chi}$ is compact. This calculations will allow us to ``push forward'' the decomposition of $\C(G)$ to obtain a decomposition of $\C(N\backslash G;\chi)$ with explicit intertwiner operators. Finally in section \ref{sec:explicitbesselplancherel} we will use this results, together with the asymptotic expansions developed in section \ref{sec:asymptotic}, to identify the multiplicity spaces $W_{\chi,\tau}(\pi)$ with the space of Bessel models $Wh_{\chi,\tau}(\pi)$.

The chapter closes with three appendixes that, although not necessary for the main results of this chapter, are related to the material discussed here. In appendix \ref{sec:representationparabolics} we discuss the basic representation theory of parabolic subgroups. Although this results are known from the work of Wolf et.\ al.\ \cite{lw:parabolic},  it is convenient to have them here to set the notation, and for easy reference. Appendix 
\ref{sec:plancherelparabolic} discusses the Plancherel measure of the space $L^{2}(P)$. Although not directly related to the Bessel-Plancherel measure, some of the ideas discussed there are used through this chapter. Finally in appendix \ref{sec:preliminary} it is proved that the support of the Whittaker-Plancherel measure with respect to any parabolic subgroup is contained in the tempered spectrum. This result may be useful in future calculations of other Whittaker-Plancherel measures. At least they give us hope that this Whittaker-Plancherel measures may not be ``too awful''.

\section{The asymptotic expansion of certain matrix coefficients}\label{sec:asymptotic}

Let $(\pi,H)$ be an admissible, finitely generated, Hilbert representation of $G$. Let $V$ be the space of smooth vectors of $H$. Given $v\in V$ and $\lambda \in V'$ we can define a smooth function $c_{\lambda,v}(g)=\lambda(\pi(g)v)$ for $g \in G$. This function is called a matrix coefficient function. In this section we will describe some asymptotic expansions of certain matrix coefficients, $c_{\lambda,v}$, where $\lambda$ satisfies certain properties. The exposition of this sections follows very closely the material developed in \cite{w:vol1} and \cite{w:vol2}. I decided to put it here as it is convenient to develop this results together, but if the reader so desires he can also look at the original exposition by looking at the references in the appropriate places.

Let's start by specifying what we mean by an asymptotic expansion.

\begin{dfn}
 By a \emph{formal exponential polynomial power series} we will mean a formal sum of the form
\begin{equation}
 \sum_{1\leq j\leq r} e^{z_{j}t}\sum_{n\geq 0}p_{j,n}(t)e^{-nt} \label{eq:formalpowerseries}
\end{equation}
where $p_{j,n}$ is a polynomial in $t$ for each $j$, $n$.
\end{dfn}
The point here is that we do not care if the series converges. Fix such a formal series, then we may rearrange it in the following way:
\begin{equation}
 \sum_{j\geq 1} e^{u_{j}t}p_{u_{j}}(t),
\end{equation}
with $u_{j}\in\{z_{k}-n\, | \,\mbox{ $1 < k < r$, $n > 0$, $n\in \mathbb{N}$}\}$, $\operatorname{Re} u_{1} \geq \operatorname{Re} u_{2}, \ldots$, and $p_{u_{j}}$ is the sum of the $p_{k,n}$ with $z_{k}-n=u_{j}$. We will call $N$ a gap of the series if $u_{N} > u_{N+1}$.

If $f$ is a function on $\mathbb{R}$, we say that $f$ is asymptotic as $t\rightarrow +\infty$ to the formal exponential polynomial power series given as in (\ref{eq:formalpowerseries}) if, for each gap $N$, there exists constants $C$ and $\epsilon$, depending on $N$, such that
\[
 |f(t)-\sum_{j\leq N}e^{u_{j}t}p_{u_{j}}(t)| \leq Ce^{(\operatorname{Re} u_{N}-\epsilon)t} \qquad \mbox{for $t\geq 1$}.
\]
Notice that if $N$ is a gap then
\[
 \lim_{t\rightarrow +\infty} e^{-(\operatorname{Re} u_{N})t}|f(t)-\sum_{j\leq N}e^{u_{j}t}p_{u_{j}}(t)|=0.
\]
\begin{lemma}
 Let
\[
 \sum_{1\leq j\leq r} e^{z_{j}t}\sum_{n\geq 0}p_{j,n}(t)e^{-nt}
\]
and
\[
 \sum_{1\leq j\leq s} e^{w_{j}t}\sum_{n\geq 0}q_{j,n}(t)e^{-nt}
\]
be two formal exponential polynomial series such that $z_{j}-z_{k}$, (respectively $w_{j}-w_{k}$) is not an integer for $j\neq k$ and $p_{j,0}\neq 0$, $q_{j,0}\neq 0$. If both formal exponential polynomial series are asymptotic to the same function $f$, then $r=s$, and after relabeling $w_{j}=z_{j}$, $p_{j,n}=q_{j,n}$.
\end{lemma}

This is precisely lemma 4.A.1.2 in \cite{w:vol1}.

\begin{dfn}
 Let $(\pi,H)$ be as above, and let $(P_{\circ},A_{\circ})$ be a minimal $p$-pair for $G$. We say that $\lambda \in V'$ is tame with respect to $(P_{\circ},A_{\circ})$ if there exists $\delta\in \mathfrak{a}_{\circ}'$ such that
\[
 |X\lambda(\pi(a)v)|\leq C_{X,v}a^{\delta},
\]
for all $X\in U(\mathfrak{g})$, $v\in V$ and $a\in Cl(A_{\circ}^{+})$.
\end{dfn}
The important point in this definition is that $\delta$ doesn't depend on the element $X\in U(\mathfrak{g})$. Observe that if $\lambda \in (\check{H})^{\infty}$, then $\lambda$ is tame for all minimal $p$-pairs $(P_{\circ},A_{\circ})$. The following proposition provides more examples of tame linear functionals.

\begin{proposition}
 Let $(\pi,H)$ be an admissible, finitely generated, Hilbert representation of $G$. Let $P=NAM$ be the Siegel parabolic described in section \ref{sec:classification}, and let $\chi$ be a character of $N$ whose stabilizer, $M_{\chi}$, is compact. If $\lambda \in Wh_{\chi}(V)(\check{\tau})$ (the $\check{\tau}$-isotypic component) for some $\tau \in \hat{M}_{\chi}$, then $\lambda$ is tame for every minimal $p$-pair $(P_{\circ},A_{\circ})$ such that $A\subset A_{\circ}$.
\end{proposition}

\begin{proof}
 By the Gelfand-Neimark decomposition $\mathfrak{g}=\bar{\mathfrak{n}}\oplus\mathfrak{m}\oplus\mathfrak{a}\oplus\mathfrak{n}$. On the other hand, according to the Iwasawa decomposition, $\mathfrak{m}=\mathfrak{m}_{\chi}\oplus\mathfrak{a}_{M}\oplus\mathfrak{n}_{M}$,  where $\mathfrak{m}_{\chi}=Lie(M_{\chi})$, and $\mathfrak{a}_{M}\oplus\mathfrak{n}_{M}\subset Lie(P_{\circ})$. Therefore
$\mathfrak{g}=\bar{\mathfrak{n}}\oplus(\mathfrak{m}_{\chi}\oplus\mathfrak{a}_{M}\oplus\mathfrak{n}_{M})\oplus\mathfrak{a}\oplus \mathfrak{n}$, and hence
\begin{equation}
 U(\mathfrak{g})=U(\mathfrak{a}\oplus\mathfrak{a}_{M}\oplus\mathfrak{n}_{M})U(\bar{\mathfrak{n}})U(\mathfrak{m}_{\chi}\oplus\mathfrak{n}). \label{eq:universalenveloping}
\end{equation}
The plan for the proof of this proposition is the following: we will show that if $\lambda \in Wh_{\chi}(V)(\check{\tau})$, then $U(\mathfrak{m}_{\chi}\oplus\mathfrak{n})\lambda \subset Wh_{\chi}(V)(\check{\tau})$, and there exists $\delta \in \mathfrak{a}_{\circ}'$ such that $|X\lambda (a\cdot v)| \leq C_{X,v}a^{\delta}$, for all $a\in Cl(A_{\circ}^{+})$. We will also show that if $\lambda \in Wh_{\chi}(V)(\check{\tau})$ satisfies that $|\lambda (a\cdot v)| \leq C_{v}a^{\delta}$ for all $a\in Cl(A_{\circ}^{+})$, and $X\in U(\bar{\mathfrak{n}})$, then $|X \lambda (a\cdot v) |\leq C_{X,v}a^{\delta}$, for all $a\in Cl(A_{\circ}^{+})$. Observe that in this case $X\lambda$ may no longer be in $Wh_{\chi}(V)(\check{\tau})$. Finally, if $|\lambda (a\cdot v)| \leq C_{v}a^{\delta}$ for all $a\in Cl(A_{\circ}^{+})$, and $X\in U(\mathfrak{a}\oplus\mathfrak{a}_{M}\oplus\mathfrak{n}_{M})$, we will show that $|X \lambda (a\cdot v) |\leq C_{X,v}a^{\delta}$, for all $a\in Cl(A_{\circ}^{+})$. It is then clear that all this statements, together with equation (\ref{eq:universalenveloping}), are enough to prove the proposition.
 
Given $\lambda \in Wh_{\chi}(V)(\check{\tau})$, set $W_{\lambda}=U(\mathfrak{m}_{\chi}\oplus \mathfrak{n})\lambda$ and observe that this is a finite dimensional subspace of $Wh_{\chi}(V)(\check{\tau})$. Let $\{\lambda_{1},\ldots,\lambda_{k}\}$ be a basis of $W_{\lambda}$. Then for any given $v\in V$, and every $1\leq i \leq k$, there exists $\delta_{i}\in\mathfrak{a}_{\circ}^{\prime}$, and a constant $C_{i}$, such that
\[
 |\lambda_{i}(a\cdot v)|\leq C_{i}a^{\delta_{i}} \qquad \mbox{for all $a\in Cl(A_{\circ}^{+}).$}
\]
 Let $\delta \in \mathfrak{a}_{\circ}^{\prime}$ be such that $a^{\delta_{i}} \leq a^{\delta}$, for all $1\leq i \leq k$, $a\in Cl(A_{\circ}^{+})$. If $X\in  U(\mathfrak{m}_{\chi}\oplus \mathfrak{n})$, then there exists some constants $b_{i}(X)$ such that
\begin{equation}
 |X\lambda(\pi(a)v)|=|\sum_{i=1}^{k}b_{i}(X)\lambda_{i}(\pi(a)v)|\leq \sum_{i=1}^{k}|b_{i}(X)|C_{i}a^{\delta_{i}} \leq Ca^{\delta} \label{eq:mchin}
\end{equation}
for some constant $C$.

Let $\lambda\in Wh_{\chi}(V)(\check{\tau})$ be such that $|\lambda (a\cdot v) |\leq C_{v}a^{\delta}$ for all $a\in Cl(A_{\circ}^{+})$. Since $A\subset A_{\circ}$, we have that $\bar{\mathfrak{n}}=\bar{\mathfrak{n}}\cap\mathfrak{n}_{\circ} \oplus \bar{\mathfrak{n}}\cap\bar{\mathfrak{n}}_{\circ}$. Let $X\in (\bar{\mathfrak{n}}\cap\mathfrak{n}_{\circ})_{\alpha}$ for some root $\alpha \in \Phi(P_{\circ},A_{\circ})$. Then
\begin{eqnarray}
 |X\lambda(a \cdot v)| & = & |\lambda(X^{T}a \cdot v)|=|\lambda(a Ad(a^{-1})(X^{T}) \cdot v)|=|\lambda(a (a^{-\alpha})X^{T} \cdot v)|\nonumber \\
& \leq &  C_{X,v}a^{\delta-\alpha}\leq C_{X,v}a^{\delta}. \label{eq:barn}
\end{eqnarray}
Now let $X\in (\bar{\mathfrak{n}}\cap\bar{\mathfrak{n}}_{\circ})_{-\beta}$ for some root $\beta \in \Phi(P_{\circ},A_{\circ})$. In this case we can always find $Y\in (\mathfrak{n}\cap\mathfrak{n}_{\circ})_{\beta}$ such that $\chi(Y)\neq 0$. For such a $Y$
\begin{eqnarray}
 |Y\lambda(a \cdot v)|& = & |\lambda(Y^{T}a \cdot v)|=|\lambda(a Ad(a^{-1})(Y^{T}) \cdot v)| \nonumber \\ 
|\chi(Y)| |\lambda(a\cdot v)|& =  & |\lambda(a(a^{-\beta})Y^{T}\cdot v)|\leq C_{Y,v}a^{\delta-\beta} \nonumber \\
|\lambda (a\cdot v)| &\leq & C_{Y,v}^{\prime}a^{\delta-\beta}, \label{eq:estimatebarndirection}
\end{eqnarray}
where in the last step we are using that $\chi(Y)\neq 0$. Using this improved estimate we see that
\begin{eqnarray}
 |X\lambda(a \cdot v)|& = & |\lambda(X^{T}a \cdot v)|=|\lambda(a Ad(a^{-1})(X^{T}) \cdot v)| \nonumber \\ 
  & =  & |\lambda(a(a^{\beta})X^{T}\cdot v)|\leq C_{X,v}^{\prime}a^{\delta-\beta + \beta} \leq  C_{X,v}^{\prime}a^{\delta}. \label{eq:barnnormalncirc}
\end{eqnarray}
 Since $\bar{\mathfrak{n}}$ is a direct sum of spaces of the form $(\bar{\mathfrak{n}}\cap\mathfrak{n}_{\circ})_{\alpha}$ and $(\bar{\mathfrak{n}}\cap\bar{\mathfrak{n}}_{\circ})_{-\beta}$ with $\alpha, \beta \in \Phi(P_{\circ},A_{\circ})$, from (\ref{eq:barn}) and (\ref{eq:barnnormalncirc}) we conclude that, if $X\in U(\bar{\mathfrak{n}})$, then $|X\lambda(a\cdot v)|\leq C_{X,v} a^{\delta}$, for all $a\in Cl(A_{\circ}^{+})$.

Now assume that $\lambda$ is a linear functional such that $|\lambda (a\cdot v) |\leq C_{v}a^{\delta}$ for all $a\in Cl(A_{\circ}^{+})$, and let $X\in U(\mathfrak{a}_{\circ})=U(\mathfrak{a}\oplus \mathfrak{a}_{M})$. Then
\begin{equation}
 |X \lambda(\pi(a)v)|=|\lambda(\pi(X^{T})\pi(a)v)|=|\lambda(\pi(a)\pi(X^{T})v)|\leq C_{X,v}a^{\delta}. \label{eq:estimateacirc}
\end{equation}
Finally, let $X\in (\mathfrak{n}_{M})_{\alpha}$ for some root $\alpha \in \Phi(P_{\circ},A_{\circ})$. Then
\begin{eqnarray}
 |X\lambda(a \cdot v)| & = & |\lambda(X^{T}a \cdot v)|=|\lambda(a Ad(a^{-1})(X^{T}) \cdot v)|=|\lambda(a(a^{-\alpha})X^{T}\cdot  v)|\nonumber \\
& \leq & C_{X,v}a^{\delta-\alpha}\leq C_{X,v}a^{\delta}.\label{eq:estimatenM}
\end{eqnarray}
Since $\mathfrak{n}_{M}$ is a direct sum of its weight spaces, from equations (\ref{eq:estimateacirc}) and (\ref{eq:estimatenM}) we conclude that $|X\lambda(a \cdot v)| \leq  C_{X,v}a^{\delta}$ for all $X\in U(\mathfrak{a}\oplus \mathfrak{a}_{M}\oplus \mathfrak{n}_{M})$. We have thus proved all the statements that we made at the beginning of the proof, so we are done.
\end{proof}

Let $K\subset G$ be a maximal compact subgroup. If $(\pi, V)$ is as above, we will denote by $V_{K}$ the space of $K$-finite vectors of $V$. Let $F_{1}$ be a subset of $\Phi(P_{0},A_{0})$, $(P_{1},A_{1})$ be the corresponding standard $p$-pair, and let $\bar{P}_{1}=M_{1}A_{1}\bar{N_{1}}$ be the parabolic subgroup opposite to $P_{1}$. Set 
$$
E(P_{1},V)=\{\mu\in (\mathfrak{a}_{1})_{\mathbb{C}}' \, | \, (V_{K}/\bar{\mathfrak{n}}_{1}V_{K})_{\mu}\neq 0 \},
$$
where $(V_{K}/\bar{\mathfrak{n}}_{1}V_{K})_{\mu}$ is the generalized $\mu$-weight space of $V_{K}/\bar{\mathfrak{n}}_{1}V_{K}$, i.e., there exists $d\geq 1$ such that for all $H\in \mathfrak{a}_{1}$
\[
 (H-\mu(H))^{d}(V_{K}/\bar{\mathfrak{n}}_{1}V_{K})_{\mu}=0.
\]
Now assume that $\Phi(P_{0},A_{0})=\{\alpha_{1},\ldots, \alpha_{r}\}$. Define $H_{1},\ldots H_{r}\in \mathfrak{a}_{0}$ by $\alpha_{i}(H_{j})=\delta_{i,j}$. Then we can define $\Lambda_{V}\in \mathfrak{a}_{0}'$ by
\[
 \Lambda_{V}(H_{j})=\max\{\operatorname{Re} \mu(H_{j})\, | \, \mu \in E(P_{0},V)\}.
\]

The following theorem provides an asymptotic expansion for the matrix coefficient function $c_{\lambda,v}$, where $v\in V$ and $\lambda$ is a tame linear functional, and it's essentially a combination of theorems 15.2.4 and 15.2.5 in \cite{w:vol2}.

\begin{theorem}\label{thm:asymptoticexpansion}
  Let $(\pi,H)$ be an admissible, finitely generated, Hilbert representation of $G$, and let $(P_{1},A_{1})$ be a standard $p$-pair with respect to the minimal $p$-pair $(P_{\circ},A_{\circ})$. If $\lambda \in V'$ is tame with respect to $(P_{\circ},A_{\circ})$ then
\begin{description}
 \item[i)] There exists $d \geq 0$ such that if $v\in V$ then for all $a\in Cl(A_{0}^{+})$
\[
 |\lambda(a\cdot v)| \leq (1+\log \|a\|)^{d}a^{\Lambda_{V}}\sigma_{\lambda_{V}}(v),
\]
for some continuous seminorm, $\sigma_{\lambda_{V}}$, of $V$.
\item[ii)] If $H\in \mathfrak{a_{1}}^{+}$, $m\in M_{1}$ and $v\in V$ then
\[
 \lambda(\exp(tH)m\cdot v)\sim \sum_{\mu\in E(P_{1},V)}e^{t\mu(H)}\sum_{Q\in L_{1}^{+}}e^{-tQ(H)}p_{\lambda,\mu,Q}(tH,m,v)
\]
as $t\rightarrow \infty$, where $p_{\lambda,\mu,Q}:\mathfrak{a_{1}}\times M_{1}\times V\rightarrow \mathbb{C}$ is a function that is polynomial in $\mathfrak{a_{1}}$, real analytic in $M_{1}$ and linear in $V$, and 
$$
L_{1}^{+}=\{\sum_{j}n_{j}\alpha_{j}\, | \, \mbox{$\alpha_{j}\in \Phi(P_{1},A_{1})$ and $n_{j}$ is a nonnegative integer}\}.
$$
\end{description}
 \end{theorem}

The proof of this result is complicated and first we will need to prove the same result, but for $v \in V_{K}$ instead of $v\in V$. Namely we need to first prove the following lemma (which is equivalent to theorem 15.2.2 in \cite{w:vol2}). 

\begin{lemma}\label{lemma:asymptoticexpansion}
  Let $(\pi,H)$ be an admissible, finitely generated, Hilbert representation of $G$, and let $(P_{1},A_{1})$ be a standard $p$-pair with respect to the minimal $p$-pair $(P_{\circ},A_{\circ})$. If $\lambda \in (V)'$ is tame with respect to $(P_{\circ},A_{\circ})$ then
\begin{description}
 \item[i)] There exists $d \geq 0$ such that if $v\in V_{K}$ then for all $a\in Cl(A_{0}^{+})$
\[
 |\lambda(a\cdot v)| \leq (1+\log \|a\|)^{d}a^{\Lambda_{V}}\sigma_{\lambda_{V}}(v),
\]
for some continuous seminorm, $\sigma_{\lambda_{V}}$, of $V$.
\item[ii)] If $H\in \mathfrak{a_{1}}^{+}$, $m\in M_{1}$ and $v\in V_{K}$ then
\[
 \lambda(\exp(tH)m\cdot v)\sim \sum_{\mu\in E(P_{1},V)}e^{t\mu(H)}\sum_{Q\in L_{1}^{+}}e^{-tQ(H)}p_{\lambda,\mu,Q}(tH,m,v)
\]
as $t\rightarrow \infty$, where $p_{\lambda,\mu,Q}:\mathfrak{a_{1}}\times M_{1}\times V_{K}\rightarrow \mathbb{C}$ is a function that is polynomial in $\mathfrak{a_{1}}$, real analytic in $M_{1}$ and linear in $V_{K}$, and 
$$
L_{1}^{+}=\{\sum_{j}n_{j}\alpha_{j}\, | \, \mbox{$\alpha_{j}\in \Phi(P_{1},A_{1})$ and $n_{j}$ is a nonnegative integer}\}.
$$
\end{description}
 \end{lemma}
\begin{proof}
 We will start with the proof of i.

Since $\lambda$ is tame with respect to $(P_{\circ},A_{\circ})$, there exists $\delta \in \mathfrak{a}'$ such that
\begin{equation}
 |\lambda(a\cdot v)|\leq a^{\delta}\sigma_{\lambda}(v) \label{eq:initialestimate}
\end{equation}
for some seminorm $\sigma_{\lambda}$ of $V$. Let $\Delta(P_{\circ},A_{\circ})=\{\alpha_{1},\ldots\alpha_{l}\}$ and choose elements $H_{j}\in \mathfrak{a_{0}}$ such that $\alpha_{i}(H_{j})=\delta_{ij}$. Then $\Lambda=\sum \Lambda_{i}\alpha_{i}$ and $\delta=\sum \delta_{i}\alpha_{i}$, with $\Lambda_{i}=\Lambda(H_{i})$, $\delta_{i}=\delta(H_{i})$. The idea of the proof is to show that if $\delta_{i} >\Lambda_{i}$ then we can replace $\delta_{i}$ with $\Lambda_{i}$ at the cost of possibly changing the seminorm $\sigma_{\lambda}$ and adding a polynomial term.

Fix $\alpha_{i}\in \Delta(P_{\circ},A_{\circ})$ and set $F_{i}=\Delta(P_{\circ},A_{\circ})-\{\alpha_{i}\}$ and $P_{i}=P_{F_{i}}$. Then $\mathfrak{a_{i}}=\mathbb{R}H_{i}$ and any given $a\in CL(A^{+}_{0})$ can be expressed uniquely as $a=a_{t}\tilde{a}$, with $a_{t}=\exp(tH_{i})$, $t\geq 0$, and $\tilde{a}=\exp(\sum c_{j}H_{j})$, with $c_{j}\geq 0$, $c_{i}=0$. We will now make use of the $K$-finiteness of $v$. Let $q_{i}$ be the canonical projection of $V_{K}$ onto $V_{K}/\bar{\mathfrak{n}}_{i}V_{K}$. We claim that
\begin{equation}
 \mbox{if $q_{i}(v)=0$,}\qquad \mbox{then \qquad $|\lambda(a \cdot v)|\leq a^{\delta-\alpha_{i}}\sigma_{\lambda}'(v)$, } \label{eq:nV}
\end{equation}
for some seminorm $\sigma_{\lambda}'(v)$. Effectively, let $\bar{X}_{1},\ldots,\bar{X}_{p}$ be a basis of $\bar{\mathfrak{n}}_{i}$ consisting of root vectors with corresponding roots $\beta_{1},\ldots,\beta_{p}$. If $q_{i}(v)=0$, then $v=\sum \bar{X}_{j} v_{j}$ and hence
\begin{eqnarray*}
 |\lambda(a\cdot v)|&=&|\sum \lambda(a\bar{X_{j}} v_{j})| \\
&\leq&\sum |\lambda(Ad(a)(\bar{X}_{j})a\cdot v_{j})| \\
&\leq&\sum a^{\beta_{j}}|\bar{X}_{j}^{T}\lambda(a\cdot v_{j})| \\
&\leq&\sum a^{\delta+\beta_{j}}\sigma_{\lambda,\bar{X}_{j}}(v_{j}) \\
&\leq& a^{\delta-\alpha_{i}}\sigma_{\lambda}'(v).
\end{eqnarray*}
Here we are using the fact that if $\beta_{j}$ is a root of $\bar{\mathfrak{n}}_{i}$ and $a\in Cl(A^{+})$ then $a^{\beta_{j}}\leq a^{-\alpha_{i}}$.

Given $v\in V_{K}$, choose vectors $v=v_{1},\ldots,v_{r}\in V_{K}$, such that $\{q(v_{1}),\ldots, q(v_{r})\}$ is a basis of $U(\mathfrak{a_{i}})q_{i}(v)$. Then
\begin{equation}
 H_{i}v_{j}=\sum_{k}b_{jk}v_{k}+w_{j} \label{eq:matrixB}
\end{equation}
with $w_{j}\in \bar{\mathfrak{n}}_{i}V_{K}$. Given $\tilde{a}\in Cl(A_{0}^{+})$ such that $\tilde{a}^{\alpha_{i}}=1$, set 
\[
 F(t,\tilde{a},v)=\left[\begin{array}{c} \lambda(a_{t}\tilde{a}\cdot v_{1}) \\ \vdots \\\lambda(a_{t}\tilde{a}\cdot v_{r})\end{array}\right]
\]
and
\[
 G(t,\tilde{a},v)=\left[\begin{array}{c} \lambda(a_{t}\tilde{a}\cdot w_{1}) \\ \vdots \\\lambda(a_{t}\tilde{a}\cdot w_{r})\end{array}\right].
\]
Then equation (\ref{eq:matrixB}) says that
\[
 \frac{d}{dt}F(t,\tilde{a},v)=BF(t,\tilde{a},v)+G(t,\tilde{a},v),
\]
where $B=[b_{jk}]$. Solving this differential equation explicitly we get that
\[
 F(t,\tilde{a},v)=e^{tB}F(0,\tilde{a},v)+e^{tB}\int_{0}^{t}e^{-sB}G(s,\tilde{a},v)\, ds.
\]
Observe that, since $v=v_{1}$,
\begin{equation}
 |\lambda(a_{t}\tilde{a}\cdot v)|\leq \|F(t,\tilde{a},v)\| \leq \|e^{tB}F(0,\tilde{a},v)\|+\|e^{tB}\int_{0}^{t}e^{-sB}G(s,\tilde{a},v)\|, \label{eq:unidirinequality}
\end{equation}
here we are using the usual norm in $\mathbb{C}^{r}$. We will now estimate the two summands in the right hand side of (\ref{eq:unidirinequality}).

By equation (\ref{eq:initialestimate})
\begin{equation}
 \|F(0,\tilde{a},v)\|\leq \tilde{a}^{\delta}\sigma_{\lambda}(v). \label{eq:estimateF}
\end{equation}
On the other hand by the definition of $\Lambda=\Lambda_{V}$
\begin{equation}
 \|e^{tB}\|\leq (1+t)^{d_{i}}e^{t\max\{\mu(H_{i})\, |\, \mu \in E(P_{i},V)\}}=(1+t)^{d_{i}}e^{t\Lambda_{i}}. \label{eq:estimateB}
\end{equation}
Using equations (\ref{eq:estimateF}) and (\ref{eq:estimateB}) we can estimate the first summand in the right hand side of (\ref{eq:unidirinequality}), namely
\begin{equation}
 \|e^{tB}F(0,\tilde{a},v)\|\leq(1+t)^{d_{i}}e^{t\Lambda_{i}}\tilde{a}^{\delta}\sigma_{\lambda}(v). \label{eq:firstsummandestimate}
\end{equation}
To estimate the second summand of (\ref{eq:unidirinequality}) observe that by equation (\ref{eq:nV})
\begin{equation}
 \|G(t,\tilde{a},v)\|\leq \tilde{a}^{\delta}e^{t(\delta_{i}-1)}\sigma_{\lambda}''(v). \label{eq:estimateG}
\end{equation}
Breaking $\mathbb{C}^{r}$ as a direct sum of the invariant subspaces of $B$, and using (\ref{eq:estimateG}), it can be shown that
\begin{equation}
\|e^{tB}\int_{0}^{t}e^{-sB}G(s,\tilde{a},v)\, ds\|\leq (1+t)^{d_{i}'}\tilde{a}^{\delta}e^{t(\delta_{i}-1)}\sigma_{\lambda}''(v) + C_{i}(1+t)^{d_{i}}e^{t\Lambda_{i}}\tilde{a}^{\delta}\sigma_{\lambda}''(v). \label{eq:secondsummandestimate}  
\end{equation}
Finally, from (\ref{eq:unidirinequality}), (\ref{eq:firstsummandestimate}) and (\ref{eq:secondsummandestimate}) we obtain the following bound
\begin{equation}
 |\lambda(a_{t}\tilde{a}\cdot v)|\leq(1+t)^{d_{i}''}e^{t\max\{\Lambda_{i},\delta_{i}-1\}}\tilde{a}^{\delta}\sigma_{\lambda}'''(v). \label{eq:recurringbound}
\end{equation}

If $\Lambda_{i}\leq \delta_{i}-1$ we can use the fact that $(1+t)^{d_{i}''}\leq Ce^{t/2}$, for some constant $C$, to get equation (\ref{eq:initialestimate}) again but for a new linear functional $\delta$ gotten from replacing $\delta_{i}$ by $\delta_{i}-1/2$. We can then repeat this same argument a finite number of times until we get $\Lambda_{i}>\delta_{i}-1$. Observe that in the last step $e^{t\Lambda_{i}}$ dominates $e^{t(\delta_{i}-1)}$, so, again from equations (\ref{eq:unidirinequality}), (\ref{eq:firstsummandestimate}) and (\ref{eq:secondsummandestimate}), we get
\begin{equation}
 |\lambda(a_{t}\tilde{a}\cdot v)|\leq (1+t)^{d_{i}}e^{t\Lambda_{i}}\tilde{a}^{\delta}\tilde{\sigma}_{\lambda}(v) \label{eq:finalestimateidirection}
\end{equation}
where $d_{i}$ is as in equation (\ref{eq:estimateB}). Observe that this $d_{i}$ is independent of $v$.

To finish the proof of i, observe that if $a\in Cl(A_{0}^{+})$ then $a=\exp(\sum t_{j}H_{j})$, for some $t_{j}\geq 0$. If we repeat the argument leading up to equation (\ref{eq:finalestimateidirection}) for all $i=1,\ldots,l$, we get that
\[
  |\lambda(a\cdot v)|\leq \bar{\sigma}_{\lambda}(v)\prod_{i=1}^{l}(1+t_{i})^{d_{i}}e^{t_{i}\Lambda_{i}}\leq \bar{\sigma}_{\lambda}'(v)(1+\log \|a\|)^{d}a^{\Lambda}
\]
for some $d\geq 0$ and some seminorm $\bar{\sigma}_{\lambda}'$.

We will now start the proof of ii. We will only prove this result for $P_{1}=P_{0}$ the remaining cases being a consequence of this result. Fix a tame linear functional $\lambda$. Let
\[
 \mathfrak{a}_{\epsilon}^{+}=\{H \in \mathfrak{a_{0}}^{+} \,|\, \mbox{$\|H\|=1$ and $\alpha(H)\geq \epsilon$ for all $\alpha\in \Phi(P_{\circ},A_{\circ})$}\},
\]
and set $L=\max \{\Lambda(H) \,|\, \mbox{$H\in \mathfrak{a_{0}}^{+}$, $\|H\|=1$}\}$. Let $q_{k}$ be the canonical projection of $V_{K}$ onto $V_{k}/\bar{\mathfrak{n}}_{0}^{k}V_{K}$. Arguing as in the proof of (\ref{eq:nV}) we can show that 
\begin{equation}
\begin{array}{c}
 \mbox{if $H\in \mathfrak{a}_{\epsilon}^{+}$ and $q_{k}(v)=0$, then}\\ \mbox{$|\lambda(\exp tH\cdot v)| \leq (1+t)^{d}e^{t(L-k\epsilon)}\sigma_{\lambda}(v)$ for all $t >0$}, \label{eq:nkV}
\end{array}
\end{equation}
where $\sigma_{\lambda}$ is a continuous seminorm on $V$. 

Given $v\in V_{K}$, choose vectors $v=v_{1},\ldots,v_{r}\in V_{K}$, such that $\{q_{k}(v_{1}),\ldots, q_{k}(v_{r})\}$ is a basis of $U(\mathfrak{a_{0}})q_{k}(v)$. Then, as in the proof of part i, we have that
\begin{equation}
 Hv_{j}=\sum_{k}b_{jk}(H)v_{k}+w_{j} \label{eq:matrixBH}
\end{equation}
with $w_{j}\in \bar{\mathfrak{n}}_{0}^{k}V_{K}$. Hence, if we set 
\[
 F(v)=\left[\begin{array}{c} \lambda( v_{1}) \\ \vdots \\ \lambda(v_{r})\end{array}\right]
\]
and
\[
 G(v)=\left[\begin{array}{c} \lambda(w_{1}) \\ \vdots \\\lambda(w_{r})\end{array}\right],
\]
then, from equation (\ref{eq:matrixBH}), we get that
\[
 \frac{d}{dt}F(\exp tH\cdot v)=B(H)F(\exp tH\cdot v)+G(\exp tH\cdot v),
\]
where $B(H)=[b_{jk}(H)]$. Solving this differential equation explicitly we get that
\begin{equation}
 F(\exp tH\cdot v)=e^{tB(H)}F(v)+e^{tB(H)}\int_{0}^{t}e^{-sB(H)}G(\exp sH\cdot v)\, ds. \label{eq:solutiondiffequation}
\end{equation}
We will use this equation to derive an asymptotic expansion for $\lambda(\exp tH\cdot v)$ for every $H\in \mathfrak{a}_{\epsilon}^{+}$.

 Given $H\in \mathfrak{a}_{\epsilon}^{+}$, define $E^{k}_{\epsilon,H}(P,V)=\{\mu\in E^{k}(P,V) \, | \, \mu(H) > L-k\epsilon\}$. Observe that, since $E^{k}_{\epsilon,H}(P,V)$ is finite, there exists $\delta >0$ such that $\mu(H)>L-k\epsilon+\delta$, for all $\mu\in E^{k}_{\epsilon,H}(P,V)$. Let $P^{k}_{\epsilon,H}$ be the projection of $\mathbb{C}^{r}$ onto the generalized eigenspaces of $B(H)$ with eigenvalues of the form $\mu(H)$, for $\mu\in E^{k}_{\epsilon,H}(P,V)$. Set $Q^{k}_{\epsilon,H}=I-P^{k}_{\epsilon,H}$, where $I$ is the indentity map on $\mathbb{C}^{r}$. Starting with equation (\ref{eq:solutiondiffequation}), and arguing as in the proof of i, we can show that
\begin{equation}
 \|Q^{k}_{\epsilon,H}F(\exp tH \cdot v) \| \leq (1+t)^{2d}e^{t(L-k\epsilon)}\sigma_{\lambda}(v).
\end{equation}
On the other hand, from the fact that $\|e^{-sB(H)}P^{k}_{\epsilon,H}\|\leq C(1+s)^{d}e^{s(k\epsilon-L-\delta)}$, and statement (\ref{eq:nkV}), we get that
\begin{eqnarray}
 \|e^{-sB(H)}P^{k}_{\epsilon,H}G(\exp sH \cdot v) \| & \leq & (1+s)^{2d}e^{s(k\epsilon-L-\delta)}e^{s(L-k\epsilon)}\sigma_{\lambda}(v) \nonumber \\
& = & (1+s)^{2d}e^{-s\delta}\sigma_{\lambda}(v). \label{eq:eminussbGsvlessthandelta}
\end{eqnarray}
From the above equation, the integral
\[
 \int_{0}^{\infty}\|e^{-sB(H)}P^{k}_{\epsilon,H}G(\exp sH \cdot v)\|\, ds <\infty.
\]
Set
\begin{equation}
 F^{k}_{\epsilon,H}(v)=P^{k}_{\epsilon,H}F(v)+\int_{0}^{\infty}e^{-sB(H)}P^{k}_{\epsilon,H}G(\exp sH \cdot v)\, ds. \label{eq:defFkepsilonH}
\end{equation}
Then the above estimates imply that
\begin{equation}
 \|F(\exp tH \cdot v)-e^{tB(H)}F^{k}_{\epsilon,H}(v)\|\leq (1+t)^{2d}e^{t(L-k\epsilon)}\sigma_{\lambda}(v). \label{eq:asymptoticestimate}
\end{equation}
Using again that $\|e^{-tB(H)}P^{k}_{\epsilon,H}\|\leq C(1+t)^{d}e^{t(k\epsilon-L-\delta)}$, we obtain that 
\begin{equation}
 \lim_{t\rightarrow \infty}e^{-tB(H)}P^{k}_{\epsilon,H}F(\exp tH \cdot v)=F^{k}_{\epsilon,H}(v).  \label{eq:fkepsilonH}
\end{equation}
From this equation we can get the following identity
\begin{eqnarray}
 F^{k}_{\epsilon,H}(v) & = & \lim_{t\rightarrow \infty}e^{-(s+t)B(H)}P^{k}_{\epsilon,H}F(\exp (s+t)H \cdot v) \nonumber \\
 & = & e^{-sB(H)}\lim_{t\rightarrow \infty}e^{-tB(H)}P^{k}_{\epsilon,H}F(\exp tH \cdot [\exp sH \cdot v]) \nonumber \\
 & = & e^{-sB(H)}F^{k}_{\epsilon,H}(\exp sH \cdot v), \label{eq:identity}
\end{eqnarray}
or equivalently, $F^{k}_{\epsilon,H}(\exp sH \cdot v)=e^{sB(H)}F^{k}_{\epsilon,H}(v)$.

Set $f^{k}_{\epsilon,H}(t,v)$ equal to the first component of $F^{k}_{\epsilon,H}(\exp tH \cdot v)$. Then
\begin{equation}
 f^{k}_{\epsilon,H}(t,v)=\sum_{ \stackrel{\mu\in E(P,V),\,\, Q\in L^{+}}{\scriptscriptstyle \operatorname{Re} (\mu-Q)(H)>L-k\epsilon + \delta}} e^{t(\mu-Q)(H)}p^{k}_{\epsilon,H,\mu,Q}(t,v) \label{eq:deffkepsilonHmuQ}
\end{equation}
for some polynomials $p^{k}_{\epsilon,H,\mu,Q}(t,v)$. Observe that, as long as $k\epsilon>L + \delta+\operatorname{Re}(Q-\mu)(H)$, this polynomials are independent of $k$, so, if we let $k\rightarrow \infty$, we can use equations (\ref{eq:asymptoticestimate}) and (\ref{eq:deffkepsilonHmuQ}) to define an asymptotic expansion for $\lambda(\exp tH\cdot v)$. Since asymptotic expansions are unique, the polynomials appearing in the expansion are independent of the $\epsilon > 0$ chosen. Summarizing, we have shown that, given $H\in \mathfrak{a}_{\epsilon}^{+}$, there exists polynomials $p_{H,\mu,Q}(t,v)$, such that
\[
 \lambda(\exp tH\cdot v)\sim  \sum_{ \mu\in E(P_{0},V),\,\, Q\in L_{0}^{+}} e^{t(\mu-Q)(H)}p_{H,\mu,Q}(t,v),
\]
as $t\rightarrow \infty$. To finish the proof, we need to show that the $p_{H,\mu,Q}(t,v)$ are actually polynomial on $\mathfrak{a}_{0}$.

 Let $H_{1}, H_{2} \in \mathfrak{a}_{\epsilon}^{+}$. If we can show that $P^{k}_{\epsilon,H_{2}}F^{k}_{\epsilon,H_{1}}(v)=P^{k}_{\epsilon,H_{1}}F^{k}_{\epsilon,H_{2}}(v)$, then we have finished the proof of the lemma. To simplify notation, we will write $P_{1}$ for $P^{k}_{\epsilon,H_{1}}$, $F_{1}$ for $F^{k}_{\epsilon,H_{1}}$, and analogously for $P_{2}$ and $F_{2}$. We will start by making the following simple observations: Since $H_{1}, H_{2} \in \mathfrak{a}_{\epsilon}^{+}$, then $1\geq \langle H_{1},H_{2}\rangle = c > 0$. Therefore, if $s,t \geq 0$,
\begin{eqnarray*}
 \langle sH_{1}+tH_{2} ,sH_{1}+tH_{2}\rangle &= & s^{2}\langle H_{1},H_{1}\rangle +2st \langle H_{1},H_{2}\rangle +t^{2}\langle H_{2},H_{2}\rangle \\ & = & s^{2} +2stc +t^{2}\geq c^{2}s^{2}+2stc +t^{2}=(sc+t)^{2}.
\end{eqnarray*}
The upshot is that we have shown $s+t \geq \|sH_{1}+tH_{2}\| \geq cs+t$. Now observe that, if $\mu \in E^{k}_{\epsilon,H_{1}}(P,V)\cap E^{k}_{\epsilon,H_{2}}(P,V)$, then
 $\mu(H_{1}), \mu(H_{2}) > L-k\epsilon+\delta$, and hence
\begin{equation}
 \mu\left(\frac{sH_{1}+tH_{2}}{\|sH_{1}+tH_{2}\|}\right) \geq \frac{s\mu(H_{1})+t\mu(H_{2})}{s+r} >L-k\epsilon+\delta. \label{eq:muinequality}
\end{equation}
Making use of this observations, and of equation (\ref{eq:eminussbGsvlessthandelta}), we can check that
\begin{eqnarray}
 \|e^{-B(sH_{1}+tH_{2})}P_{2}P_{1}G(\exp(sH_{1}+tH_{2})\cdot v)\| & \leq  &   (1+ \|sH_{1}+tH_{2}\|)^{2d}e^{-\delta\|sH_{1}+tH_{2}\|} \nonumber \\ & & {} \times \sigma_{\lambda}(v) \nonumber\\ & \leq&   (1+ s+t)^{2d}e^{-\delta (cs+ t)}\sigma_{\lambda}(v). \label{eq:estimatesplust}
\end{eqnarray}
Hence, by the definition of $F_{1}$,
\begin{eqnarray*}
 e^{-tB(H_{2})}P_{2}F_{1}(\exp t H_{2} \cdot v) & = &  e^{-tB(H_{2})}P_{2}P_{1}F(\exp t H_{2}\cdot v) \\
 & & {} + \int_{0}^{\infty}e^{-B(sH_{1}+tH_{2})}P_{2}P_{1}G(\exp(sH_{1}+tH_{2})\cdot v)\, ds,
\end{eqnarray*}
where the convergence of the last integral is guaranteed by equation (\ref{eq:estimatesplust}). Taking the limit as $t\rightarrow \infty$, and using again equation (\ref{eq:estimatesplust}), we get
\begin{eqnarray}
 \lim_{t\rightarrow\infty} e^{-tB(H_{2})}P_{2}F_{1}(\exp t H_{2}\cdot v) & = & \lim_{t\rightarrow\infty} e^{-tB(H_{2})}P_{2}P_{1}F(\exp t H_{2}\cdot v) + 0 \nonumber \\
& = & P_{1}F_{2}(v), \label{eq:H1toH2}
\end{eqnarray}
where the last equality follows from equation (\ref{eq:fkepsilonH}). But now, using equation (\ref{eq:identity}), we have that 
\begin{eqnarray*}
e^{-sB(H_{1})}\lim_{t\rightarrow\infty}e^{-tB(H_{2})}P_{2}F_{1}(\exp t H_{2}\cdot[\exp s H_{1}\cdot v]) & = &  e^{-sB(H_{1})} P_{1}F_{2}(\exp s H_{1}\cdot v) \\
\lim_{t\rightarrow\infty} e^{-tB(H_{2})}P_{2}F_{1}(\exp t H_{2}\cdot v) & = &   P_{1}F_{2}(v).
\end{eqnarray*}
Using this identity, and reversing the roles of $H_{1}$ and $H_{2}$ in equation (\ref{eq:H1toH2}), we get that
\begin{eqnarray*}
 \lim_{t\rightarrow \infty}e^{-tB(H_{1})}P_{1}F_{2}(\exp t H_{1}\cdot v) & = & P_{2}F_{1}(v) \\
\lim_{t\rightarrow \infty} P_{1}F_{2}(v) & = & P_{2}F_{1}(v),
\end{eqnarray*}
as we wanted to show. 
\end{proof}

With this result at hand we will now start the proof of the theorem.

\begin{proof}[Proof (of theorem)]
 Once again we will only consider the case $P_{1}=P_{0}$, the remaining cases being a consequence of this result. The plan for the proof is teh following: Given $H\in \mathfrak{a}_{0}^{+}$ such that $\alpha(H)$ is an integer for all $\alpha \in \Phi(P_{0},A)$, we will show that, for any $v\in V$, the matrix coefficient $c_{\lambda,v}$ has asymptotic expansion
\begin{equation}
 \lambda(\exp tH\cdot  v) \sim \sum_{i=1}^{p}e^{z_{i}t}\sum_{n\geq 0}p_{i,n}(t,v)e^{-nt}, \label{eq:generalasymptoticexpansion}
\end{equation}
where $p_{i,n}(t,v)$ are functions that are polynomial on $t$ and continuous on $v$. Since asymptotic expansions are unique, if $v$ is a $K$-finite vector, this asymptotic expansion must coincide with the asymptotic expansion given in lemma \ref{lemma:asymptoticexpansion}. Now using that the $p_{j,n}$ are continuous, and that $V_{K}$ is dense in $V$, we see that this asymptotic expansion should be of the form specified in the statement of the theorem. 

To get the asymptotic expansion described in (\ref{eq:generalasymptoticexpansion}), we will first need to make some observations. By the Gelfand-Naimark decomposition $\mathfrak{g}=\bar{\mathfrak{n}}_{0}\oplus \mathfrak{m}_{0} \oplus \mathfrak{a}_{0} \oplus \mathfrak{n}_{0}$. Let $p:\mathfrak{g}\longrightarrow \mathfrak{m}_{0}\oplus \mathfrak{a}_{0}$ be the canonical projection with respect to this decomposition. It's a known result of Harish-Chandra that there exist elements $1=e_{1},\ldots,e_{d}\in Z((\mathfrak{m}_{0})_{\mathbb{C}}\oplus (\mathfrak{a}_{0})_{\mathbb{C}})$ such that
\[
 Z((\mathfrak{m}_{0})_{\mathbb{C}}\oplus (\mathfrak{a}_{0})_{\mathbb{C}})\simeq \bigoplus_{i} p(Z(\mathfrak{g}_{\mathbb{C}}))e_{i}.
\]
In particular for all $H\in \mathfrak{a}_{0}^{+}$, there exist elements $z_{ij}\in Z(\mathfrak{g})$, such that 
$He_{i}=\sum_{j}p(z_{ij})e_{j}. $
Let $\{\bar{X}_{1},\ldots,\bar{X}_{p}\}$ be a basis of $\bar{\mathfrak{n}}_{0}$ consisting of root vectors with corresponding roots $-\alpha_{1},\ldots,-\alpha_{p}$, $\alpha_{i} \in \Phi(P_{0},A)^{+}$. Then, there exists $Y_{ijk}\in U(\mathfrak{g})$ such that $z_{ij}=p(z_{ij})+\sum_{k}\bar{X}_{k}Y_{ijk}$, and hence
\[
 He_{i}=\sum_{j}z_{ij}e_{j}-\sum_{j,k}\bar{X}_{k}Y_{ijk}e_{j}.
\]
Therefore, given $v\in V$, if we set $U_{ik}=\sum_{j}g_{ijk}e_{j}$, then we get
\[
 He_{i} \cdot v=\sum_{j}\chi(z_{ij})e_{j}\cdot v-\sum_{k}\bar{X}_{k}U_{ik}\cdot v. 
\]
 Let $q$ be the canonical projection of $V$ onto $V/\overline{\bar{\mathfrak{n}}V}$. The above equation implies that $U(\mathfrak{a}_{0})q(v)$ is contained in the span of $e_{1}\cdot q(v),\ldots,e_{d}\cdot q(v)$.
Set $\chi_{ij}=\chi(z_{ij})$, and let $z_{1},\ldots,z_{p}$ be the generalized eigenvalues of the matrix $[\chi_{ij}]$. 
Now observe that 
\[
 H\bar{X}_{k}e_{j}U_{ik}\cdot v = \alpha_{k}(H)\bar{X}_{k}e_{j}U_{ik}\cdot v + \sum_{l}\bar{X}_{k}\chi_{jl}e_{l}U_{ik}\cdot v + \sum_{s}\bar{X}_{k}\bar{X}_{s}U_{sj}U_{ik}\cdot v.
\]
Let $q_{2}$ be the canonical projection of $V$ onto $V/\overline{\bar{\mathfrak{n}}^{2}V}$. Then, from the above equation, we see that the span of the $e_{i}\cdot v$'s and the $X_{k}e_{j}u_{ik}\cdot v$'s contains the subspace $U(\mathfrak{a}_{0})q_{2}(v)$ of $V/\overline{\bar{\mathfrak{n}}^{2}V}$. Furthermore for fixed $i,k$ the subspaces generated by the $X_{k}e_{j}u_{ik}\cdot v$ are invariant under the action of $H$, and the generalized eigenvalues associated with this action are of the form $z_{s}-\alpha_{k}(H)$. If we continue with this process, we can find a finite number of generating vectors for $U(\mathfrak{a}_{0})q_{k}(v)$ on $V/\overline{\bar{\mathfrak{n}}^{k}V}$. Furthermore the action of $H$ on this subspace has generalized eigenvalues of the form $z_{s}-m$ with $m$ a non negative integer (because $\alpha_{k}(H)$ is a non negative integer for all $k$). If we now proceed as in the proof of lemma \ref{lemma:asymptoticexpansion}, we would get an asymptotic expansion as the one described in (\ref{eq:generalasymptoticexpansion}). To finish the proof of the theorem, we just need to refine the above argument to include all $H\in \mathfrak{a}_{0}^{+}$, which is easily done.
\end{proof}

%\begin{corollary}[of proof of theorem]
% If $H$ is an admissible, finitely generated, Hilbert representation of $G$ with an infinitesimal character $\chi$. Then $V/\bar{\mathfrak{n}}V$ is isomorphic to the Casselman-Wallach closure of $V_{K}/\bar{\mathfrak{n}}V_{K}$.
%\end{corollary}

%\begin{proof}
% By looking at the asymptotic expansion of matrix coefficients, we see that $\bar{\mathfrak{n}}V$ is closed in $V$. Hence $V/\bar{\mathfrak{n}}V$ is a Frechet space containing $V_{K}/\bar{\mathfrak{n}}V_{K}$ as a dense subspace. Therefore, by the Casselman-Wallach theorem, $V/\bar{\mathfrak{n}}V$ is isomorphic to the Casselman-Wallach closure of $V_{K}/\bar{\mathfrak{n}}V_{K}$.
%\end{proof}

\section{The Schwartz space for $L^2(N\backslash G;\chi)$} \label{sec:Schwartz}

Before giving the definition of the Schwartz space for $L^2(N\backslash G;\chi)$, we will first recall the definition of the Schwartz space for $L^{2}(G)$, and some of its properties.

\begin{dfn}
 If $f\in C^{\infty}(G)$, $X$, $Y \in U(\mathfrak{g}_{\mathbb{C}})$, and $d \in \mathbb{N}$, set
\[
q_{X,Y,d}(f)=\sup_{g\in G}\,|L_{Y}R_{X}f(g)| \Xi(g)^{-1}(1+\log{\|g\|})^{d},
\]
where $\Xi$ is Harish-Chandra's ``$\Xi$-function''. We define the Schwartz space of $G$ to be
\[
 \C(G)=\{f\in C^{\infty}(G)\, | \, \mbox{$q_{X,Y,d}(f) < \infty$ for all $X$, $Y \in U(\mathfrak{g}_{\mathbb{C}})$, $d \in \mathbb{N}$} \}.
\]
\end{dfn}
If we endow $\C(G)$ with the topology induced by the seminorms $q_{X,Y,d}$, then $\C(G)$ becomes a Fr\'echet space. Furthermore, it is well known that $C_{c}^{\infty}(G)\subset \C(G) \subset L^{2}(G)$, and that this inclusions are continuous and dense if we use the usual topologies on this spaces.

Let $P=MAN$ be a Siegel parabolic subgroup of $G$, with given Langlands decomposition. Then we will fix, as usual, a generic character $\chi$ of $N$.
\begin{dfn}
 If $f\in C^{\infty}(N\backslash G;\chi)$, $X\in U(\mathfrak{g}_{\mathbb{C}})$, and $d_{1}, d_{2} \in \mathbb{N}$, set
\[
q_{X,d_{1},d_{2}}(f)=\sup_{g\in G}a(g)^{-\rho}\Xi_{M}(m(g))^{-1}(1+\log{\|a(g)\|})^{d_{1}}((1+\log{\|m(g)\|})^{d_{2}} |R_{X}f(g)|.                                                                                                                          
\]
Here $\Xi_{M}$ is the ``$\Xi$-function'' for the group $M$. Define the Schwartz space for $L^2(N\backslash G;\chi)$ to be the space
$$
\C(N\backslash G;\chi)=\{f\in C^{\infty}(N\backslash G;\chi)\, |\,\mbox{$q_{X,d_{1},d_{2}}(f) < \infty$ for all $X \in U(\mathfrak{g}_{\mathbb{C}})$, $d_{1}$, $d_{2}\in \mathbb{N}$}\}.
$$
\end{dfn}
 We will endow $\C(N\backslash G;\chi)$ with the topology induced by the seminorms $q_{X,d_{1},d_{2}}$. Then it is easily seen that $\C(N\backslash G;\chi)$ is a Fr\'echet space and that the space $C_{c}^{\infty}(N\backslash G;\chi)$ of all $f\in C^{\infty}(N\backslash G;\chi)$, such that $|f|\in C_{c}^{\infty}(N\backslash G)$, is dense in $\C(N\backslash G;\chi)$.

\begin{lemma}\label{lemma:NschwartzthenL2}
 If $f\in \C(N\backslash G;\chi)$, then $f\in L^2(N\backslash G;\chi)$. Furthermore, there exists $d_{1}, d_{2} \in \mathbb{N}$ and $0\leq C < \infty$ such that $\|f\|_{2}\leq C q_{1,d_{1},d_{2}}(f)$.
\end{lemma}

\begin{proof}
 Let $d_{1}$ and $d_{2}$ be so large that
\[
 \int_{A}(1+\log{\|a\|})^{-2d_{1}} \, da =C_{A}^{2} < \infty,
\]
and
\[
 \int_{M} \Xi_{M}(m)^{2}(1+\log{\|m(g)\|})^{-2d_{2}}\, dm = C_{M}^{2} < \infty,
\]
for some positive constants $C_{A}$ and $C_{M}$. If $f\in \C(N\backslash G;\chi)$, then
\[
 |f(namk)|\leq a^{\rho}(1+\log{\|a(g)\|})^{-d_{1}}\Xi_{M}(m)((1+\log{\|m(g)\|})^{-d_{2}}q_{1,d_{1},d_{2}}(f).
\]
Thus,
\begin{eqnarray*}
\|f\|^{2}& = & \int_{N\backslash G}|f(g)|^{2}\, dg =\int_{A}\int_{M}\int_{K}a^{-2\rho}|f(amk)|^{2}\,da\,dm\,dk \\
 & \leq & q_{1,d_{1},d_{2}}(f)^{2}\int_{A}(1+\log{\|a\|})^{-2d_{1}} \, da\int_{M} \Xi(m)^{2}(1+\log{\|m(g)\|})^{-2d_{2}}\, dm\\
& \leq & C_{A}^{2}C_{M}^{2}q_{1,d_{1},d_{2}}^{2}.
\end{eqnarray*}
\end{proof}

\begin{lemma}
 Let $\phi\in \C(N\backslash G;\chi)$ and $f\in \C(G)$. Then
\[
 \int_{G}\phi(g)f(g)\, dg=: (\phi,f)
\]
converges absolutely and there exist continuous seminorms $q_{1}$ and $q_{2}$ on $\C(N\backslash G;\chi)$ and $\C(G)$, respectively, such that $|(\phi,f)|\leq q_{1}(\phi)q_{2}(f)$.
\end{lemma}

\begin{proof}
 We are looking at 
\[
 \int_{N}\int_{A}\int_{M}\int_{K}a^{-2\rho}\chi(n)\phi(amk)f(namk)\,dn\,da\,dm\,dk.
\]
Now for each $d_{1},d_{2}$, we have
\[
 |\phi(amk)|\leq q_{1,d_{1},d_{2}}(\phi)a^{\rho}(1+\log\|a\|)^{-d_{1}}\Xi(m)(1+\log{\|m(g)\|})^{-d_{2}}.
\]
Thus,
\begin{eqnarray*}
 \int_{G}|\phi(g)f(g)|dg & \leq & q_{1,d_{1},d_{2}}(\phi)\int_{N\times A \times M \times K}a^{-\rho}(1+\log\|a\|)^{-d_{1}} \\
 & & {} \times \Xi_{M}(m)(1+\log{\|m(g)\|})^{-d_{2}}|f(namk)|\,dn\,da\,dm\,dk.
\end{eqnarray*}
In the proof of theorem 7.2.1 of \cite{w:vol1}, it is shown that there exists $p$, a continuous semi-norm on $\C(G)$ such that
\begin{equation}
 a^{-\rho}\Xi_{M}(m)^{-1}\int_{N}|f(nam)|\, dn\leq p(f). \label{eq:harishchandratransformiscontinuous}
\end{equation}
Thus, if we take $q_{2}(f)=\sup_{k\in K}p(R_{k}f)$, then we have
\begin{eqnarray*}
 \int_{G}|\phi(g)f(g)|\,dg & \leq & q_{1,d_{1},d_{2}}(\phi)q_{2}(f)\int_{A}(1+\log\|a\|)^{-d_{1}} \, da \\
& & {} \times \int_{M} \Xi(m)^{2}(1+\log{\|m(g)\|})^{-d_{2}})\, dm.
\end{eqnarray*}
If $d_{1}$ and $d_{2}$ are sufficiently large, then the integral on the right converges. Therefore, if we take $q_{1}=q_{1,d_{1},d_{2}}$, with $d_{1}$ and $d_{2}$ large enough, we obtain the statement of the lemma.
\end{proof}

Note that equation (\ref{eq:harishchandratransformiscontinuous}) immediately implies the following lemma:
\begin{lemma}\label{lemma:NL1}
 If $f\in \C(G)$, set 
\[
 f_{\chi}(g)=\int_{N}\chi(n)^{-1}f(ng)\, dn.
\]
The integral converges absolutely and the map $f\mapsto f_{\chi}$ is continuous from $\C(G)$ to $\C(N\backslash G;\chi)$.
\end{lemma}

Let $(\pi,H_{\pi})$ be an irreducible, square integrable, Hilbert representation of $G$, and let $V_{\pi}$ be the space of $C^{\infty}$-vectors of $H_{\pi}$. Then $(\pi,V_{\pi})$ is an irreducible, Fr\'echet representation of moderate growth. Now, let $H_{\pi}'$ be the dual of $H_{\pi}$, and let $(\pi',H_{\pi}')$ be the contragradient representation defined by $(\pi'(g)\phi)(v)=\phi(\pi(g)^{-1}v)$, for $\phi\in H_{\pi}'$, $v \in H_{\pi}$, $g\in G$. Let $\check{V_{\pi}}=V_{\check{\pi}}$ be the space of $C^{\infty}$-vectors of $H_{\pi}'$, and let $\check{\pi}=\pi'|_{V_{\check{\pi}}}$. Then $(\check{\pi},V_{\check{\pi}})$ is also an irreducible, Fr\'echet representation of moderate growth. Observe that there is a natural $G$-invariant bilinear pairing
\[
 (\cdot,\cdot):V_{\check{\pi}} \times V_{\pi} \longrightarrow \mathbb{C}
\]
given by $(\phi,v)=\phi(v)$. 

\begin{lemma} \label{lemma:matrixcoefficientinSchwartspace} Given $\phi\in V_{\check{\pi}}$, and $v\in V_{\pi}$, define
 \[
  c_{\phi,v}(g)=(\phi,\pi(g)v).
 \]
Then $c_{\phi,v} \in \C(G)$.
\end{lemma}

\begin{proof}
 Let $g\in G$. According to the $KAK$ decomposition, there exists $k_{1},k_{2}\in K$, and $a\in A_{\circ}^{+}$ such that $g=k_{1}ak_{2}$. Hence according to part i) of theorem \ref{thm:asymptoticexpansion}, there exists $d_{1}\geq 0$, and continuous seminorms $q_{1}$, $q_{2}$ on $V_{\check{\pi}}$, $V_{\pi}$, respectively, such that, for all $X$, $Y \in U(\mathfrak{g}_{\mathbb{C}})$,
\begin{eqnarray*}
|L_{Y}R_{X}c_{\phi,v}(g)| & = & |(\check{\pi}(k_{1})^{-1}\check{\pi}(Y)\phi,\pi(a)\pi(k_{2})\pi(X)v)|\\
&  \leq & (1+\log{\|a\|})^{d_{1}} a^{\Lambda} q_{1}(\check{\pi}(k_{1})^{-1}\check{\pi}(Y)\phi)q_{2}(\pi(k_{2})\pi(X)v)  \\
& \leq &  (1+\log{\|a\|})^{d_{1}} a^{\Lambda} q_{Y}(\phi)q_{X}(v),
\end{eqnarray*}
where $q_{Y}(\phi)=\sup_{k\in K}q_{1}(\check{\pi}(k)^{-1}\check{\pi}(Y)\phi)$, and $q_{X}(v)=\sup_{k\in K}q_{2}(\check{\pi}(k)\check{\pi}(X)v)$. On the other hand theorem 4.5.3 of \cite{w:vol1} says that there exist constants $C$, $d_{2}$ such that
\[
 a^{-\rho_{\circ}}\leq \Xi(g) \leq Ca^{-\rho_{\circ}}(1+\log{\|a\|})^{d_{2}}.
\]
Therefore, for all $d\geq 0$
\[
 |L_{Y}R_{X}c_{\phi,v}(g)|\Xi(g)^{-1}(1+\log{\|g\|})^{d}\leq q_{Y}(\phi)q_{X}(v) (1+\log{\|a\|})^{d_{1}+d} a^{\Lambda+\rho_{\circ}}.
\]
But now since $V_{\pi}$ is square integrable, $\Lambda + \rho_{\circ} \in - {}^{+}\mathfrak{a}_{\circ}'$. Hence there exists a constant, $C_{X,Y,d}$, such that $|L_{Y}R_{X}c_{\phi,v}(g)|\Xi(g)^{-1}(1+\log{\|g\|})^{d} \leq C_{X,Y,d}$ for all $g\in G$. Since $X,Y$ and $d$ were arbitrary, we conclude that $c_{\phi,v} \in \C(G)$ as we wanted to show.
\end{proof}

\begin{proposition}\label{prop:squareintegrablemultiplicity}
 Let $(\pi, H_{\pi})$ be a square integrable Hilbert representation of $G$, and let $P=MAN$ be a Siegel parabolic subgroup with given Langlands decomposition. Let $\chi$ be a character of $N$ whose stabilizer $M_{\chi}$ in $M$ is compact, $(\tau, H_{\tau})$ an irreducible, finite dimensional representation of $M_{\chi}$ and let $Wh_{\chi}(V_{\pi})(\check{\tau})$ be the $\check{\tau}$ isotypic component of $Wh_{\chi}(V_{\pi})$ under the action of $M_{\chi}$.  Given $\lambda \in Wh_{\chi}(V_{\pi})(\check{\tau})$ we will set $T_{\lambda}(v)=c_{\lambda,v}$. Then $T_{\lambda}$ defines a continuous intertwiner operator between $V_{\pi}$ and $\C(N\backslash G;\chi)(\check{\tau})$ (the $\check{\tau}$ isotypic component of $\C(N\backslash G;\chi)$ under the left action of $M_{\chi}$). Furthermore, each $T_{\lambda}$ extends to a continuous intertwining operator from $H$ to $L^2(N\backslash G;\chi)(\check{\tau})$, and
\[
 \operatorname{Hom}_{G}(H,L^{2}(N\backslash G;\chi)(\check{\tau}))=\{T_{\lambda}\, | \, \lambda \in Wh_{\chi}(V_{\pi})(\check{\tau}) \}.
\]
\end{proposition}

\begin{proof}
We will first show $T_{\lambda}(v)\in \C(N\backslash G;\chi)(\check{\tau})$ for all $v \in V_{\pi}$. According to the Iwasawa decomposition, given $g\in G$, we can find $n\in N$, $a\in A$, $m\in M$, and $k\in K$ such that $g=namk$. If we now use the $KAK$ decomposition for $M$, we can find $k_{1}$, $k_{2}\in M_{\chi}$ and $a_{m}\in (A_{\circ}\cap M)^{+}$ such that $m=k_{1}a_{m}k_{2}$. Therefore 
\begin{eqnarray}
|R_{X}c_{\lambda,v}(g)| & = & |\pi^{T}(k_{1}n)^{-1}\lambda(\pi(aa_{m})\pi(k_{2}k)\pi(X)v)|\nonumber\\
&  = & |\chi(n)\pi^{T}(k_{1})^{-1}\lambda(\pi(aa_{m})\pi(k_{2}k)\pi(X)v)|\nonumber\\
&  = & |\pi^{T}(k_{1})^{-1}\lambda(\pi(aa_{m})\pi(k_{2}k)\pi(X)v)| \label{eq:leavingaamalone}
\end{eqnarray}
Now, since $Wh_{\chi}(V_{\pi})(\check{\tau})$ is finite dimensional, there exists $\lambda_{1},\ldots, \lambda_{r}\in Wh_{\chi}(V_{\pi})(\check{\tau})$ such that
\[
 \pi^{T}(k)\lambda = \sum \phi_{i}(k)\lambda_{i},
\]
for some functions $\phi_{1},\ldots,\phi_{r}\in C^{\infty}(M_{\chi})$. From this and equation (\ref{eq:leavingaamalone})
\begin{eqnarray*}
 |R_{X}c_{\lambda,v}(g)| & = & |\sum \phi_{i}(k)\lambda_{i}(\pi(aa_{m})\pi(k_{2}k)\pi(X)v)| \\
 & \leq & \sum |\phi_{i}(k)||\lambda_{i}(\pi(aa_{m})\pi(k_{2}k)\pi(X)v)|.
\end{eqnarray*}
Now let $Q$ be a minimal parabolic subgroup, such that $aa_{m}\in \operatorname{Cl}(A_{Q}^{+})$. Then, by part i) of theorem \ref{thm:asymptoticexpansion}, there exists $d\geq 0$, and continuous seminorms $\sigma_{\lambda_{i}}$ such that
\begin{eqnarray*}
 |R_{X}c_{\lambda,v}(g)| & \leq & \sum |\phi_{i}(k)|(1+\log \|aa_{m}\|)^{d}(aa_{m})^{\Lambda_{Q}}\sigma_{\lambda_{i}}(\pi(k_{2}k)\pi(X)v)) \\
 & \leq & C_{X}(1+\log \|aa_{m}\|)^{d}(aa_{m})^{\Lambda_{Q}} \\
 & \leq & C_{X}(1+\log \|a\|)^{2d}(1+\log \|a_{m}\|)^{2d}(aa_{m})^{\Lambda_{Q}}
\end{eqnarray*}
where
\[
 C_{X}= \sup_{\begin{array}{c} \scriptstyle k_{1}\in M_{\chi}\\ \scriptstyle k_{2}\in K\end{array}} \sum_{i} |\phi_{i}(k_{1})||\sigma_{\lambda_{i}}(\pi(k_{2})Xv).
\]
Observe that $\Lambda_{Q}=-\mu_{Q}-\rho_{Q}$ with $\mu_{Q} \in {}^{+}\mathfrak{a}_{Q}'$. Now, since 
$$
\mathfrak{n}_{\circ}=\mathfrak{n}_{\circ}\cap \bar{\mathfrak{n}}_{Q}\oplus \mathfrak{n}_{\circ}\cap \mathfrak{n}_{Q}= \mathfrak{n}_{\circ}\cap \mathfrak{n}_{Q} \oplus \mathfrak{n}_{\circ}\cap \bar{\mathfrak{n}}_{Q},
$$ 
then $\rho_{\circ}=-\rho_{Q}+\delta_{Q}$, with $\delta_{Q}\in \operatorname{Cl}({}^{+}\mathfrak{a}_{Q}')$ and also $\rho_{\circ}=\rho_{Q}-\gamma_{Q}$, with $\gamma_{Q}\in \operatorname{Cl}({}^{+}\mathfrak{a}_{Q}')$. Therefore $\Lambda=-\mu_{Q}-\delta_{Q}+\rho_{\circ}=-\mu_{Q}-\gamma_{Q}-\rho_{\circ}$.
On the other hand theorem 4.5.3 of \cite{w:vol1} says that there exist constants $C$, $d_{2}$ such that
\[
 a^{-\rho_{\circ}}\leq \Xi_{M}(m) \leq Ca^{-\rho_{m}}(1+\log{\|a\|})^{d_{2}}.
\]
Therefore for all $d_{1}$, $d_{2}\geq 0$
\begin{eqnarray*}
\lefteqn{|R_{X}c_{\lambda,v}(g)|\Xi_{M}(m)^{-1}(1+\log{\|a_{m}\|})^{d_{1}}(1+\log{\|a\|})^{d_{2}}a^{-\rho}}\\
& \leq & C_{X} (1+\log{\|a_{m}\|})^{d_{1}+2d}(1+\log{\|a\|})^{d_{2}+2d} a^{-\mu_{Q}-\delta_{Q}+\rho_{\circ}}a_{m}^{-\mu_{Q}-\gamma_{Q}-\rho_{\circ}}a^{-\rho}a_{m}^{\rho_{m}} \\
& = & C_{X} (1+\log{\|a_{m}\|})^{d_{1}+2d}(1+\log{\|a\|})^{d_{2}+2d} a^{-\mu_{Q}-\delta_{Q}}a_{m}^{-\mu_{Q}-\gamma_{Q}}.
\end{eqnarray*}
Now, since $\mu_{Q}\in {}^{+}\mathfrak{a}_{Q}'$, and $\delta_{Q}$, $\gamma_{Q} \in \operatorname{Cl}({}^{+}\mathfrak{a}_{Q}')$, we conclude that there exists a constant $C_{X,d_{1},d_{2}}$ such that
\[
 |R_{X}c_{\lambda,v}(g)|\Xi_{M}(m)^{-1}(1+\log{\|a_{m}\|})^{d_{1}}(1+\log{\|a\|})^{d_{2}}a^{-\rho} \leq C_{X,d_{1},d_{2}}.
\]
Since $X$, $d_{1}$ and $d_{2}$ were arbitrary we conclude that $c_{\lambda,v}\in \C(N\backslash G;N)$ as we wanted to show.

We will now define a new $G$-invariant inner product on $V_{\pi}$ in the following way: given $v_{1}$, $v_{2}\in V_{\pi}$, we have that $T_{\lambda}(v_{1})$, $T_{\lambda}(v_{2}) \in \C(N\backslash G;N)$, and hence, by lemma \ref{lemma:NschwartzthenL2}, $T_{\lambda}(v_{2})$, $T_{\lambda}(v_{2})\in L^{2}(N\backslash G)$. Set
\[
 \langle v_{1}, v_{2} \rangle_{\lambda} = \langle T_{\lambda}(v), T_{\lambda}(w) \rangle.
\]
Then it is clear that $\langle \cdot, \cdot \rangle_{\lambda}$ is $G$-invariant. Therefore, by Schur lemma,
\[
 \langle v_{1}, v_{2} \rangle_{\lambda}=\langle T_{\lambda}(v_{1}), T_{\lambda}(v_{2}) \rangle = c(\lambda)\langle v_{1}, v_{2} \rangle
\]
for some constant $c(\lambda)$. Thus $T_{\lambda}$ extends to a bounded operator from $H$ to $L^{2}(N\backslash G;\chi)(\check{\tau})$. We therefore see that
\[
 \operatorname{Hom}_{G}(H,L^{2}(N\backslash G;\chi)(\check{\tau}))\supset \{T_{\lambda}\, | \, \lambda \in Wh_{\chi}(V_{\pi})(\check{\tau}) \}.
\]
To prove the other inclusion observe that if $T\in \operatorname{Hom}_{G}(H,L^{2}(N\backslash G;\chi)(\check{\tau}))$, then $T$ maps $C^{\infty}$ vectors to $C^{\infty}$ vectors and defines a continuous intertwining operator on smooth Frechet representations. Now $L^{2}(N\backslash G;\chi)(\check{\tau})^{\infty} \subset C^{\infty}(N\backslash G;\chi)(\check{\tau})$ and evaluation at $1$ is continuous on $L^{2}(N\backslash G;\chi)(\check{\tau})^{\infty}$. Define $\lambda_{T}(v)=T(v)(1)$ for $v\in V_{\pi}$. Then, $\lambda_{T}\in Wh_{\chi}(V_{\pi})(\check{\tau})$ and $T=T_{\lambda_{T}}$. The result now follows.
\end{proof}

\section{The generalized Bessel-Plancherel theorem} \label{sec:besselplancherel}
After the work done in the previous two sections, we are finally ready to tackle conjecture \ref{conj:mainconjecture}. We will focus on proving the decomposition given in equation (\ref{eq:restrictiontoparabolic}), and we will then show how we can use this result to prove the generalized Bessel-Plancherel theorem given in (\ref{eq:besselplancherel}).

We will start by considering an irreducible, square integrable, Hilbert representation $(\pi, H_{\pi})$ of $G$. Let $V_{\pi}$, $V_{\check{\pi}}$ be as in the past section, and let $\phi \in V_{\check{\pi}}$ and $v \in V_{\pi}$ be arbitrary. By theorem \ref{lemma:matrixcoefficientinSchwartspace}, the function $c_{\phi,v}\in \C(G)$, and hence, by lemma \ref{lemma:NL1}, the integral
\[
 \int_{N}\chi(n)^{-1} ( \phi ,\pi(n) v ) \, dn
\]
converges absolutely. Therefore, we can define a map
\[
W_{\pi}^{\chi}: V_{\check{\pi}} \longrightarrow Wh_{\chi}(V_{\pi})                                               
\]
by
\[
 W_{\pi}^{\chi}(\phi)(v)=\int_{N}\chi(n)^{-1}(\phi,\pi(n) v) \, dn.
\]
When it's clear from the context what the the representation $\pi$ is, we will sometimes drop the suffix $\pi$ and denote this map simply by $W^{\chi}$. Observe that the matrix coefficient function $c_{W^{\chi}(\phi),v}\in \C(N\backslash G; \chi)$, and also observe that $R(g)c_{W^{\chi}(\phi),v}=c_{W^{\chi}(\phi),\pi(g)v}$. 

\begin{lemma}
 With notation and assumptions as above:
\begin{enumerate}
\item \label{l1}$W^{\chi}_{\pi}$ is $M_{\chi}N$ equivariant.
 \item \label{l2} There exists a non-degenerate, hermitian form $\langle \cdot, \cdot \rangle_{\pi}$ on $W^{\chi}(V_{\check{\pi}})$ such that 
\[
 \langle c_{W^{\chi}(\phi_{1}),v_{1}}, c_{W^{\chi}(\phi_{2}),v_{2}}\rangle = \frac{1}{d_{\pi}} \langle W^{\chi}(\phi_{1}) ,W^{\chi}(\phi_{2})\rangle_{\pi} \langle v_{1}, v_{2} \rangle,
\]
where $d_{\pi}$ is the formal degree of $\pi$.
\item \label{l3} For all $\phi_{1}$, $\phi_{2} \in V_{\check{\pi}}$,
\[
  \langle \phi_{1}, \phi_{2} \rangle = \int_{\hat{N}} \langle W^{\chi}(\phi_{1}), W^{\chi}(\phi_{2}) \rangle_{\pi} d\chi.
\]
Here $d\chi=d\lambda(\chi)$ in the notation of equation (\ref{eq:PlancheremeasureN}).
\item \label{l4} If $\phi \in V_{\check{\pi}}$, $v\in V_{\pi}$, then
\[
 W^{\chi}(\phi)(v)= \langle W^{\chi}(\check{v}), W^{\chi}(\phi) \rangle_{\pi},
\]
where $\check{v}$ is the element in $V_{\check{\pi}}$ defined by $\check{v}(w)=\langle v, w \rangle$, for all $w \in H_{\pi}$.
\item \label{l5} There is a non-degenerate bilinear pairing
\[
 (\cdot, \cdot )_{\pi} : W^{\chi}(V_{\check{\pi}}) \times W^{\overline{\chi}}(V_{\pi}) \longrightarrow \mathbb{C},
\]
 given by 
\[
 (W^{\chi}(\phi), W^{\overline{\chi}}(v) )_{\pi} := \langle W^{\chi}(\check{v}), W^{\chi}(\phi) \rangle_{\pi}.
\]
\end{enumerate}
\end{lemma}

\begin{proof}
Part \ref{l1} of this lemma follows directly from the definition of $W^{\chi}_{\pi}$ so we will start with the proof of \ref{l2}. Given $\lambda_{1}, \lambda_{2} \in W^{\chi}(V_{\check{\pi}})$ define a $G$-invariant inner product on $V_{\pi}$ by
\[
 \langle v_{1}, v_{2} \rangle_{\lambda_{1},\lambda_{2}}:= d_{\pi}\langle c_{\lambda_{1},v_{1}}, c_{\lambda_{2},v_{2}} \rangle.
\]
Since this inner product is $G$-invariant, then, by Schur lemma, there exists a constant $\langle \lambda_{1},\lambda_{2} \rangle_{\pi}$ such that
\[
 d_{\pi}\langle c_{\lambda_{1},v_{1}}, c_{\lambda_{2},v_{2}} \rangle= \langle v_{1}, v_{2} \rangle_{\lambda_{1},\lambda_{2}}=\langle \lambda_{1},\lambda_{2}\rangle_{\pi} \langle v_{1}, v_{2} \rangle.
\]
It is then clear that the bilinear form $\langle \cdot, \cdot \rangle_{\pi}$ defined this way is hermitian and non-degenerate.

We will now prove part \ref{l3}.  By classical Fourier analysis, if  $f\in L^{1}(N)\cap L^{2}(N)$, then
\begin{equation}
 \int_{N} \overline{f(n)}f(n)\, dn = \langle f, f \rangle = \langle \hat{f}, \hat{f} \rangle = \int_{\hat{N}} \overline{\hat{f}(\chi)}\hat{f}(\chi)\, d\chi.
\end{equation}
Hence, if $\phi_{1}$, $\phi_{2} \in V_{\check{\pi}}$, and $v_{1}, v_{2} \in V_{\pi}$, 
\begin{eqnarray*}
 \langle c_{\phi_{1},v_{1}}, c_{\phi_{2},v_{2}}\rangle & = & \int_{\hat{N}} \langle c_{W^{\chi}(\phi_{1}),v_{1}}, c_{W^{\chi}(\phi_{2}),v_{2}}\rangle \, d\chi \\
\frac{1}{d_{\pi}}\langle \phi_{1},\phi_{2} \rangle \langle v_{1},v_{2} \rangle & = & \int_{\hat{N}} \frac{1}{d_{\pi}}  \langle W^{\chi}(\phi_{1}),W^{\chi}(\phi_{2})\rangle_{\pi} \langle v_{1},v_{2}\rangle \,d\chi. 
\end{eqnarray*}
Here we are using that $c_{\phi_{1},v_{1}}$, $c_{\phi_{2},v_{2}}$ are in $L^{1}(N)\cap L^{2}(N)$ when restricted to $N$, according to lemma \ref{lemma:NL1} and lemma \ref{lemma:matrixcoefficientinSchwartspace}.
Since this equation holds for all $v_{1}, v_{2}\in V_{\pi}$, we conclude that
 \[
\langle \phi_{1}, \phi_{2} \rangle = \int_{\hat{N}}\langle W^{\chi}(\phi_{1}),W^{\chi}(\phi_{2})\rangle_{\pi} \, d\chi. \label{eq:pimeasure}  
 \]

We will now move to part \ref{l4}. From part \ref{l2},
\begin{equation} \label{eq:innerproduct}
 \langle c_{W^{\chi}(\phi_{1}),v_{1}}, c_{W^{\chi}(\phi_{2}),v_{2}} \rangle = \frac{1}{d_{\pi}}  \langle W^{\chi}(\phi_{1}),W^{\chi}(\phi_{2})\rangle_{\pi} \langle v_{1},v_{2}\rangle.
\end{equation}
On the other hand, by definition
\begin{eqnarray}
 \lefteqn{\langle c_{W^{\chi}(\phi_{1}),v_{1}}, c_{W^{\chi}(\phi_{2}),v_{2}} \rangle} \nonumber \\
& = & \int_{N\backslash G} \overline{W^{\chi}(\phi_{1})(\pi(g) v_{1})}W^{\chi}(\phi_{2})(\pi(g) v_{2}) dNg \nonumber \\
& = & \int_{N\backslash G} \int_{N}\chi(n_{1})\overline{( \phi_{1},\pi({n_{1}g})v_{1})}\, dn_{1} \int_{N}\chi(n_{2})^{-1}( \phi_{2},\pi({n_{2}g})v_{2})\, dn_{2} \,dNg \nonumber \\
& = & \int_{N\backslash G} \int_{N}\int_{N}\chi(n_{1}n_{2}^{-1}) \overline{( \phi_{1},\pi({n_{1}g})v_{1})}\, ( \phi_{2},\pi(n_{2}g)v_{2})\, dn_{1}\,dn_{2}\,dNg \nonumber \\
& = & \int_{N\backslash G} \int_{N}\int_{N}\chi(n_{1}) \overline{( \phi_{1},\pi({n_{1}n_{2}g})v_{1})}\, ( \phi_{2},\pi({n_{2}g})v_{2})\, dn_{1}\,dn_{2}\,dNg \nonumber \\
& = &\int_{N}\chi(n_{1}) \int_{N\backslash G} \int_{N} \overline{( \check{\pi}({n_{1})^{-1}\phi_{1},\pi(n_{2}g})v_{1})}\, ( \phi_{2},\pi({n_{2}g})v_{2}) \,dn_{2}\,dNg \, dn_{1} \nonumber \\
& = &\int_{N}\chi(n_{1}) \int_{G}  \overline{( \check{\pi}(n_{1})^{-1}\phi_{1},\pi(g)v_{1})}\, ( \phi_{2},\pi(g)v_{2}) \,dg \, dn_{1}\nonumber \\
& = & \int_{N}\chi(n_{1})^{-1}\frac{1}{d_{\pi}} \langle \check{\pi}(n_{1})\phi_{1},\phi_{2}\rangle \langle v_{1}, v_{2}\rangle \, dn_{1}\nonumber \\
& = & \int_{N}\chi(n_{1})^{-1}\frac{1}{d_{\pi}} (\phi_{2}, \pi(n_{1})\check{\phi}_{1}) \langle v_{1}, v_{2}\rangle \, dn_{1}\nonumber \\
& = & \frac{1}{d_{\pi}} W^{\chi}(\phi_{2})(\check{\phi_{1}}) \langle v_{1}, v_{2}\rangle.\label{eq:evaluation}
\end{eqnarray}
Now, since equations (\ref{eq:innerproduct}) and (\ref{eq:evaluation}) hold for all $v_{1}$, $v_{2}\in V_{\pi}$, we conclude that
\[
 W^{\chi}(\phi_{2})(\check{\phi}_{1})=\langle W^{\chi}(\phi_{1}), W^{\chi}(\phi_{2}) \rangle_{\pi},
\]
which is equivalent to the equation appearing in part \ref{l4}.

Finally, for part \ref{l5} we only need to check that the definition only depends on $W^{\overline{\chi}}(v)$, but by part \ref{l4}
\[
 (W^{\chi}(\phi), W^{\overline{\chi}}(v) )_{\pi} :=   \langle W^{\chi}(\check{v}), W^{\chi}(\phi) \rangle_{\pi}
  =   W^{\chi}(\phi)(v)=W^{\overline{\chi}}(v)(\phi),
\]
where the last equality follows from the intrinsic symmetry between $\phi$ and $v$ in the definition of $ W^{\chi}(\phi)(v)$.
\end{proof}
We will denote by $W_{\chi}(H_{\pi})$ the closure of $W^{\chi}(V_{\check{\pi}})$ with respect to the inner product $\langle \cdot, \cdot \rangle_{\pi}$. When it is clear what the representation $\pi$ is, we will also drop the suffix $\pi$ in the notation of this inner product.

Let $P=MAN$ be a Siegel parabolic subgroup of $G$, with given Langlands decomposition. Given $\chi \in \hat{N}$ and $p\in P$, define
\[
 \check{\pi}(p): Wh_{\chi}(V_{\pi})\longrightarrow Wh_{p\cdot \chi}(V_{\pi}),
\]
by $(\check{\pi}(p)\lambda)(v)=\lambda(\pi(p)^{-1}v)$, for $\lambda \in Wh_{\chi}(V_{\pi})$, and $v \in V_{\pi}$. Observe that, if $p_{1}$, $p_{2}\in P$, then $\check{\pi}(p_{1})\check{\pi}(p_{2})=\check{\pi}(p_{1}p_{2})$.  Also observe that, if $\phi \in V_{\check{\pi}}$, $v\in V_{\pi}$, then
\begin{eqnarray*}
 (\check{\pi}(p) W^{\chi}(\phi))(v) & = &  W^{\chi}(\phi)(\pi(p)^{-1}v)= \int_{N} \chi(n)^{-1} (\phi, \pi(n)\pi(p)^{-1}v)\, dn\\
& = & \int_{N} \chi(n)^{-1} (\phi, \pi(p)^{-1}\pi(pnp^{-1})v)\, dn \\
& = & \int_{N} \chi(p^{-1}np)^{-1} (\check{\pi}(p)\phi, \pi(n)v)\delta_{P}(p) \, dn \\
& = &  \delta_{P}(p)W^{p\cdot \chi}(\check{\pi}(p)\phi)(v),
\end{eqnarray*}
where $\delta_{P}$ is the modular function of $P$. Since $v\in V_{\pi}$ was arbitrary, we conclude that $\check{\pi}(p) W^{\chi}(\phi)=\delta_{P}(p)W^{p\cdot \chi}(\check{\pi}(p)\phi)$.

 Let $\Omega$ be the set of open orbits for the action of $P$ on $\hat{N}$. For every $\omega \in \Omega$, we will fix a character $\chi_{\omega}\in \omega$. Now, given $\phi \in V_{\check{\pi}}$, we will define
$f_{\phi,\omega}(p)\in W^{\chi}(V_{\check{\pi}})$ by
\[
 f_{\phi,\omega}(p)=\check{\pi}(p)W^{p^{-1}\cdot\chi_{\omega}}(\phi).
\]
Set $f_{\phi}(p)=\sum_{\omega \in \Omega} f_{\phi,\omega}(p)$.

\begin{proposition}\label{prop:descompositionparabolicsquareintegrable}
 The map $\phi \mapsto f_{\phi}$ induces a $P$-equivariant isometry between $H_{\check{\pi}}$ and $\oplus_{\omega \in \Omega}\operatorname{Ind}_{M_{\chi_{\omega}}N}^{P} W_{\chi_{\omega}}(H_{\pi})$.
\end{proposition}

\begin{proof}
 Let $\phi \in V_{\check{\pi}}$, then
\begin{eqnarray*}
 \langle f_{\phi}, f_{\phi} \rangle & =  & \int_{M_{\chi_{\omega}}N\backslash P} \langle f_{\phi}(p), f_{\phi}(p) \rangle \, dp\\
  & = & \int_{M_{\chi\chi_{\omega}}N\backslash P} \langle \sum_{\omega\in \Omega} \check{\pi}(p) W^{p^{-1}\cdot\chi_{\omega}}(\phi), \sum_{\omega\in \Omega} \check{\pi}(p) W^{p^{-1}\cdot\chi_{\omega}}(\phi) \rangle \, dp \\
& = & \sum_{\omega\in \Omega} \int_{\omega} \langle W^{\chi}(\phi) , W^{\chi}(\phi) \rangle \, d\chi \\
 & = & \int_{\hat{N}} \langle W^{\chi}(\phi) , W^{\chi}(\phi) \rangle \, d\chi = \langle \phi,\phi \rangle.
\end{eqnarray*}
Now given $f\in \operatorname{Ind}_{M_{\chi_{\omega}}N}^{P} W^{\chi_{\omega}}(V_{\check{\pi}})$, define $\phi_{f}$ by setting
\[
 W^{p^{-1}\cdot\chi_{\omega}}(\phi_{f})=\check{\pi}(p)^{-1}f_{\omega}(p).
\]
Then
\begin{eqnarray*}
 \langle \phi_{f}, \phi_{f} \rangle &= & \int_{\hat{N}} \langle W^{\chi}(\phi_{f}) , W^{\chi}(\phi_{f}) \rangle_{\pi}\, d\chi \\
 & = & \int_{M_{\chi_{\omega}}N\backslash P} \langle \check{\pi}(p)^{-1}f(p) , \check{\pi}(p)^{-1}f(p) \rangle a(p)^{2\rho} \, dp \\ 
 & = & \int_{M_{\chi_{\omega}}N\backslash P} \langle f(p) , f(p) \rangle  \, dp \\ 
& = & \langle f, f \rangle.
\end{eqnarray*}
Besides 
\[
 f_{\phi_{f}}(p) = \sum_{\omega\in \Omega} \check{\pi}(p) W^{p^{-1}\cdot\chi_{\omega}} (\phi_{f})=\sum_{\omega\in \Omega}\check{\pi}(p)\check{\pi}(p)^{-1}f_{\omega}(p)=f(p).
\]
and
\[
 W^{p^{-1}\cdot\chi_{\omega}}(\phi_{f_{\phi}})=\check{\pi}(p)^{-1}f_{\phi,\omega}(p)=\tau^{-1}(p)\check{\pi}(p) W^{p^{-1}\cdot\chi_{\omega}}(\phi) = W^{p^{-1}\cdot\chi_{\omega}}(\phi), 
\]
which implies that $\phi_{f_{\phi}}= \phi$.
\end{proof}

We now want to construct a map analogous to $W_{\pi}^{\chi}$ in the case where $(\pi, H_{\pi})$ is an induced representation. Let $P_{0}=M_{0}A_{0}N_{0}$ be a minimal parabolic subgroup, and assume that $P=MAN$ is a Siegel parabolic subgroup dominating $P_{0}$. Let $P_{1}=M_{1}A_{1}N_{1}$ be another parabolic subgroup dominating $P_{0}$. If $(\sigma, H_{\sigma})$ is an admissible, Hilbert representation of $M_{1}$, and $\nu \in (\mathfrak{a}_{1})_{\mathbb{C}}'=Lie(A_{1})_{\mathbb{C}}'$, we will set 
$$
I_{\sigma,\nu}^{\infty}=\left\{f:G\longrightarrow V_{\sigma} \, \left| \, \begin{array}{c} \mbox{$f$ is $C^{\infty}$, and $f(\bar{n}amk)=a^{\nu-\rho}\sigma(m)f(k)$} \\ \mbox{for all $n\in \bar{N}_{1}$, $a\in A_{1}$, and $m\in M_{1}$} \end{array} \right.\right\}.
$$
Here $\bar{P}_{1}=M_{1}A_{1}\bar{N}_{1}$ is the parabolic opposite to $P_{1}$, $\rho$ is half the sum of the positive roots of $P_{1}$ relative to $A_{1}$, and $V_{\sigma}$ is the set of smooth vectors of $H_{\sigma}$. We will denote by $I_{\sigma,\nu}$ the completion of this space with respect to the inner product
\[
 \langle f , f \rangle = \int_{K} \langle f(k), f(k)\rangle \, dk, \qquad f\in I_{\sigma,\nu}^{\infty},
\]
where $K\subset G$ is a maximal compact subgroup. 

Let $\chi$ be a generic character of $N$. Then, according to theorem \ref{thm:matsukidecomposition}, there is a subset $W_{\chi}$ of the Weyl Group of $M$ such that
\begin{equation}
 \bigcup_{w\in W_{\chi}} \bar{P}_{M,1} w M_{\chi} \subset M \label{eq:mchiorbits}
\end{equation}
is open and dense, here $\bar{P}_{M,1}=\bar{P}_{1}\cap M$. Set $M_{\chi_{1}}:=M_{\chi}\cap P_{1}$ and $N_{N_{1}}=N\cap N_{1}$. If $(\sigma, H_{\sigma})$ is unitary, and $\nu \in i\mathfrak{a}_{1}'$, then for all $f\in I_{\sigma,\nu}$
\begin{equation}
 \langle f, f\rangle = \sum_{w\in W_{\chi}}\int_{M_{\chi_{1}}\backslash M_{\chi}} \int_{N_{N_{1}}} \langle f(n_{1}wm), f(n_{1}wm) \rangle \, dn_{1}\,dm. \label{eq:mchiinnerproduct}
\end{equation}
Set 
$$
U_{\sigma,\nu}=\{f\in I_{\sigma,\nu}^{\infty} \, | \, \operatorname{supp} f \subset \bar{P}_{1}P\}.
$$ 
Then, by (\ref{eq:mchiorbits}) and (\ref{eq:mchiinnerproduct}), $U_{\sigma,\nu}$ is dense in $I_{\sigma,\nu}$ and $P$-invariant. Given $f\in U_{\sigma,\nu}$ we will set 
\[
 J_{\sigma,\nu}^{\chi}(f)=\int_{N_{N_{1}}}\chi(n_{1})^{-1} f(n_{1})\, dn_{1}.
\]

We will now consider the natural $G$-invariant pairing between $I_{\sigma,\nu}$ and $I_{\check{\sigma},-\nu}$ given by
\[
 (\phi, f)=\int_{P_{1}\backslash G} (\phi(g),f(g)) \, dg,
\]
where $dg=a^{-2\rho}dk$ for $g=\bar{n}amk$, $\bar{n}\in \bar{N}_{1}$, $a\in A_{1}$, $m\in M_{1}$, $k\in K$. If $\phi\in U_{\check{\sigma},-\nu}$, and $m\in M_{\chi}$, we will set
\[
W_{\sigma,\nu}^{\chi,w}(\phi)(m)=W^{\overline{\chi}_{1}}J_{\check{\sigma},-\nu}^{\overline{\chi}}(\pi(wm)\phi) \in W^{\chi_{1}}(V_{\check{\sigma}}),
\]
where $\chi_{1}=\chi|_{N_{M_{1}}}$, $N_{M_{1}}=N\cap M_{1}$.

\begin{lemma}
 If $\tilde{m}\in M_{\chi_{1}}$, then
\[
 W_{\sigma,\nu}^{\chi,w}(\phi)(\tilde{m}m)=\check{\sigma}^{w}(\tilde{m})W_{\sigma,\nu}^{\chi,w}(\phi)(m),
\]
\end{lemma}
where $\check{\sigma}^{w}(\tilde{m})=\check{\sigma}(w\tilde{m}w^{-1})$.
\begin{proof}
 By definition
\begin{eqnarray*}
 W_{\sigma,\nu}^{\chi,w}(\phi)(\tilde{m}m) & = & W^{\overline{\chi}_{1}}J_{\check{\sigma},-\nu}^{\overline{\chi}}(\pi(w\tilde{m}m)\phi) \\
 & = & W^{\overline{\chi}_{1}}\bigg(\int_{N_{N_{1}}}\chi(n_{1})^{-1} \phi(n_{1}(w\tilde{m}w^{-1})wm)dn_{1}\bigg). 
\end{eqnarray*}
Set $\tilde{m}^{w}=w\tilde{m}w^{-1}\in M_{\chi_{1}}$. Then
\begin{eqnarray*}
W_{\sigma,\nu}^{\chi,w}(\phi)(\tilde{m}m)& = & W^{\overline{\chi}_{1}}\bigg(\int_{N_{N_{1}}}\chi(n_{1})^{-1} \phi(\tilde{m}^{w} \big((\tilde{m}^{w})^{-1}n_{1}\tilde{m}^{w}\big)wm)dn_{1}\bigg) \\
& = & W^{\overline{\chi}_{1}}\bigg(\int_{N_{N_{1}}}\chi(n_{1})^{-1} \check{\sigma}(\tilde{m}^{w})\phi( \big((\tilde{m}^{w})^{-1}n_{1}\tilde{m}^{w}\big)wm))dn_{1}\bigg).
\end{eqnarray*}
Now since $\tilde{m}^{w}$ normalizes $N_{1}$ we have that
\begin{eqnarray*}
 W_{\sigma,\nu}^{\chi,w}(\phi)(\tilde{m}m) & = & \check{\sigma}(\tilde{m}^{w}) W^{\overline{\chi}_{1}}(\int_{N_{N_{1}}}\chi(\tilde{m}^{w}n_{1}(\tilde{m}^{w})^{-1})^{-1} \phi(n_{1}wm)dn_{1} \\
& = & \check{\sigma}(\tilde{m}^{w}) W^{\overline{\chi}_{1}}(\int_{N_{N_{1}}}\chi(n_{1})^{-1} \phi(n_{1}wm)dn_{1} \\
& = & \check{\sigma}(\tilde{m}^{w}) W^{\overline{\chi}_{1}}(\int_{N_{N_{1}}}\chi(n_{1})^{-1} \phi(n_{1}wm)dn_{1} \\
& = & \check{\sigma}^{w}(\tilde{m})W_{\sigma,\nu}^{\chi,w}(\phi)(m).
\end{eqnarray*}
\end{proof}

\begin{lemma}\label{lemma:inducedfouriertransform}
 If $\phi \in U_{\check{\sigma},-\nu}$ and $f\in U_{\sigma,\nu}$, then $c_{\phi,f}|_{N}\in L^{1}(N)\cap L^{2}(N)$ and
\[
 \int_{N}\chi(n) (\phi, \pi(n)^{-1}f)\, dn = \sum_{w\in W_{\chi}}\int_{M_{\chi_{1}}\backslash M_{\chi}} ( W_{\sigma,\nu}^{\chi,w}(\phi)(m), W_{\check{\sigma},-\nu}^{\overline{\chi},w}(\phi)(m))\, dm.
\]
\end{lemma}
\begin{proof}
From equation (\ref{eq:mchiinnerproduct})
 \begin{eqnarray*}
\lefteqn{\int_{N} | (\phi, \pi(n)^{-1}f)| \, dn   }\\ & = & \int_{N}|\int_{K_{M_{1}}\backslash K_{M}} \int_{N_{N_{1}}} (\phi(n_{1}kn),f(n_{1}k) )\,dn_{1}\,dk | \,dn \\
& \leq &   \int_{N}\int_{K_{M_{1}}\backslash K_{M}} \int_{N_{N_{1}}} |(\phi(n_{1}kn),f(n_{1}k) )|\,dn_{1}\,dk  \,dn \\
& \leq &   \int_{K_{M_{1}}\backslash K_{M}} \int_{N_{N_{1}}}\int_{N} |(\phi(nn_{1}k),f(n_{1}k) )|\,dn \,dn_{1}\,dk   \\
& \leq &   \int_{K_{M_{1}}\backslash K_{M}} \int_{N_{N_{1}}}\int_{N_{N_{1}}}\int_{N_{M_{1}}} |(\check{\sigma}(n_{m})\phi(n_{2}n_{1}k),f(n_{1}k) )|\,dn_{m}\,dn_{2}\,dn_{1}\,dk   \\
& \leq &   \int_{K_{M_{1}}\backslash K_{M}} \int_{N_{N_{1}}}\int_{N_{N_{1}}} |W^{1}_{\sigma}(\phi(n_{2}k))(f(n_{1}k) )|\,dn_{2}\,dn_{1}\,dk   \\
& \leq &   \int_{K_{M_{1}}\backslash K_{M}} \int_{N_{N_{1}}}\int_{N_{N_{1}}} |(W^{1}_{\sigma}(\phi(n_{2}k)), W^{1}_{\check{\sigma}}(f(n_{1}k)) )|\,dn_{2}\,dn_{1}\,dk.   
 \end{eqnarray*}
Now, since the support of $\phi, f$ is compact modulo $\bar{P}_{1}$ and contained in $\bar{P}_{1}P$ the last integral is convergent. Since $n\mapsto (\phi, \pi(n)^{-1}f)$ is continuous, bounded and $L^{1}$, then it is also $L^{2}$. Now
\begin{eqnarray*} 
\lefteqn{\int_{N} \chi(n)(\phi, \pi(n)^{-1}f)\, dn } \\ & = & \sum_{w\in W_{\chi}} \int_{N} \int_{ M_{\chi_{1}} \backslash M_{\chi}} \int_{N_{N_{1}}} \chi(n)(\phi(n_{1}wmn),f(n_{1}wm) )\,dn_{1}\,dm   \,dn \\
& = & \sum_{w\in W_{\chi}}  \int_{N}\int_{ M_{\chi_{1}} \backslash M_{\chi}} \int_{N_{N_{1}}}  \chi(n)(\phi(n_{1}wmn),f(n_{1}wm) ) \,dn_{1}\,dm  \,dn \\
& = &  \sum_{w\in W_{\chi}} \int_{ M_{\chi_{1}} \backslash M_{\chi}} \int_{N_{N_{1}}}\int_{N}  \chi(n)(\phi(nn_{1}wm),f(n_{1}wm) ) \,dn \,dn_{1}\,dm   \\
& = & \sum_{w\in W_{\chi}}  \int_{ M_{\chi_{1}} \backslash M_{\chi}} \int_{N_{N_{1}}}\\ & & \qquad \int_{N_{N_{1}}}\int_{N_{M_{1}}} \chi(n_{m}n_{2})  (\check{\sigma}(n_{m})\phi(n_{2}n_{1}wm),f(n_{1}wm) ) \,dn_{m}\,dn_{2}\,dn_{1}\,dm   \\
& = & \sum_{w\in W_{\chi}}  \int_{ M_{\chi_{1}} \backslash M_{\chi}} \int_{N_{N_{1}}}\int_{N_{N_{1}}} \chi(n_{2}n_{1}^{-1}) W^{\chi_{1}}_{\sigma}(\phi(n_{2}wm))(f(n_{1}wm) ) \,dn_{2}\,dn_{1}\,dm   \\
& = &  \sum_{w\in W_{\chi}}\int_{ M_{\chi_{1}} \backslash M_{\chi}} \int_{N_{N_{1}}} \\ & & \qquad \int_{N_{N_{1}}} \chi(n_{2}) \chi(n_{1})^{-1}(W^{\chi_{1}}_{\sigma}(\phi(n_{2}wm)), W^{\overline{\chi_{1}}}_{\check{\sigma}}(f(n_{1}wm)) ) \,dn_{2}\,dn_{1}\,dm.    \\
& = & \sum_{w\in W_{\chi}}\int_{M_{\chi_{1}}\backslash M_{\chi}} ( W_{\sigma,\nu}^{\chi,w}(\phi)(m), W_{\check{\sigma},-\nu}^{\overline{\chi},w}(\phi)(m))\, dm.
\end{eqnarray*}
\end{proof}

With this results in place, we are now ready to state the analog of proposition \ref{prop:descompositionparabolicsquareintegrable} for induced representations. This result is a very important step in the way of proving equation (\ref{eq:restrictiontoparabolic}).

\begin{proposition} \label{prop:descompositionparabolic}
 Given $\phi \in U_{\check{\sigma},-\nu}$, define
\[
 f_{\phi,\omega,w}(p)=W_{\sigma,\nu}^{\chi_{\omega},w}(\check{\pi}(p)\phi)(e), \qquad \omega \in \Omega, \quad w\in W_{\chi_{\omega}}.
\]
The map $\phi\mapsto \sum_{\omega,w}f_{\phi,\omega,w}$ extends to a $P$-equivariant isometry between $I_{\check{\sigma},-\nu}$ and $\oplus_{\omega,w} \operatorname{Ind}_{M_{(\chi_{\omega})_{1}}N}^{P} W_{(\chi_{\omega})_{1}}(H_{\sigma^{w}})$.
\end{proposition}

\begin{proof}
Observe that, by the definition of $W_{\sigma,\nu}^{\chi,w}$, if $m\in M_{\chi}$, then
\[
 W_{\sigma,\nu}^{\chi,w}(\check{\pi}(mp)\phi)(e)=W_{\sigma,\nu}^{\chi,w}(\check{\pi}(p)\phi)(m).
\]
Hence, if $\phi \in U_{\check{\sigma},-\nu}$,
\begin{eqnarray*}
\lefteqn{\langle \sum_{\omega,w} f_{\phi,\omega,w},\sum_{\omega,w} f_{\phi,\omega,w} \rangle}\\ & = & \sum_{\omega,w} \int_{M_{\chi_{\omega}}N\backslash P}\int_{M_{(\chi_{\omega})_{1}}\backslash M_{\chi}} ( W_{\sigma,\nu}^{\chi,w}(\check{\pi}(p)\phi)(m), W_{\check{\sigma},-\nu}^{\overline{\chi},w}(\check{\pi}(p)\phi)(m))\, dm\, dp \\
& = & \sum_{w} \int_{\hat{N}} \int_{M_{(\chi_{\omega})_{1}}\backslash M_{\chi}} ( W_{\sigma,\nu}^{\chi,w}(\phi)(m), W_{\check{\sigma},-\nu}^{\overline{\chi},w}(\phi)(m))\, dm d\chi \\
 & = & \int_{\hat{N}}\int_{N} \chi(n)^{-1}\langle \phi, \pi(n)^{-1}\phi \rangle\, dn \,d\chi = \langle \phi,\phi\rangle,
\end{eqnarray*}
where in the last step we have used lemma \ref{lemma:inducedfouriertransform} and Fourier inversion formula. We can extend this map to an injective isometry between the Hilbert spaces $I_{\check{\sigma},-\nu}$ and $\oplus_{\omega,w} \operatorname{Ind}_{M_{(\chi_{\omega})_{1}}N}^{P} W_{(\chi_{\omega})_{1}}(H_{\sigma^{w}})$. Now since
\[
 \{W_{\sigma,\nu}^{\chi,w}(\phi)(e)\, | \, \phi \in U_{\check{\sigma},-\nu}\} = W^{\chi_{1},w}(V_{\check{\sigma}^{w}})
\]
and $W^{\chi_{1},w}(V_{\check{\sigma^{w}}})$ is dense in $ W_{\chi_{1}}(H_{\sigma^{w}})$ we conclude that this extended map is surjective.
\end{proof}

We will now show how we can use this results to prove conjecture \ref{conj:mainconjecture}. Let $\chi$ be a generic character of $N$. From abstract representation theory, we have that as an $M_{\chi}$-module, 
\begin{equation}
 \operatorname{Ind}_{M_{\chi_{1}}}^{M_{\chi}}W_{\chi_{1}}(H_{\sigma^{w}}) \cong \int_{\hat{M}_{\chi}} \tilde{W}^{w}_{\sigma,\chi}(\tau) \otimes \tau^{\ast}\, d\eta_{\sigma,\chi}^{w}(\tau) \label{eq:taudecomposition}
\end{equation}
for some measure $\eta_{\sigma,\chi}^{w}$ that depends on $\sigma$, $\chi$, $w$, and some multiplicities $\tilde{W}_{\sigma,\chi}(\tau)$ that also depend on $\sigma$, $\tau$ and $w$. On the other hand, if $(P_{1},A_{1})$ is a standard parabolic subgroup with respect to $(P_{0},A_{0})$, then we will  write $(P_{1},A_{1})\succ (P_{\circ},A_{\circ})$. Let $\mathcal{E}_{2}(M_{1})$ be the set of irreducible, square integrable representations of $M_{1}$ up to equivalence. Then Harish-Chandra's Plancherel theorem states that
\begin{equation}
 L^{2}(G) \cong \bigoplus_{(P_{1},A_{1})\succ (P_{\circ},A_{\circ})} \bigoplus_{\sigma\in\mathcal{E}_{2}(M_{1})} \int_{i(\mathfrak{a}_{1}')^{+}} I_{\check{\sigma},-\nu}\otimes I_{\sigma,\nu}\, d\mu_{\sigma}(\nu), \label{eq:explicitplancherel}
\end{equation}
where $(\mathfrak{a}_{1}')^{+}$ is the positive Weyl chamber of $\mathfrak{a}_{1}'$ relative to $P_{1}$, and $\mu_{\sigma}$ is a measure on $(\mathfrak{a}_{1}')^{+}$ that can be calculated explicitly.
On the other hand, by proposition \ref{prop:descompositionparabolic} and equation (\ref{eq:taudecomposition})
\begin{eqnarray}
 I_{\check{\sigma},-\nu} & \cong & \bigoplus_{\omega,w} \operatorname{Ind}_{M_{(\chi_{\omega})_{1}}N}^{P} W_{(\chi_{\omega})_{1}}(H_{\sigma^{w}}) \nonumber \\
& \cong & \bigoplus_{\omega,w} \operatorname{Ind}_{M_{\chi_{\omega}}N}^{P}\operatorname{Ind}_{M_{(\chi_{\omega})_{1}}N}^{M_{\chi_{\omega}}N} W_{(\chi_{\omega})_{1}}(H_{\sigma^{w}}) \nonumber \\
& \cong & \bigoplus_{\omega,w} \operatorname{Ind}_{M_{\chi}N}^{P}\int_{\hat{M}_{\chi}} \tilde{W}^{w}_{\sigma,\chi}(\tau) \otimes \tau^{\ast}\, d\eta_{\sigma,\chi}^{w}(\tau) \nonumber \\
& \cong & \bigoplus_{\omega,w} \int_{\hat{M}_{\chi}} \tilde{W}^{w}_{\sigma,\chi}(\tau) \otimes \operatorname{Ind}_{M_{\chi}N}^{P}\tau^{\ast}\, d\eta_{\sigma,\chi}^{w}(\tau) \label{eq:explicitPdecomposition}
\end{eqnarray}
Hence, by equations (\ref{eq:explicitplancherel}) and equation (\ref{eq:explicitPdecomposition})
\begin{eqnarray*}
 L^{2}(G)&\cong& \bigoplus_{(P_{1},A_{1})} \bigoplus_{\sigma} \int_{(\mathfrak{a}_{1}')^{+}} \bigoplus_{w} \int_{\hat{M}_{\chi}} \tilde{W}^{w}_{\sigma,\chi}(\tau) \otimes \operatorname{Ind}_{M_{\chi}N}^{P}\tau^{\ast}\, d\eta_{\sigma,\chi}^{w}(\tau)\otimes I_{\sigma,\nu}\, d\mu_{\sigma}(\nu) \nonumber \\
 &\cong& \bigoplus_{(P_{1},A_{1})} \bigoplus_{\sigma}\bigoplus_{w} \int_{\hat{M}_{\chi}} \int_{(\mathfrak{a}_{1}')^{+}}  \tilde{W}^{w}_{\sigma,\chi}(\tau) \otimes \operatorname{Ind}_{M_{\chi}N}^{P}\tau^{\ast}\, \otimes I_{\sigma,\nu}\, d\mu_{\sigma}(\nu) \,d\eta_{\sigma,\chi}^{w}(\tau).
\end{eqnarray*}
Where in the last equation we have used that $\eta_{\sigma,\chi}^{w}$ is independent of $\nu$ to reverse the order of integration. On the other hand equation (\ref{eq:PtimesGgeometrical}) says that
\[
 L^{2}(G) \cong \bigoplus_{\omega \in \Omega} \int_{\hat{M}_{\chi_{\omega}}}\int_{\hat{G}} W_{\chi_{\omega},\tau}(\pi)\otimes \operatorname{Ind}_{M_{\chi_{\omega}}N}^{P}(\overline{\chi}_{\omega}\otimes \tau^{\ast})\otimes   \pi\, d\mu_{\omega,\tau}(\pi)d\eta(\tau).
\]
From this two equations, we conclude that $\eta_{\sigma,\chi}^{w}$ is absolutely continuous with respect to $\eta$, $\mu_{\omega,\tau}$ is absolutely continuous with respect to $\mu$ and $W_{\chi,\tau}(I_{\sigma,\nu}) \cong \bigoplus_{w} \tilde{W}^{w}_{\sigma,\chi}(\tau)$. Using this, and the series of equations leading to equation (\ref{eq:PtimesGgeometrical}), we conclude that
\begin{equation} \label{Bessel-Plancherel}
L^2(N \backslash G;\chi) \cong \ \int_{\hat{G}}\int_{\hat{M}_{\chi}}  W_{\chi,\tau}(\pi)\otimes \tau^{\ast} \otimes \pi \,d\eta(\tau)d\mu(\pi), 
\end{equation}
as we wanted to show.

\section{The Fourier transform of a wave packet}\label{sec:fouriertransform}
Let $P=MAN$ be a Siegel parabolic subgroup of a Lie group of tube type $G$. In the last section we proved the \emph{generalized Bessel-Plancherel theorem}, that is, we showed that if $\chi$ is a generic character of $N$, then
\[
L^2(N \backslash G;\chi) \cong \ \int_{\hat{G}}\int_{\hat{M}_{\chi}}  W_{\chi,\tau}(\pi)\otimes \tau^{\ast} \otimes \pi \,d\eta(\tau)d\mu(\pi), 
\]
where $\eta$, $\mu$ are the Plancherel measures of $M_{\chi}$ and $G$ respectively, and $W_{\chi,\tau}(\pi)$ is some multiplicity space, that we identified with a subspace of $Wh_{\chi,\tau}(\pi)$. What we want to show now is that if $M_{\chi}$ is compact, then $W_{\chi,\tau}(\pi)$ is actually isomorphic with $Wh_{\chi,\tau}(\pi)$, and hence finite dimensional. 

Fix a generic character $\chi$ of $N$ with compact stabilizer, and let $\O_{\chi}$ be the orbit of $\chi$ under the action of $P$ on $\hat{N}$. Then we have the following lemma
\begin{lemma}
Let $(\sigma,H_{\sigma})$ be an admissible, Hilbert representation of $M$. Given $\nu\in \mathfrak{a}_{\mathbb{C}}'$, denote by $I_{\sigma,\nu}$ the representation induced from the opposite parabolic to $P$ by $\sigma$ and $\nu$. Let $\phi\in I_{\check{\sigma},-\nu}^{\infty}$ be such that 
\[
 \mbox{$\phi|_{N}\in L^{1}(N)\cap L^{2}(N)$ and $\operatorname{supp} \hat{\phi} \subset \O_{\chi}$ is compact.}
\]
Let $f\in I_{\sigma,\nu}^{\infty}$ be arbitrary. Then
\[
 (\phi, f)= \int_{\hat{N}} (J_{\check{\sigma},-\nu}^{\overline{\chi_{0}}}(\phi),J_{\sigma,\nu}^{\chi_{0}}(f))\, d\chi_{0},
\]
if we use the convention that $J_{\sigma,\nu}^{\chi_{0}}(f)=0$ if $\chi_{0} \notin \O_{\chi}$.
\end{lemma}

\begin{proof}
 Observe that, for all $\chi_{0}\in \hat{N}$, $\hat{\phi}(\chi_{0})=J_{\check{\sigma},-\nu}^{\overline{\chi_{0}}}(\phi)$. Now
\[
 (\phi, f)=\int_{N}(\phi(n), f(n))\, dn.
\]
Given $\lambda \leq 0$, let $f^{\lambda}\in I_{\sigma,\nu}^{\infty}$ be defined by 
\[
 f^{\lambda}(k(n))=a(n)^{\lambda}f(k(n)) \qquad \mbox{$\lambda\in \mathfrak{a}'$ real,}
\]
where $n=\bar{n}(n)a(n)m(n)k(n)$, with $\bar{n}(n)\in \bar{N}$, $a(n)\in A$, $m(n)\in M$, $k(n)\in K$. Set
\[
 \Phi(\phi,f,\nu,\lambda)= \int_{N}(\phi(n), f^{\lambda}(n))\, dn.
\]
Then $\Phi$ is continuous for $\lambda \leq 0$, and analytic for $\lambda < 0$. Now for $\lambda \ll 0$
\begin{eqnarray*}
  \Phi(\phi,f,\nu,\lambda) & = &  \int_{N}(\phi(n),a(n)^{\lambda}f(n))\, dn. \\
 & = & \int_{\hat{N}} (J_{\check{\sigma},-\nu}^{\overline{\chi_{0}}}(\phi),J_{\sigma,\nu+\lambda}^{\chi_{0}}(f))\, d\chi_{0}.
\end{eqnarray*}
On the other hand, if we set
$$
\tilde{\Phi}(\phi,f,\nu,\lambda)= \int_{\hat{N}} (J_{\check{\sigma},-\nu}^{\overline{\chi_{0}}}(\phi),J_{\sigma,\nu+\lambda}^{\chi_{0}}(f))\, d\chi_{0},
$$
then $\tilde{\Phi}$ is analytic for all $\lambda$ and, for $\lambda \ll 0$, $\Phi(\phi,f,\nu,\lambda)= \tilde{\Phi}(\phi,f,\nu,\lambda)$. Therefore
\begin{eqnarray*}
 \Phi(\phi,f,\nu,0) & = & \tilde{\Phi}(\phi,f,\nu,0) \\
(\phi, f) & = & \int_{\hat{N}}( J_{\check{\sigma},-\nu}^{\overline{\chi_{0}}},J_{\sigma,\nu}^{\chi_{0}}(f))\, d\chi_{0},
\end{eqnarray*}
as we wanted to show.
\end{proof}

Let $P_{0}=M_{0}A_{0}N_{0}$ be a minimal parabolic subgroup, with given Langlands decomposition, and assume that $P$ dominates $P_{0}$. Let $P_{1}=M_{1}A_{1}N_{1}$ be another parabolic subgroup dominating $P_{0}$. Let $P_{M_{1}}=P\cap M_{1}=M_{M_{1}}A_{M_{1}}N_{M_{1}}$, where
\[
 M_{M_{1}}=M\cap M_{1} \qquad A_{M_{1}}=A\cap M_{1} \qquad N_{M_{1}}=N\cap M_{1}.
\]
Let $(\sigma, H_{\sigma})$ be a Hilbert, square integrable representation of $M_{1}$, and let $V_{\sigma}$ be its space of smooth vectors. Then, by the Casselman-Wallach theorem, there exist and injective intertwiner map
\[
 \alpha:V_{\sigma}\longrightarrow I_{\xi,\lambda}^{\infty}
\]
where $\xi$ is an irreducible, smooth representation of $M_{M_{1}}$, and $\lambda \in Lie(A_{M_{1}})_{\mathbb{C}}'$. The Casselman-Wallach theorem also implies the existence of a map
\[
 \alpha^{T}:I_{\check{\xi},-\lambda}^{\infty}\longrightarrow V_{\check{\sigma}}
\]
such that
\[
 \alpha^{T}(\phi)(v)=\phi(\alpha(v)),
\]
for all $\phi \in I_{\check{\xi},-\lambda}^{\infty}$, $v\in V_{\sigma}$.
\begin{lemma}\label{lemma:wchiforsmallinduced}
Let $V_{\sigma}$ and $I_{\xi,\lambda}$ be as above, and let $\chi_{1}=\chi|_{N_{M_{1}}}$. Then for all $\phi \in I_{\check{\xi},-\lambda}^{\infty}$, $v\in V_{\sigma}$,
\[
 W^{\chi_{1}}(\alpha^{T}(\phi))(v)=  (J_{\check{\xi},-\lambda}^{\overline{\chi_{1}}}(\phi),J_{\xi,\lambda}^{\chi_{1}}(f)).
\]
where $f=\alpha(v)$.
\end{lemma}

\begin{proof}
Let $\phi \in I_{\check{\xi},-\lambda}^{\infty}$ be such that
\[
 \mbox{$\phi|_{N_{M_{1}}}\in L^{1}(N_{M_{1}})\cap L^{2}(N_{M_{1}})$ and $\operatorname{supp} \hat{\phi} \subset \O_{\chi_{1}}$ is compact.}
\]
Then
\begin{eqnarray*}
 W^{\chi_{1}}(\alpha^{T}(\phi))(v) & = & \int_{N_{M_{1}}} \chi_{1}(n) ( \phi, \pi(n)^{-1}f ) \, dn \\
& = & \int_{N_{M_{1}}} \chi_{1}(n)\int_{\hat{N}_{M_{1}}} (J_{\check{\xi},-\lambda}^{\overline{\chi_{0}}}(\phi),J_{\xi,\lambda}^{\chi_{0}}(\pi(n)^{-1}f)) \, d\chi_{0} \, dn\\
& = & \int_{N_{M_{1}}}\chi_{1}(n) \int_{\hat{N}_{M_{1}}} \chi_{0}(n)^{-1}(J_{\check{\xi},-\lambda}^{\overline{\chi_{0}}}(\phi),J_{\xi,\lambda}^{\chi_{0}}(f)) \, d\chi_{0}\, dn \\
 & = &(J_{\check{\xi},-\lambda}^{\overline{\chi_{1}}}(\phi),J_{\xi,\lambda}^{\chi_{1}}(f)).
\end{eqnarray*}
On the other hand, we know that for all $v\in V_{\sigma}$ the map $\phi \mapsto  W^{\chi_{1}}(\alpha^{T}(\phi))(v)$ is in $Wh_{\overline{\chi_{1}}}(I_{\check{\xi},-\lambda}^{\infty})$. Hence there exists $\mu_{v}\in V_{\check{\xi}}$ such that
\[
  W^{\chi_{1}}(\alpha^{T}(\phi))(v)=\mu_{v}(J_{\check{\xi},-\lambda}^{\overline{\chi_{1}}}(\phi))
\]
From this two equations we conclude that $\mu_{v}=J_{\xi,\lambda}^{\chi_{1}}(f)$, and hence
\[
 W^{\chi_{1}}(\alpha^{T}(\phi))(v)=  (J_{\check{\xi},-\lambda}^{\overline{\chi_{1}}}(\phi),J_{\xi,\lambda}^{\chi_{1}}(f)),
\]
as we wanted to show.
\end{proof}

\begin{corollary}
 If $(\sigma, V_{\sigma})$ is a square integrable representation of $M_{1}$, define
\[
 F_{\sigma,\nu}^{\chi}: I_{\check{\sigma}, -\nu} \longrightarrow W^{\chi_{1}}(V_{\sigma})
\]
by
\[
 F_{\sigma,\nu}^{\chi}(\phi)=W^{\chi_{1}}\circ J_{\check{\sigma},-\nu}^{\overline{\chi}}(\phi).
\]
Then  $F_{\sigma,\nu}^{\chi}$ has holomorphic continuation to all $\nu\in (\mathfrak{a}_{1})_{\mathbb{C}}'$.
\end{corollary}

\begin{proof}
Let 
$$
\alpha:V_{\sigma}\longrightarrow I_{\xi,\lambda}^{\infty}
$$
and
\[
 \alpha^{T}: I_{\check{\xi},-\lambda}^{\infty} \longrightarrow V_{\check{\sigma}}.
\]
be as before. Observe that $\alpha$ induces an intertwining map
\[
 \tilde{\alpha}:I_{\sigma,\nu}^{\infty}\longrightarrow I_{\pi_{\xi,\lambda},\nu}^{\infty},
\]
defined by $\tilde{\alpha}(f)(g)=\alpha(f(g))$, for $f\in I_{\sigma,\nu}^{\infty}$, and a corresponding surjective map
\[
 \tilde{\alpha}^{T}:I_{\pi_{\check{\xi},-\lambda},-\nu}^{\infty}\longrightarrow I_{\check{\sigma},-\nu}^{\infty}.
\]
Let $\phi\in I_{\pi_{\check{\xi},-\lambda},-\nu}^{\infty}$, $v\in V_{\sigma}$. Then from lemma \ref{lemma:wchiforsmallinduced}
\begin{eqnarray*}
 W^{\chi_{1}}\circ J_{\check{\sigma},-\nu}^{\overline{\chi}}(\tilde{\alpha}^{T}(\phi)))(v)& = &W^{\chi_{1}}(\alpha^{T}(J_{\pi_{\check{\xi},-\lambda},-\nu}^{\overline{\chi}}(\phi))(v)\\
 & = &  (J_{\check{\xi},-\lambda}^{\overline{\chi_{1}}}( J_{\pi_{\check{\xi},-\lambda},-\nu}^{\overline{\chi}}(\phi)),J_{\xi,\lambda}^{\chi_{1}}(f)),
\end{eqnarray*}
where $f=\alpha(v)$. Now, if we use the natural identification $I_{\pi_{\xi,\lambda},\nu} \cong I_{\xi,\lambda+\nu}$, then the above equation becomes
\[
 F_{\sigma,\nu}^{\chi}(\tilde{\alpha}(\phi))= (J_{\check{\xi},-\lambda-\nu}^{\overline{\chi}}(\phi),J_{\xi,\lambda}^{\chi_{1}}(f)). \label{eq:Fsigmanucontinuation}
\]
Observe that the right hand side has holomorphic continuation to all $\nu$.
\end{proof}

\begin{dfn}
 Define
\[
 W_{\sigma,\nu}^{\chi}: I_{\check{\sigma},-\nu}^{\infty} \longrightarrow \operatorname{Ind}_{M_{\chi_{1}}}^{M_{\chi}}W^{\chi_{1}}(V_{\check{\sigma}})
\]
by
\[
 W_{\sigma,\nu}^{\chi}(\phi)(m)=F_{\sigma,\nu}^{\chi}(\pi(m)\phi).
\]
\end{dfn}
Observe that if $\phi\in U_{\check{\sigma},-\nu}$, then this definition is consistent with the previous definition of $W_{\sigma,\nu}^{\chi}$.  Also observe that
\[
 (W_{\sigma,\nu}^{\chi})^{T}: (\operatorname{Ind}_{M_{\chi_{1}}}^{M_{\chi}}W^{\chi_{1}}(V_{\check{\sigma}}))' \longrightarrow (I_{\check{\sigma},-\nu}^{\infty})'
\]
is injective. Furthermore, $(W_{\sigma,\nu}^{\chi})^{T}((\operatorname{Ind}_{M_{\chi_{1}}}^{M_{\chi}}W^{\chi_{1}}(V_{\check{\sigma}}))')\subset Wh_{\overline{\chi}}(I_{\check{\sigma},-\nu}^{\infty})$.

\begin{proposition} \label{prop:Pdecompositioncompactstabilizer}
The map $\phi\mapsto \sum_{\chi}f_{\phi}^{\chi}$ extends to a $P$-equivariant isometry between $I_{\check{\sigma},-\nu}$ and $\sum_{\chi} \operatorname{Ind}_{M_{\chi_{1}}N}^{P} W_{\chi_{1}}(H_{\sigma})$.
\end{proposition}

\begin{proof}
 The proof is completely analogous to the proof of \ref{prop:descompositionparabolic}. 
\end{proof}

We are now ready to state the main result of this section. Let $P_{0}$, $P$ and $P_{1}$ be as before. Let $(\sigma,H_{\sigma})$ a square integrable, Hilbert representation of $M_{1}$, and set
\[
 I_{\sigma}^{\infty}=\left\{f:K \longrightarrow V_{\sigma} \, \left| \,\begin{array}{c} \mbox{$f$ is $C^{\infty}$, and $f(mk)=\sigma(m)f(k)$}\\ \mbox{ for all  $m\in K_{M_{1}}:= K\cap M_{1}$}  \end{array} \right.\right\}.
\]
Observe that as a $K$-module $I_{\sigma}^{\infty}\cong I_{\sigma,\nu}^{\infty}$ for all $\nu\in (\mathfrak{a}_{1})_{\mathbb{C}}'$. Or, in other words, there is a family of representations $(\pi_{\nu}, I_{\sigma}^{\infty})$, $\nu\in (\mathfrak{a}_{1})_{\mathbb{C}}'$ with the same underlying representation space. 

Let $\phi\in I_{\check{\sigma}}$, $f\in I_{\sigma}$, and let $\alpha:i\mathfrak{a}_{1}' \longrightarrow \mathbb{C}$ be a smooth compactly supported function. For any $g\in G$ we will set
\[
 \Psi(f,\phi,\alpha)(g)=\int_{\mathfrak{a}_{1}'} \alpha(\nu) (\phi,\pi_{\nu}(g)f)  \, d\mu_{\sigma}(\nu)
\]
where $\mu_{\sigma}$ is the Plancherel measure. Then $\Psi(f,\phi,\alpha)\in \C(G)$ \cite[thm 12.7.7]{w:vol2}.
\begin{proposition}\label{prop:Fouriertransformofpacket}
 If $\Psi(f,\phi,\alpha)\in \C(G)$ is defined as above, then
\[
 \int_{N}\chi(n)\Psi(f,\phi,\alpha)(n^{-1})\, dn= \int_{\mathfrak{a}_{1}'} \alpha(\nu) (W_{\check{\sigma},-\nu}^{\overline{\chi}})^{T}(W_{\sigma,\nu}^{\chi}(\phi))(f) \, d\mu_{\sigma}(\nu).
\]
\end{proposition}

\begin{proof}
Given $\lambda_{1}$, $\lambda_{2}\in \mathfrak{a}_{1}'$, set
\begin{eqnarray*}
 \lefteqn{\Phi(\phi,f,\nu,\lambda_{1},\lambda_{2})} \\ 
&= & \int_{N}\int_{\mathfrak{a}_{1}'} \int_{N_{N_{1}}} \int_{M_{\chi_{1}}\backslash M_{\chi}} \chi(n)^{-1}(\phi_{-\nu+\lambda_{1}}(n_{1}m),f_{\nu+\lambda_{2}}(n_{1}mn)) \, dm\,dn_{1}\, d\mu_{\sigma}(\nu)\, dn.
\end{eqnarray*}
Observe that if $\lambda_{i}(\alpha_{j}) \leq 0$ for all $i=1,2$, $\alpha_{j}\in\Phi(P_{1},A_{1})$, then $\Phi(\phi,f,\nu,\lambda_{1},\lambda_{2})<\infty$, and this integral defines a continuous function that is real analytic if $\lambda_{i}(\alpha_{j}) < 0$ for all $i=1,2$, $\alpha_{j}\in\Phi(P_{1},A_{1})$ (recall that we are inducing from the parabolic opposite to $P_{1}$). Also observe that equation (\ref{eq:mchiinnerproduct}) implies that
\[
 \Phi(\phi,f,\nu,0,0)= \int_{N}\chi(n)\int_{\mathfrak{a}_{1}'} \alpha(\nu) (\phi,\pi_{\nu}(n)^{-1}f).
\]
On the other hand set 
\[
 \tilde{\Phi}(\phi,f,\nu,\lambda_{1},\lambda_{2})= \int_{\mathfrak{a}_{1}'} \alpha(\nu) (W_{\check{\sigma},-\nu+\lambda_{1}}^{\overline{\chi}})^{T}\circ W_{\sigma,\nu+\lambda_{2}}^{\chi}(\phi)(f) \, \mu(\sigma,\nu) \, d\nu.
\]
Then $\tilde{\Phi}$ is analytic for $\lambda_{1}$, $\lambda_{2}$.

If $\lambda_{i}(\alpha_{j})  \ll 0$, for all $i=1,2$, $\alpha_{j}\in\Phi(P_{1},A_{1})$, then the following integrals are absolutely convergent, so we can change the order of integration in the next series of equations.
\begin{eqnarray*}
 \lefteqn{\Phi(\phi,f,\nu,\lambda_{1},\lambda_{2})}\\
& = & \int_{N}\int_{\mathfrak{a}_{1}'} \int_{N_{N_{1}}} \int_{M_{\chi_{1}}\backslash M_{\chi}} \chi(n)^{-1}(\phi_{-\nu+\lambda_{1}}(n_{1}m),f_{\nu+\lambda_{2}}(n_{1}mn)) \, dm\,dn_{1}\, d\mu_{\sigma}(\nu)\, dn \\
& = & \int_{\mathfrak{a}_{1}'}\int_{M_{\chi_{1}}\backslash M_{\chi}} \int_{N_{N_{1}}} \int_{N} \chi(n)^{-1}(\phi_{-\nu+\lambda_{1}}(n_{1}m),f_{\nu+\lambda_{2}}(n_{1}nm))  \, dn\,dn_{1}\, dm\, d\mu_{\sigma}(\nu) \\
& = & \int_{\mathfrak{a}_{1}'}\int_{M_{\chi_{1}}\backslash M_{\chi}} \int_{N_{N_{1}}} \int_{N} \chi(n_{1}^{-1}n)^{-1}(\phi_{-\nu+\lambda_{1}}(n_{1}m),f_{\nu+\lambda_{2}}(nm))  \, dn\,dn_{1}\, dm\, d\mu_{\sigma}(\nu) \\
&  = & \int_{\mathfrak{a}_{1}'} \int_{M_{\chi_{1}}\backslash M_{\chi}}\\  & & \quad  \int_{N_{M}} \chi(n_{m})^{-1} (J_{\check{\sigma},-\nu+\lambda_{1}}^{\overline{\chi}}(\pi(m)\phi),\sigma(n_{m})J_{\sigma,\nu+\lambda_{2}}^{\chi}(\pi(m)f))\alpha(\nu)\, dn_{m}\, dm\, d\mu_{\sigma}(\nu) \\
&  = & \int_{N}\int_{\mathfrak{a}_{1}'} \int_{N_{N_{1}}}\\   & &  \int_{M_{\chi_{1}}\backslash M_{\chi}} (W^{\chi_{1}}\circ J_{\check{\sigma},-\nu+\lambda_{1}}^{\overline{\chi}}(\pi(m)\phi),W^{\overline{\chi_{1}}}J_{\sigma,\nu+\lambda_{2}}^{\chi}(\pi(m)f))\alpha(\nu)     \, dm\,dn_{1}\, d\mu_{\sigma}(\nu) dn \\
& = & \int_{\mathfrak{a}_{1}'} \int_{M_{\chi_{1}}\backslash M_{\chi}} (W_{\sigma,\nu+\lambda_{1}}^{\chi}(\phi)(m),W_{\check{\sigma},-\nu+\lambda_{2}}^{\overline{\chi}}(f)(m)) \,  dm\, d\mu_{\sigma}(\nu) \\
& = & \int_{\mathfrak{a}_{1}'} \alpha(\nu) (W_{\check{\sigma},-\nu+\lambda_{2}}^{\overline{\chi}})^{T}\circ W_{\sigma,\nu+\lambda_{1}}^{\chi}(\phi)(f) \, d\mu_{\sigma}(\nu) \\
& = & \tilde{\Phi}(\phi,f,\nu,\lambda_{1},\lambda_{2}).
\end{eqnarray*}
Now, since $\Phi$ and $\tilde{\Phi}$ are analytic, $\Phi(\phi,f,\nu,0,0) = \tilde{\Phi}(\phi,f,\nu,0,0)$, i.e.,
\[
 \int_{N}\chi(n)\int_{\mathfrak{a}_{1}'} \alpha(\nu) \phi(\pi_{\nu}(n)^{-1}f)  \, d\mu_{\sigma}(\nu) = \int_{\mathfrak{a}_{1}'} \alpha(\nu) (W_{\check{\sigma},-\nu}^{\overline{\chi}})^{T}\circ W_{\sigma,\nu}^{\chi}(\phi)(f) \, d\mu_{\sigma}(\nu),
\]
as we wanted to show.
\end{proof}

\section{The explicit Bessel-Plancherel theorem} \label{sec:explicitbesselplancherel}

Let $G$ be a simple Lie group of tube type. Let $P_{\circ}=M_{\circ}A_{\circ}N_{\circ}$ be a minimal parabolic subgroup, and let $P=MAN$ be a Siegel parabolic subgroup dominating $P_{\circ}$. Let $P_{1}=M_{1}A_{1}N_{1}$ be another parabolic subgroup dominating $P_{\circ}$. Let $\chi$ be a generic character of $N$ whose stabilizer in $M$, $M_{\chi}$, is compact. We will set
\[
 \begin{array}{c}  N_{M_{1}}=N\cap M_{1}, \quad M_{M_{1}}= M\cap M_{1}, \quad A_{M_{1}}=A\cap M_{1}, \\
 N_{N_{1}}=N\cap N_{1}, \quad M_{\chi_{1}}=M_{\chi}\cap P_{1}. 
 \end{array} 
\]
The purpose of this section is to prove the following theorem:
\begin{theorem} With notation as above,
\begin{enumerate}
 \item Let $(\sigma, V_{\sigma})$ be a square integrable representation of $M_{1}$, and let $\nu \in (\mathfrak{a}_{1})_{\mathbb{C}}'$. If $(\tau, H_{\tau})$ is an irreducible representation of $M_{\chi}$, then $ Wh_{\chi}(I_{\sigma,\nu}^{\infty})(\check{\tau})$ is finite dimensional, and for all $\nu \in (\mathfrak{a}_{1})_{\mathbb{C}}'$ there exists an isomorphism
 \[
j_{\sigma,\nu}^{\chi,\tau}:Wh_{\chi_{1},\tau_{1}}(V_{\sigma})\otimes H_{\check{\tau}}\longrightarrow Wh_{\chi}(I_{\sigma,\nu}^{\infty})(\check{\tau}),
\]
where $\chi_{1}=\chi|_{N_{M_{1}}}$, $\tau_{1}=\tau|_{M_{\chi_{1}}}$.
\item As an $M_{\chi}N\times G$ representation
 \[
  L^2(N\backslash G) \cong \bigoplus_{\tau\in \check{M}_{\chi}} \bigoplus_{(P_{1},A_{1})\succ (P_{\circ},A_{\circ})} \bigoplus_{\sigma\in\mathcal{E}_{2}(M_{1})} \int_{i(\mathfrak{a}_{1}')^{+}} Wh_{\chi_{1},\tau_{1}}(V_{\sigma})\otimes H_{\check{\tau}}\otimes I_{\sigma,\nu}\, d\mu_{\sigma}(\nu).
 \]
where $\mu$ is the usual Plancherel measure for $G$.
\item Given $\alpha \in C_{c}^{\infty}(i\mathfrak{a}_{1}';Wh_{\chi_{1},\tau_{1}}(V_{\sigma})\otimes H_{\check{\tau}})$, and $f\in I_{\sigma}^{\infty}$, define
\[
 c_{\alpha,f}(g)=\int_{i\mathfrak{a}_{1}'} j_{\sigma,\nu}^{\chi,\tau}(\alpha(\nu))(\pi_{v}(g)f_{v})d\mu_{\sigma}(\nu).
\]
Then $c_{\alpha,f} \in \C(N\backslash G;\chi).$
\end{enumerate}
\end{theorem}
 
We call this the \emph{explicit Bessel-Plancherel theorem} as here, unlike the case for the generalized Bessel-Plancherel theorem, the multiplicities appearing in the decomposition are associated with finite dimensional vector spaces of interest in their own, and we also have explicit intertwiner operators.

We will start this section by stating a version of Frobenius reciprocity on which we will rely for the rest of the section. Although this result is well known, we will include a proof of it here as it will be useful to have at hand the explicit formulas for the isomorphism.

 Let $(\tau, V_{\tau})$ be an irreducible representation of a compact group $K$, and let $(\sigma, V_{\sigma})$ be a smooth, Fr\'echet representation of a subgroup $M\subset K$. Let
\[
 I^{\infty}_{\sigma}=\{f:K\longrightarrow V_{\sigma} \, | \, \mbox{$f$ is smooth, and $f(mk)=\sigma(m)f(k)$ for all $m\in M$}\}.
\]
We define a smooth Fr\'echet representation, $(I^{\infty}_{\sigma},\pi)$, of $K$ by setting $\pi(\bar{k})f(k)=f(k\bar{k})$ for all $k,\bar{k} \in K$. Then we have the following result
\begin{lemma}\label{lemma:Frobenius reciprocity}
 With notation as above, there exists a canonical isomorphism
\[
 Hom_{K}(I^{\infty}_{\sigma},V_{\tau}) \cong Hom_{M}(V_{\sigma},V_{\tau}).
\]
\end{lemma}
\begin{proof}
 Part 2 of theorem \ref{thm:kolk-varadarajan}  says that, given $\lambda \in Hom_{K}(I^{\infty}_{\sigma},V_{\tau})$, there exists a unique $\mu_{\lambda}\in Hom_{M}(V_{\sigma},V_{\tau})$, such that
\begin{equation}
 \lambda(f)=\int_{M\backslash K} \tau(k)^{-1}\mu_{\lambda}(f(k))\, dk, \label{eq:compactKolkVaradarajan}
\end{equation}
and the map $\lambda \mapsto \mu_{\lambda}$ defines a linear isomorphism between $Hom_{K}(I^{\infty}_{\sigma},V_{\tau})$ and $Hom_{M}(V_{\sigma},V_{\tau})$. This is enough to prove the lemma, but we would like to give a more explicit description of $\mu_{\lambda}$. What it's clear is that, if for any $\mu\in Hom_{M}(V_{\sigma},V_{\tau})$ and $f\in I^{\infty}_{\sigma}$, we set
\begin{equation}
 \lambda_{\mu}(f)=\int_{M\backslash K} \tau(k)^{-1}\mu(f(k))\, dk,
\end{equation}
then $\lambda_{\mu_{\lambda}}=\lambda$ and $\mu_{\lambda_{\mu}}=\mu$.

Given $v\in V_{\sigma}$, set
\[
 \chi_{\tau,v}(k)=\int_{M}\chi_{\tau}(mk)\sigma(m)^{-1}v\,dm.
\]
It's straightforward to check that $\chi_{\tau,v}\in I^{\infty}_{\sigma}$. Furthermore, for all $\bar{m}\in M$
\begin{eqnarray*}
 \chi_{\tau,\sigma(\bar{m})v}(k) & = & \int_{M}\chi_{\tau}(mk)\sigma(m)^{-1}\sigma(\bar{m})v\,dm\\
 & = & \int_{M}\chi_{\tau}(\bar{m}mk)\sigma(m)^{-1}v\,dm= \int_{M}\chi_{\tau}(mk\bar{m})\sigma(m)^{-1}v\,dm\\
& = & \chi_{\tau,v}(k\bar{m})=(\pi(\bar{m})\chi_{\tau,v})(k).
\end{eqnarray*}
Let $d_{\tau}$ be the dimension of $V_{\tau}$. Then from equation (\ref{eq:compactKolkVaradarajan}),
\begin{eqnarray*}
 \lambda(d_{\tau}\chi_{\tau,v}) & = &\int_{M\backslash K} \tau(k)^{-1}\mu_{\lambda}(d_{\tau}\chi_{\tau,v}(k))\, dk \\
& = & \int_{M\backslash K} d_{\tau} \tau(k)^{-1}\mu_{\lambda}(\int_{M}\chi_{\tau}(mk)\sigma(m)^{-1}v\,dm)\, dk \\
& = & \int_{M\backslash K}\int_{M} d_{\tau}\chi_{\tau}(mk) \tau(k)^{-1}\mu_{\lambda}(\sigma(m)^{-1}v)\,dm\, dk \\ 
& = & \int_{M\backslash K}\int_{M} d_{\tau}\chi_{\tau}(mk) \tau(mk)^{-1}\mu_{\lambda}(v)\,dm\, dk \\ 
& = & \int_{K} d_{\tau}\chi_{\tau}(k) \tau(k)^{-1}\mu_{\lambda}(v)\, dk= \mu_{\lambda}(v),
\end{eqnarray*}
that is
\begin{equation}
 \mu_{\lambda}(v)=\lambda(d_{\tau}\chi_{\tau,v}).
\end{equation}
This is the formula that we wanted to obtain.
\end{proof}

\begin{proposition}\label{prop:forwardequality}
 If $(\sigma, H_{\sigma})$ is a square integrable, Hilbert representation of $M_{1}$, then $Wh_{\chi_{1},\tau_{1}}(V_{\sigma})=W_{\chi,\tau}(I_{\sigma,\nu})$.
\end{proposition}

\begin{proof}
By Proposition \ref{prop:squareintegrablemultiplicity}, if $(\pi,H_{\pi})$ is a square integrable, Hilbert representation of $G$, then
\begin{equation}
 W_{\chi}(H_{\pi})(\check{\tau})=Wh_{\chi}(V_{\pi})(\check{\tau}). \label{eq:WequalWhforsquareintegrable}
\end{equation}
On the other hand, by proposition  \ref{prop:Pdecompositioncompactstabilizer}
\[
 W_{\chi}(I_{\sigma,\nu}) \cong \operatorname{Ind}_{M_{\chi_{1}}}^{M_{\chi}}W_{\chi_{1}}(H_{\sigma}).
\]
Let $\check{W}_{\chi_{1}}(H_{\sigma})$ be the dual of $W_{\chi_{1}}(H_{\sigma})$. Then, by Frobenius reciprocity,
\begin{eqnarray*}
 W_{\chi}(I_{\sigma,\nu})(\check{\tau})& \cong &  H_{\check{\tau}}\otimes Hom_{M_{\chi}}(\operatorname{Ind}_{M_{\chi_{1}}}^{M_{\chi}}\check{W}_{\chi_{1}}(H_{\sigma}),H_{\tau}) \\
& \cong & H_{\check{\tau}}\otimes Hom_{M_{\chi_{1}}}(\check{W}_{\chi_{1}}(H_{\sigma}), H_{\tau}) \\
& \cong & H_{\check{\tau}}\otimes Wh_{\chi_{1},\tau_{1}}(V_{\sigma}),
\end{eqnarray*}
where the last equations follows from (\ref{eq:WequalWhforsquareintegrable}) and the definition of $Wh_{\chi_{1},\tau_{1}}(V_{\sigma})$.
\end{proof}

We will now want to show that $Wh_{\chi,\tau}(I_{\sigma,\nu}^{\infty})\cong  Wh_{\chi_{1},\tau_{1}}(V_{\sigma})$ for all $\nu$. For this we will first need the following lemma.

\begin{lemma}\label{lemma:snake}%(this lemma is also useful to prove that closed subspace of a reflexive space is reflexive)
 Consider the following commutative diagram:
\[
 \begin{array}{crcrcrccc}
  0 & \longrightarrow & U & \stackrel{p_{1}}{\longrightarrow} & V & \stackrel{p_{2}}{\longrightarrow} & W & \longrightarrow & 0 \\
    &                    &              &                   S &   \downarrow&                          T & \downarrow &  & \\
0 & \longrightarrow & U' & \stackrel{p'_{1}}{\longrightarrow} & V' & \stackrel{p'_{2}}{\longrightarrow} & W' & \longrightarrow & 0,
 \end{array}
\]
where the two rows are short exact sequences. If $S$ is an isomorphism, and $T$ is injective, then $T$ is an isomorphism. Furthermore there exists an isomorphism $R: U\longrightarrow U'$ that makes the whole diagram commute.
\end{lemma}
\begin{proof}
The proof is a classical diagram chasing argument, and it's provided below, but the reader my want to amuse himself and do the diagram chasing by his own.
 We will first show that $T$ is surjective. Let $w'\in W'$. Since $p'_{2}$ is surjective, there exists $v'\in V'$ such that $w'=p'_{2}(v')$. Set $w=p_{2}\circ S^{-1}(v')\in W$. Then
\[
 T(w)=T\circ p_{2}\circ S^{-1}(v')=p'_{2}(v')=w'.
\]
Since $T$ was already injective by hypothesis, we conclude that $T$ is an isomorphism.

Now let $u\in U$ and observe that
\[
 p'_{2}(S(p_{1}(u)))=T(p_{2}( p_{1}(u)))=0,
\]
since $p_{2}\circ p_{1}=0$. But now by the exactness of the bottom row, there exists a unique $u'\in U'$ such that $p'_{1}(u')=S(p_{1}(u))$. Set $R(u)=u'$. It's easy to check that this defines a linear map between $U$ and $U'$ that makes the diagram commute.

Let $u\in U$ be such that $R(u)=0$. Then $S^{-1}\circ p'_{1}\circ R(u)=0=p_{1}(u)$. Since $p_{1}$ is injective, this implies that $u=0$. Since the only condition in $u$ was that $R(u)=0$ we conclude that $R$ is injective.

Now let $u'\in U'$. Then $p'_{2}\circ p'_{1}(u')=(T\circ p_{2} \circ S^{-1})(p'_{1}(u'))=0$. Now since $T$ is an isomorphism, we conclude that $p_{2}(S^{-1}\circ p'_{1}(u'))=0$. Hence there exists a unique $u\in U$ such that $p_{1}(u) = S^{-1}\circ p'_{1}(u')$, therefore
\[
 p'_{1}(u')=S\circ p_{1}(u)=p'_{1}(R(u)).
\]
Now since $p'_{1}$ is injective we conclude that $R(u)=u'$, finishing the proof of the surjectivity of $R$ and finishing the proof of the lemma.
\end{proof}

Let $(\sigma, H_{\sigma})$ be an irreducible, square integrable, Hilbert representation of $M_{1}$, and let $\nu\in (\mathfrak{a}_{1})_{\mathbb{C}}'$. Let $(\tau,H_{\tau})$ be an irreducible representation of $M_{\chi}$. Given $\mu \in Wh_{\chi_{1},\tau_{1}}(V_{\sigma})$ and $f\in I_{\sigma,\nu}^{\infty}$, define
\[
 j_{\sigma,\nu}^{\chi,\tau}(\mu)(f)= \int_{N_{N_{1}}}\int_{M_{\chi_{1}}\backslash M_{\chi}} \chi(n)^{-1}\tau(m)^{-1}\mu(f(nm))\,dn\,dm.
\]

\begin{lemma}
 Let $\Phi(P_{1},A_{1})^{+}$ be the system of positive roots of $A_{1}$ induced by $P_{1}$. Let $B$ denote the Cartan-Killing form on $\mathfrak{g}_{\mathbb{C}}'$. Let $\mu \in Wh_{\chi_{1},\tau_{1}}(V_{\sigma})$ and $f\in I_{\sigma,\nu}^{\infty}$. If $\operatorname{Re} B(\nu,\alpha) \gg 0$ for all $\alpha \in \Phi(P_{1},A_{1})^{+}$, then the integral defining $j_{\sigma,\nu}^{\chi,\tau}(\mu)(f)$ converges absolutely.
\end{lemma}

\begin{proof}
 The proof of this lemma is identical to the proof of lemma \ref{lemma:hyperplane}
\end{proof}

\begin{proposition}\label{thm:adjointholomorphiccontinuation}
Let $\mu \in Wh_{\chi_{1},\tau_{1}}(V_{\sigma})$ and $f\in I_{\sigma,\nu}^{\infty}$. The map $\nu\mapsto j_{\sigma,\nu}^{\chi,\tau}(\mu)(f)$ extends to a holomorphic map from $(\mathfrak{a}_{1})_{\mathbb{C}}'$ to $\mathbb{C}$. Furthermore
\[
 j_{\sigma,\nu}^{\chi,\tau}:Wh_{\chi_{1},\tau_{1}}(V_{\sigma})\longrightarrow Wh_{\chi,\tau}(I_{\sigma,\nu}^{\infty})
\]
is a linear bijection for all $\nu \in (\mathfrak{a}_{1})_{\mathbb{C}}'$.
\end{proposition}

\begin{proof}
Let $Q=P_{\circ}\cap M_{1}$ be a minimal parabolic subgroup of $M_{1}$. Let $(\xi,V_{\xi})$ be an irreducible finite dimensional representation of $M_{0}$, and let $\delta \in (\mathfrak{a}_{0}\cap \mathfrak{m}_{1})$ be such that there exists a surjective map from $I^{\infty}_{Q,\xi,\delta}$ onto $V_{\sigma}$. Here $I^{\infty}_{Q,\xi,\delta}=I^{\infty}_{\xi,\delta}$ but we are including the parabolic subgroup from which we are inducing to avoid confusion with the several induced representations that we will use in this proof. Let $W$ denote the kernel of this map. If $\nu \in (\mathfrak{a}_{1})_{\mathbb{C}}'$, then we have the exact sequence
\[
0\longrightarrow  I^{\infty}_{\bar{P_{1}},\pi_{Q,\xi,\delta}|_{W},\nu} \longrightarrow I^{\infty}_{\bar{P_{1}},\pi_{Q,\xi,\delta},\nu} \longrightarrow I^{\infty}_{\bar{P_{1}},\sigma,\nu} \longrightarrow 0.                                                                                                                                                                                                                                                                                                                                                                                                                                                                                                                                                                                                                                                           \]
In this sequence the first arrow is given by the obvious homomorphism $S_{1}$, given by $S_{1}(f)(k)=f(k)$, since $W\subset I^{\infty}_{Q,\xi,\delta}$, and the second arrow is given by $S_{2}(f)(k)=S(f(k))$. The point is that the total spaces and $S_{1}$, $S_{2}$ are independent of $\nu$. we therefore have the exact sequence
\begin{equation}
 0\longrightarrow Wh_{\chi,\tau}(I^{\infty}_{\bar{P_{1}},\sigma,\nu}) \stackrel{S_{2}^{T}}{\longrightarrow} Wh_{\chi,\tau}(I^{\infty}_{\bar{P_{1}},\pi_{Q,\xi,\delta},\nu}) \stackrel{S_{1}^{T}}{\longrightarrow} Wh_{\chi,\tau}(I^{\infty}_{\bar{P_{1}},\pi_{Q,\xi,\delta},\nu})|_{I^{\infty}_{\bar{P_{1}},\mu,\nu}} \longrightarrow 0.
\end{equation}
We also have
\begin{equation}
 0\longrightarrow Wh_{\chi_{1},\tau_{1}}(V_{\sigma}) \stackrel{S^{T}}{\longrightarrow} Wh_{\chi_{1},\tau_{1}}(I^{\infty}_{Q,\xi,\delta}) \longrightarrow Wh_{\chi_{1},\tau_{1}}(I^{\infty}_{Q,\xi,\delta})|_{W} 
\longrightarrow 0.
\end{equation}
To simplify notation we will denote $\pi_{Q,\xi,\delta}$ by $\eta$. Then we have that
\[
 Wh_{\chi,\tau}(I^{\infty}_{\bar{P_{1}},\pi_{Q,\xi,\delta},\nu})\cong Hom_{M_{0}}(V_{\xi},V_{\tau}) \cong Wh_{\chi_{1},\tau_{1}}(I^{\infty}_{Q,\xi,\delta}).
\]
The isomorphism induced by this two isomorphisms is precisely $j_{\eta,\nu}=j_{\xi,\eta+\nu}\circ \Phi_{\xi,\eta}^{\chi_{1},\tau_{1}}$.

We now want to show that $j_{\eta,\nu}$ induces a well-defined injective map between $Wh_{\chi_{1},\tau_{1}}(I^{\infty}_{Q,\xi,\delta})|_{W}$ and $Wh_{\chi,\tau}(I^{\infty}_{\bar{P_{1}},\pi_{Q,\xi,\delta},\nu})|_{I^{\infty}_{\bar{P_{1}},\mu,\nu}}$. To show that the induced map is well defined, we need to show that if $\mu \in Wh_{\chi_{1},\tau_{1}}(I^{\infty}_{Q,\xi,\delta})$, and $\mu|_{W}=0$, then $j_{\eta,\nu}(\mu)|_{I^{\infty}_{\bar{P_{1}},\mu,\nu}}=0$. Let $f\in I^{\infty}_{\bar{P_{1}},\mu,\nu}$. Then the map $\nu\mapsto j_{\eta,\nu}(\mu)(f)$ is holomorphic on $\nu$. Let $\operatorname{Re} B(\nu,\alpha) \gg 0 \gg 0$. Then
\[
 j_{\eta,\nu}(\mu)(f)=\int_{N_{N_{1}}}\int_{M_{\chi_{1}}\backslash M_{\chi}} \chi(n)^{-1}\tau(m)^{-1}\mu(f(nm))\,dn\,dm = 0,
\]
since $f(nm)\in W$ for all $n\in N_{N_{1}}$, $m\in M_{\chi}$, and $\mu|_{W}=0$. Since the map $\nu\mapsto j_{\eta,\nu}(\mu)(f)$ is holomorphic on $\nu$ we conclude that $j_{\eta,\nu}(\mu)(f)=0$ for all $\nu\in (\mathfrak{a}_{1})_{\mathbb{C}}$. Since $f\in I^{\infty}_{\bar{P_{1}},\mu,\nu}$ was arbitrary, we conclude that $j_{\eta,\nu}$ induces a well-defined map between  $Wh_{\chi_{1},\tau_{1}}(I^{\infty}_{Q,\xi,\delta})|_{W}$ and $Wh_{\chi,\tau}(I^{\infty}_{\bar{P_{1}},\pi_{Q,\xi,\delta},\nu})|_{I^{\infty}_{\bar{P_{1}},\mu,\nu}}$.

Now we want to show that the map is injective. Assume that $j_{\eta,\nu}(\mu)|_{I^{\infty}_{\bar{P_{1}},\mu,\nu}}=0$. Let $w\in W$. We will define a function on $U_{\bar{P_{1}},\sigma,\nu}\cap I^{\infty}_{\bar{P_{1}},\mu,\nu}$ in the following way: given $m\in M_{\chi}$, and $n\in N_{N_{1}}$ we set
\[
f(nm)=\phi(n)\chi_{\tau,w}(m),                                                                                                                                                                                                                                                                      \]
where $\phi\in C_{c}^{\infty}(N_{N_{1}})$ is a function such that
\[
 \int_{N_{N_{1}}} \chi(n)^{-1}\phi(n) \, dn = 1.
\]
Then
\begin{eqnarray*}
 0 &  = &  j_{\eta,\nu}(\mu)(f)=\int_{N_{N_{1}}}\int_{M_{\chi_{1}}\backslash M_{\chi}} \chi(n)^{-1}\tau(m)^{-1}\mu(f(nm))\,dn\,dm \\
 & = & \int_{N_{N_{1}}}\int_{M_{\chi_{1}}\backslash M_{\chi}} \chi(n)^{-1}\tau(m)^{-1}\mu(\phi(n)\chi_{\tau,w}(m))\,dn\,dm \\
& = & \int_{M_{\chi_{1}}\backslash M_{\chi}}\tau^{-1}(m)\mu(\chi_{\tau,w}(m)) \, dm = \mu(w),
\end{eqnarray*}
according to the proof of lemma \ref{lemma:Frobenius reciprocity}. But this says that $\mu(w)=0$ for all $w\in W$, i.e., $\mu|_{W}=0$. Therefore the map $j_{\eta,\nu}$ is well defined and injective between $Wh_{\chi_{1},\tau_{1}}(I^{\infty}_{Q,\xi,\delta})|_{W}$ and $Wh_{\chi,\tau}(I^{\infty}_{\bar{P_{1}},\pi_{Q,\xi,\delta},\nu})|_{I^{\infty}_{\bar{P_{1}},\mu,\nu}}$. Therefore we are in the situation of lemma \ref{lemma:snake}, and hence we can define an isomorphism $j_{\sigma,\nu}^{\chi,\tau}$ from $Wh_{\chi_{1},\tau_{1}}(V_{\sigma})$ to $Wh_{\chi,\tau}(I^{\infty}_{\bar{P_{1}},\sigma,\nu})$ such that if $\operatorname{Re} B(\nu,\alpha) \gg 0$ for all alpha in $\Phi(P_{1},A_{1})^{+}$ or $f\in U_{\bar{P_{1}},\sigma,\nu}$, then
\[
 j_{\sigma,\nu}^{\chi,\tau}(\mu)(f)= \int_{N_{N_{1}}}\int_{M_{\chi_{1}}\backslash M_{\chi}} \chi(n)^{-1}\tau(m)^{-1}\mu(f(nm))\,dn\,dm,
\]
and the map $\nu \mapsto j_{\sigma,\nu}^{\chi,\tau}(\mu)(f)$ is holomorphic in $\nu$.
\end{proof}

To prove the theorem given at the beginning of this section we just have to put together all the results we have obtained. More concretely,
part 1 is just a restatement of proposition \ref{thm:adjointholomorphiccontinuation}. Part 2 follows from part 1, the generalized Bessel-Plancherel theorem and propositions \ref{prop:forwardequality} and \ref{thm:adjointholomorphiccontinuation}. Finally, part 3 is just a restatement of proposition \ref{prop:Fouriertransformofpacket}.

\begin{subappendices}
\section{Irreducible representations of Siegel Parabolic Subgroups} \label{sec:representationparabolics}
\begin{theorem}
 Let $P$ be a Siegel parabolic subgroup of a Lie group $G$, and let $P=MAN$ be its Langlands decomposition. If $(\pi,H)$ is an irreducible unitary representation of $P$, then
\[
 H\cong \operatorname{Ind}_{M_{\chi}  N}^{P}\tau\chi \qquad \mbox{with $\tau \in \hat{M_{\chi}}$, $\chi \in \hat{N}$}.
\]
\end{theorem}
\begin{proof}
 As an N-module, we have that
\[
 H\cong \int_{\hat{N}}E_{\chi}\, d_{\nu}(\chi),
\]
where $E_{\chi}\cong L_{\chi}\otimes V_{\chi}$, $V_{\chi}\in \hat{N}$, and $L_{\chi}$ is a multiplicity space. This means that there exists a vector bundle
\[
 \begin{array}{c}
  E \\ \downarrow \\ \hat{N}
 \end{array}
\]
and a measure $\nu$ on $\hat{N}$, such that
\[
 H\cong L^2(\hat{N},E,\nu):=\{s:\hat{N}\rightarrow E\, | \, s(\chi)\in E_{\chi}, \int_{\hat{N}} \|s(\chi)\|^{2}\, d\nu(\chi) < \infty\}
\]
under the action
\[
 (\pi(n)\cdot s)(\chi)=\chi(n)s(\chi).
\]
Under this isomorphism we can extend this action of $N$ on $L^2(\hat{N},E,\nu)$ to an action of $P$ on the same space.

Let $m\in M$, and define
\[
 \begin{array}{c}
  E^{m} \\ \downarrow \\ \hat{N}
 \end{array}
\]
to be the vector bundle such that $E_{\chi}^{m}=E_{m\cdot \chi}$. Define a measure $\nu_{m}$ on $\hat{N}$ by
\[
 \nu_{m}(X)=\nu(m\cdot X) \qquad \mbox{for $X\subset \hat{N}$ a measurable set}
\]
and define
\[
 \tau(m):L^2(\hat{N},E,\nu) \longrightarrow L^2(\hat{N},E^{m},\nu_{m})
\]
by
\[
 (\tau(m)s)(\chi)=(\phi(m)s)(m\cdot \chi).
\]
We claim that $\tau(m)$ is an isometry. Effectively
\begin{eqnarray*}
 \|\tau(m)s\|_{m}^{2} & = & \int_{\hat{N}}\|(\tau(m)s)(\chi)\|^{2}\, d{\nu}_{m}(\chi) \\
                      & = & \int_{\hat{N}}\|(\pi(m)s)(m\cdot \chi)\|^{2}\, d_{\nu}(m\cdot \chi) \\
                      & = & \int_{\hat{N}}\|(\pi(m)s)(\chi)\|^{2}\, d_{\nu}(\chi)\\
                      & = & \|\pi(m)s)\|^{2}=\|s\|^{2},
\end{eqnarray*}
where the last equality comes from the fact that the action of $P$ is unitary. Now if we define an action of $N$ on $L^2(\hat{N},E^{m},\nu_{m})$ by
\[
 (\pi_{m}\cdot s)(\chi)=\chi(n)s(\chi),
\]
then $\tau(m)$ becomes an $N$-intertwiner,that is,
\begin{eqnarray*}
 \tau(m)(\pi(n)s)(\chi)& = & \pi(m)\pi(n)s(m\cdot \chi) \\
                       & = & \pi(mnm^{-1})(\pi(m)s)(m\cdot \chi) \\
                       & = & (m\cdot \chi)(mnm^{-1})(\pi(m)s)(m\cdot \chi) \\
                       & = & \chi(m^{-1}mnm^{-1}m)(\pi(m)s)(m\cdot \chi)\\
                       & = & \chi(n)(\tau(m)s)(\chi)=(\pi_{m}(n)\tau(m)s)(\chi).
\end{eqnarray*}
But now since $N$ is a CCR group the $N$-interwiner
\[
 \tau(m):L^2(\hat{N},E,\nu) \longrightarrow L^2(\hat{N},E^{m},\nu_{m})
\]
should come from a morphism of vector bundles
\[
 \tilde{\tau}(m):E\longrightarrow E^{m},
\]
that is, $(\tau(m)s)(\chi)=\tilde{\tau}(m)s(\chi)$, and hence
\begin{eqnarray*}
 (\tau(m)s)(\chi)& = & \tilde{\tau}(m)s(\chi) \\
 (\pi(m)s)(m\cdot \chi) & = &\tilde{\tau}(m)s(\chi),
\end{eqnarray*}
which says that
\[
 (\pi(m)s)(\chi)=\tilde{\tau}(m)s(m^{-1} \cdot \chi).
\]
Now since $L^2(\hat{N},E,\nu)$ is irreducible as a representation of $P$, the support of $\nu$ should be contained in a unique $P$-orbit on $\hat{N}$, and hence
\[
 L^2(\hat{N},E,\nu) \cong L^2(MA/M_{\chi},E) \cong \operatorname{Ind}_{M_{\chi}  N}^{P} E_{\chi}.
\]
Using again that $L^2(\hat{N},E,\nu)$ is irreducible we conclude that $E_{\chi}\cong \tau\chi$ with $\tau \in \hat{M_{\chi}}$, $\chi \in \hat{N}$. Putting all of this together we get that
\[
 H\cong \operatorname{Ind}_{M_{\chi}  N}^{P} \tau\chi
\]
 as we wanted to show.
\end{proof}

\section{Decomposition of $L^2(P,d_{r}p)$ under the action of $P\times P$} \label{sec:plancherelparabolic}

We will now decompose $L^2(P,d_{r}p)$ under the action of $P\times P$ given by
\[
 (p_{1},p_{2})\cdot f=\delta(p_{1})^{-1}L_{p_{1}}R_{p_{2}}f.
\]
As a left $N$-module
\begin{eqnarray*}
 L^{2}(P) & \cong & \operatorname{Ind}_{N}^{P}\operatorname{Ind}_{1}^{N} 1 \cong \operatorname{Ind}_{N}^{P}(L^2(N)) \\
          & \cong & \operatorname{Ind}_{N}^{P}(\int_{\hat{N}} HS(V_{\chi}) \, d\mu(\chi) \\
          & \cong & \int_{\hat{N}}\operatorname{Ind}_{N}^{P}HS(V_{\chi}) \, d\mu(\chi) \cong L^2(\hat{N},E,\mu),
\end{eqnarray*}
with $E_{\chi}=HS(V_{\chi})$. The isomorphism is given in the following way: Given $f\in C_{c}(P)$, define $s_{f}\in L^2(\hat{N},E,\nu)$ by
\[
 s_{f}(\chi)(p)=\int_{N}\chi(n)^{-1}f(np) \, dn.
\]
Observe that $\|f\|=\|s_{f}\|$ and hence this map extends to an isometry between $L^{2}(P)$ and $L^2(\hat{N},E,\nu)$.
Furthermore
\begin{eqnarray*}
 s_{R_{p_{1}}f}(\chi)(p)&=&\int_{N}\chi(n)^{-1}R_{p_{1}}f(np) \, dn=\int_{N}\chi(n)^{-1}f(npp_{1}) \, dn\\ & =& s_{f}(\chi)(pp_{1})=(R_{p_{1}}s_{f}(\chi))(p),
\end{eqnarray*}
and
\begin{eqnarray*}
 s_{L_{p_{1}}f}(\chi)(p)& = & \int_{N}\chi(n)^{-1}\delta(p_{1})^{-1}L_{p_{1}}f(np) \, dn\\ &= &\int_{N}\chi(n)^{-1}\delta(p_{1})^{-1}f(p_{1}^{-1}np_{1}p_{1}^{-1}p) \, dn \\
                    & = & \int_{N}\chi(p_{1}np_{1}^{-1})^{-1}f(np_{1}^{-1}p) \, dn \\ &= & \int_{N}(p_{1}^{-1}\chi)(n)^{-1}f(np_{1}^{-1}p) \, dn  \\
                    & = & s_{f}(p_{1}^{-1}\chi)(p_{1}^{-1}p)=[L_{p_{1}}s_{f}(p_{1}^{-1}\chi)](p).
\end{eqnarray*}
Therefore
\begin{eqnarray*}
 L^2(\hat{N},E,\mu)& \cong & \bigoplus_{\omega\in \Omega} L^2(\omega,E,\tilde{\mu}) \\
                   & \cong & \bigoplus_{\omega\in \Omega} \operatorname{Ind}_{M_{\chi}  N \times P}^{P\times P} \operatorname{Ind}_{\Delta M_{\chi} N\times N}^{M_{\chi}N\times P}\overline{\chi}\otimes \chi  \\
                   & \cong & \bigoplus_{\omega\in \Omega} \operatorname{Ind}_{M_{\chi}  N}^{P}(\int_{\hat{M_{\chi}}}\tau^{\ast}\chi^{\ast} \otimes \operatorname{Ind}_{M_{\chi}  N}^{P}\tau\chi)\, d\nu(\tau)  \\
                   & \cong & \bigoplus_{\omega\in \Omega}\int_{\hat{M_{\chi}}} \operatorname{Ind}_{M_{\chi}  N}^{P}\tau^{\ast}\chi^{\ast} \otimes \operatorname{Ind}_{M_{\chi}  N}^{P}\tau\chi\, d\nu(\tau),
\end{eqnarray*}
where $\nu$ is the Plancherel measure of $M_{\chi}$.
\section{Temperedness of the spectrum} \label{sec:preliminary}

Let $G$ be a real reductive group and let $K$ be a maximal compact subgroup. Let $P_{\circ}=N_{\circ}A_{\circ}M_{\circ}$ be a minimal parabolic subgroup of $G$ with given Langlands decomposition and let $P=NAM$ be another parabolic subgroups dominating $P_{\circ}$, i.e., $P_{\circ} \subset P$, $N \subset N_{\circ}$, $A \subset A_{\circ}$ and $M_{\circ}\subset M$. Let $\chi$ be a unitary character of $N$ and let
\[
 L^2(N\backslash G;\chi)=\left\lbrace f:G \longrightarrow \mathbb{C} \, \left| \,\begin{array}{c}
				\mbox{$f(ng)=\chi(n)f(g)$, $\forall n\in N$, $g\in G$ } \\
                                \mbox{and $\int_{N\backslash G} |f(g)|^2 \, dNg < \infty$}
                              \end{array}  \right\rbrace\right. .
\]
The measure on $N\backslash G$ is chosen so that if $dg$ and $dn$ are some fixed invariant measures on $G$ and $N$, respectively, then 
\[
 \int_{N\backslash G}\int_{N} f(ng)\,dn \, d(Ng)=\int_{G} f(g)\, dg
\]
for all $f$ integrable on $G$. Set $C_{c}(N\backslash G;\chi )$ equal to the space of all continuous functions on $G$ such that $f(ng)=\chi(n)f(g)$ for all $n\in N$, $g\in G$ and such that $g\mapsto |f(g)|$ is in $C_c(N\backslash G)$.

\begin{lemma}
 Let $N_{M}=N_{\circ}\cap M$, $A_{M}=A_{\circ}\cap M$ and $K_{M}=K\cap M$. There is a choice of measures  $dn_{M}$, $da_{\circ}$, $da$ and $dm$, on  $N_{M}$, $A_{\circ}$, $A$ and $M$, respectively, such that if $f\in C_{c}(N \backslash G)$, then
\begin{eqnarray}
 \int_{N\backslash G} f(g)\, dNg & = & \int_{A}\int_{M}\int_{K} a^{-2\rho_{P}} f(amk)\, dk\, dm\, da \\
                                 & = & \int_{N_{M}}\int_{A_{\circ}}\int_{K} a_{\circ}^{-2\rho}f(n_{m}a_{\circ}k)dk\, da_{\circ}\, dn_{M}
\end{eqnarray}
where $\rho$ and $\rho_{P}$ are half the sum of the roots of $(P,A)$ and $(P_{\circ},A_{\circ})$, respectively.
\end{lemma}
\begin{proof}
 Let $\rho_{M}$ be equal to half the sum of the roots of $(P_{M},A_{M})$, $P_{M}=P\cap M$. The lemma follows from the integral formulas of the Iwasawa decomposition of $G$ and $M$ and from the fact that $\rho = \rho_{M}+\rho_{P}$.
\end{proof}

Given $f\in L^2(N\backslash G; \chi)$ define $(\pi_{\chi}(g)f)(x)=f(xg)$. Then $(\pi_{\chi},L^2(N\backslash G; \chi))$ is a unitary representation of $G$. We now state the main result of this section.

\begin{lemma}
 $\operatorname{supp}(\pi_{\chi})$ is contained in the tempered spectrum of $G$.
\end{lemma}

\begin{proof}
 By the arguments given in chapter 14 of \cite{w:vol2}, it suffices to show that if $f\in C_{c}(N\backslash G;\chi)$, then
\begin{equation}
 |\langle \pi_{\chi}(g)f,f\rangle | \leq C_{f}\Xi(g) \label{eq:tempinequality} 
\end{equation}
where $\Xi$ is Harish-Chandra's $\Xi$ function. Let $\gamma \in \hat{K}$ and let $C_{c}(N\backslash G;\chi)(\gamma)$ be the $\gamma$-isotypic component of $C_{c}(N\backslash G;\chi)$. Since the direct sum of the isotypic components is dense in $C_{c}(N\backslash G;\chi)$, it suffices to take $f\in C_{c}(N\backslash G;\chi)(\gamma)$. For such an $f$ define $\tilde{f}(g)=\sup\{|f(gk)|\, | \, k\in K \}$. Then $\tilde{f} \in C_{c}(N\backslash G/K)$ and
\begin{equation}
 |\langle \pi_{\chi}(g)f,f\rangle | \leq d(\gamma) |\langle \pi_{1}(g)f,f\rangle | \label{eq:tempinequalityisotypic} 
\end{equation}
with $1$ denoting the trivial character of $N$. Thus, to complete the proof we may assume that $\chi=1$ and that $f\in C_{c}(N\backslash G/K)$. With this assumptions
\begin{eqnarray}
 |\langle \pi_{1}(g)f,f\rangle | & = & \left|\int_{N\backslash G} f(xg)\overline{f(x)}\, dx \right| \nonumber \\
 & = & \left|\int_{N_{M}}\int_{A_{\circ}} \int_{K} a_{\circ}^{-2\rho}f(n_{M}a_{\circ}kg)\overline{f(n_{M}a_{\circ}k)}\, dk\, da_{\circ}\,dn_{M}\right| \nonumber  \\
& \leq & \int_{K} \left|\int_{N_{M}}\int_{A_{\circ}}a_{\circ}^{-2\rho} f(n_{M}a_{\circ}kg)\overline{f(n_{M}a_{\circ})}\, da_{\circ}\,dn_{M}\right| \, dk  \nonumber  \\
& \leq & \int_{K} \left[\int_{N_{M}}\int_{A_{\circ}} a_{\circ}^{-2\rho} |f(n_{M}a_{\circ}kg)|^{2}\, da_{\circ}\, dn_{M}\right]^{1/2} \nonumber \\ &  &\times  \left[\int_{N_{M}}\int_{A_{\circ}} a_{\circ}^{-2\rho} |f(n_{M}a_{\circ})|^{2}\, da_{\circ}\, dn_{M}\right]^{1/2} \, dk  \nonumber  \\
& \leq & \|f \| \int_{K}\left[\int_{N_{M}}\int_{A_{\circ}} a_{\circ}^{-2\rho} |f(n_{M}a_{\circ}kg)|^{2}\, da_{\circ}\, dn_{M}\right]^{1/2} \, dk.
\end{eqnarray}
We will now write $kg=n_{\circ}(kg)a_{\circ}(kg)k(kg)$, with $n_{\circ}(kg) \in N_{\circ}$, $a_{\circ}(kg)\in A_{\circ}$ and $k(kg)\in K$. Then
\begin{equation}
 f(n_{M}a_{\circ}kg)=f(n_{M}a_{\circ}n_{\circ}(kg)a_{\circ}(kg)k(kg))=f(n_{M}(a_{\circ}n_{\circ}(kg)a_{\circ}^{-1})a_{\circ}a_{\circ}(kg)), \label{eq:normal}
\end{equation}
with $a_{\circ}n_{\circ}(kg)a_{\circ}^{-1} \in N_{\circ}$. Now observe that $N_{\circ} = N \rtimes N_{M}$ and hence 
\[
a_{\circ}n_{\circ}(kg)a_{\circ}^{-1}=n(a_{\circ},k,g)n_{M}(a_{\circ},k,g) 
\]
 with $n(a_{\circ},k,g)\in N$ and $n_{M}(a_{\circ},k,g) \in N_{M}$. Plugging this into equation (\ref{eq:normal}) we get that
\begin{eqnarray*}
 f(n_{M}a_{\circ}kg) & = & f(n_{M}n(a_{\circ},k,g)n_{M}(a_{\circ},k,g)a_{\circ}a_{\circ}(kg)) \\
& = & f((n_{M}n(a_{\circ},k,g)n_{M}^{-1})n_{M}  n_{M}(a_{\circ},k,g)a_{\circ}a_{\circ}(kg)) \\
& = & f(n_{M}  n_{M}(a_{\circ},k,g)a_{\circ}a_{\circ}(kg)),
\end{eqnarray*}
where the last equality follows from the fact that $n_{M}n(a_{\circ},k,g)n_{M}^{-1} \in N$. Therefore
\begin{eqnarray}
 \lefteqn{ \int_{K}\left[\int_{N_{M}}\int_{A_{\circ}} a_{\circ}^{-2\rho} |f(n_{M}a_{\circ}kg)|^{2}\, da_{\circ}\, dn_{M}\right]^{1/2} \, dk.} \nonumber \\ 
& = & \int_{K}\left[\int_{N_{M}}\int_{A_{\circ}} a_{\circ}^{-2\rho} |f(n_{M}n_{M}(a_{\circ},k,g)a_{\circ}a_{\circ}(kg))|^{2}\, da_{\circ}\, dn_{M}\right]^{1/2} \, dk. \nonumber \\
& = & \int_{K}\left[\int_{N_{M}}\int_{A_{\circ}} a_{\circ}^{-2\rho}a_{\circ}(kg)^{2\rho} |f(n_{M}a_{\circ})|^{2}\, da_{\circ}\, dn_{M}\right]^{1/2} \, dk. \nonumber \\
& = & \int_{K} a_{\circ}(kg)^{\rho} \, dk \left[\int_{N_{M}}\int_{A_{\circ}} a_{\circ}^{-2\rho} |f(n_{M}a_{\circ})|^{2}\, da_{\circ}\, dn_{M}\right]^{1/2} \nonumber \\
& = & \Xi(g) \|f \|. \label{eq:inequalityXi}
\end{eqnarray}
The lemma now follows from (\ref{eq:tempinequality}), (\ref{eq:tempinequalityisotypic}), (\ref{eq:normal}) and (\ref{eq:inequalityXi}).
\end{proof}

\end{subappendices}
\chapter{Applications: Howe duality}

\section{Howe duality and the relative Langlands program}
Let $G$ be the set of $\k$-points of a reductive algebraic group defined over a local field $\k$. Associated to this group $G$ we can find the dual group $\check{G}$, which is a complex reductive algebraic group, and its $L$-group ${}^{L}G$, which is a semi-direct product of the absolute Galois group of $\k$ with $\check{G}$. Let $WD_{\k}$ be the Weil-Deligne group of $\k$.
In their current form, the local Langlands conjectures establish that there is a natural finite to one correspondence between the sets 
\[
 \left\lbrace \begin{array}{c} \mbox{Conjugacy classes of } \\ \mbox{$L$-parameters}\\ \phi: WD_{\k}  \longrightarrow {}^L G \end{array} \right\rbrace \longleftrightarrow \left\lbrace  \begin{array}{c} \mbox{Equivalent classes of} \\ \mbox{Irreducible smooth} \\ \mbox{representations of $G$} \end{array}
\right\rbrace. 
\]

Let ${}^L H$ be a subgroup of the $L$-group of $G$, and consider the set of $L$-parameters $\phi$ that factor through ${}^L H$. The natural question here is: What is the set of irreducible representations of $G$ associated to this $L$-parameters? The general consensus is that there should be a subgroup, $\widetilde{H} \subset G$, such that 
%How can we characterize the subset of irreducible representations parametrized by this $L$-parameters? what should we put on the right hand side of the following equation?
\[
 \left\lbrace \begin{array}{rcl}\phi: WD_{\k} & \longrightarrow & {}^L G \\
                                          \searrow            &    \bigcirc      &   \nearrow \\
                                                              & {}^L H   &
                                          \end{array}
 \right\rbrace  \longleftrightarrow   
\left\lbrace \begin{array}{c} \mbox{Irreducible representations} \\
                              \mbox{of $G$ with an $\widetilde{H}$}\\
                              \mbox{invariant functional}
 \end{array} \right\rbrace. 
\]
One can also ask the same question in the opposite direction: Given a subgroup $\widetilde{H} \subset G$, satisfying certain properties, is there an ${}^L {H} \subset {}^L G$ such that we have a correspondence like the mentioned above?

In recent years there has been a lot of progress in formalizing this ideas. For example, if $X$ is a $G$-spherical variety, then Gaitsgory and Nadler \cite{gn:2010} have defined a subgroup $\check{G} _{X}$, of the dual group $\check{G}$ of $G$, that encodes many aspects of the geometry and the representation theory of the variety $X$. This result set into motion the so~called ``relative Langlands program'', which aims to set a framework for the study of $\widetilde{H}$-distinguished representations of $G$. Building on this ideas, Sakellaridis and Venkatesh \cite{yv:sph}, have stated a conjecture that relates the harmonic analysis of the space $L^{2}(X)$ with the group $\check{G}_{X}$. The ideal result in this direction is the following: Given a $G$-spherical variety X, we want to find a group $G_{X}$ and a correspondence
\[
 \Theta: A\subset \hat{G}_{X} \longrightarrow \hat{G},
\]
between the unitary duals of $G$ and $G_{X}$, with the following properties:
\begin{enumerate}
 \item If $\pi \in A$ has $L$-parameter $\phi: WD_{\k}  \longrightarrow  {}^L G_{X}$, then $\Theta(\pi)$ has $L$-parameter $i\circ \phi : WD_{\k}  \longrightarrow {}^L G$, where $i$ is the natural inclusion of ${}^{L}G_{X}$ into ${}^{L}G$.
\item We have the following spectral decomposition:
\[
 L^{2}(X)\cong \int_{A}M(\pi)\otimes \Theta(\pi)\, d\mu(\pi),
\]
where $\mu$ is the Plancherel measure of $G_{X}$ restricted to $A$, and $M(\pi)$ is some multiplicity space.
\end{enumerate}

We will consider the following classical example to illustrate this ideas. Let
\[
 X=S^{n-1}\cong O(n-1,\mathbb{R}) \backslash O(n,\mathbb{R}),
\]
where $S^{n-1}$ is the $(n-1)$-th sphere, and $O(n,\mathbb{R})$ is the group of $n\times n$ orthogonal matrices. We want to understand the decomposition of $L^{2}(S^{n-1})$ under the natural action of $O(n,\mathbb{R})$.

Let $\mathbb{C}[x_{1},\ldots,x_{n}]$ be the space of complex valued polynomials in $n$ variables. This space has a natural action of $O(n,\mathbb{R})$ and, from classical invariant theory,
\[
 \mathbb{C}[x_{1},\ldots,x_{n}] \cong \bigoplus_{k\geq 0} \mathcal{H}^{k}\otimes \mathbb{C}[r^{2}],
\]
where $r^{2}=x_{1}^{2}+\cdots x_{n}^{2}$ and
\[
 \mathcal{H}^{k} = \{p(x)\in \mathbb{C}[x_{1},\ldots,x_{n}] \, | \, \mbox{$\deg p(x)=k$, and $\Delta p=0$}\}.
\]
Here
\[
 \Delta=\frac{\partial^{2}}{\partial x_{1}^{2}}+\cdots + \frac{\partial^{2}}{\partial x_{n}^{2}}
\]
is the Laplace operator. The spaces $\mathcal{H}^{k}$ are irreducible under the action of $O(n,\mathbb{R})$ and, if we restrict this polynomials to the unit circle, we can identify them with square integrable functions on $S^{n-1}$. The functions obtained this way are the so~called spherical harmonics, and it's a classical result that
\[
 L^2(S^{n-1}) \cong \hat{\bigoplus_{k\geq 0}} \mathcal{H}^{k}.
\]

Let
\begin{eqnarray*}
h & = & E + n/2 = x_{1} \frac{\partial}{\partial x_{1}} + \cdots +x_{n} \frac{\partial}{\partial x_{n}}  + n/2 \\
e & = & r^2/2 = (x_{1}^{2} + \cdots + x_{n}^{2})/2 \\
f & = & -\Delta/2 = (\frac{\partial}{\partial x_{1}^{2}} + \cdots +\frac{\partial}{\partial x_{n}^{2}})/2 .
\end{eqnarray*}
Then, an easy calculation shows that
$[h,e]=2e$, $[h,f]=-2f$ and $[e,f]=h$,
i.e., ${\mathfrak sl}(2,\mathbb{C}) \simeq \mbox{Span}_{\mathbb{C}}\{h,e,f\}$. Observe that the action of this differential operators commutes with the action of $O(n,\mathbb{R})$. Furthermore, for all $k\geq 0$, 
\[
 \mathcal{H}^{k}\otimes \mathbb{C}[r^{2}] \cong D_{k+\frac{n}{2}}^{+},
\]
 where $D_{k+\frac{n}{2}}^{+}$ is an irreducible, lowest weight representation of ${\mathfrak sl}(2,\mathbb{C})$  with lowest weight $k+\frac{n}{2}$. This representation integrates to a discrete series representation of $\widetilde{SL}(2,\mathbb{R})$, the double cover of $SL(2,\mathbb{R})$. 

Let $A=\{D_{k+\frac{n}{2}}^{+}\, | \, k\geq 0\}\subset (\widetilde{SL}(2,\mathbb{R}))^{\wedge}$ and define
\[
 \Theta:A \longrightarrow \hat{O(n,\mathbb{R})},
\]
by
\[
 \Theta(D_{k+\frac{n}{2}}^{+})=\mathcal{H}^{k}.
\]
Then, since the $D_{k+\frac{n}{2}}^{+}$ are square integrable, it's clear that
\[
 L^{2}(S^{n-1}) \cong \int_{A}\Theta(\pi) \, d\mu(\pi),
\]
where $\mu$ is the Plancherel measure of $\widetilde{SL}(2,\mathbb{R})$.

We will now describe a family of examples of this kind of correspondence where the space $X$ is not a spherical variety. What this examples will show is that the ideas discussed here have applications beyond the spherical variety setting. To construct this examples we will use Howe's theory of dual pairs.

 Let $\widetilde{Sp}(n,\mathbb{R})$ be the double cover of the symplectic group $Sp(n,\mathbb{R})$. There is a special representation of $\widetilde{Sp}(n,\mathbb{R})$ on $L^2(\mathbb{R}^n)$ called the oscillator representation \cite{h:1989}. Let $G_{1}, G_{2} \subset \widetilde{Sp}(n,\mathbb{R})$ be two reductive subgroups. We say that they form a  {reductive dual pair} if one group is the centralizer of the other one in $\widetilde{Sp}(n,\mathbb{R})$ and viceversa. 
In this case Howe duality theory states that, if we restrict the oscillator representation to the subgroup generated by $G_{1}$ and $G_{2}$, then
\[
 L^2(\mathbb{R}^n) \simeq \int_{\hat{G}_{1}} \pi \otimes \Theta(\pi) d\mu(\pi),
\]
for some measure $\mu$, where $\Theta(\pi)$ is an irreducible representation of $G_{2}$. Even more,
\[
 \Theta(\pi) = \Theta(\pi ')  \Longleftrightarrow \pi = \pi '.
\]
We will focus on the dual pair $Sp(m,\mathbb{R})\times O(p,q)\subset Sp(m(p+q),\mathbb{R})$. Let's start with the case $m=1$. In this case, Howe has shown that, if $p$, $q\geq 2$,
\[
 L^{2}(O(p-1,q)\backslash O(p,q))\cong \int_{\widetilde{SL}(2,\mathbb{R})} Wh_{\chi}(\pi)\otimes \Theta(\pi)\, d\mu(\pi),
\]
%and
%\[
% L^{2}(O(p,q-1)\backslash O(p,q))\cong \int_{\widetilde{SL}(2,\mathbb{R})} Wh_{\overline{\chi}}(\pi)\otimes \Theta(\pi)\, d\mu(\pi),
%\]
where $\mu$ is the Plancherel measure of $Sp(m,\mathbb{R})$, and
\[
 Wh_{\chi}(\pi) = \{\lambda: V_{\pi} \longrightarrow \mathbb{C} \,| \,\mbox{$ \lambda(\pi(n)v)=\chi(n)\lambda(v)$ for all $n\in N$}\},
\]
for some generic character $\chi$ of the unipotent radical, $N$, of some minimal parabolic subgroup $P=MAN$. On the other hand, Wallach \cite{w:vol2} has shown that
\[
 L^{2}(N\backslash \widetilde{SL}(2,\mathbb{R});\chi)\cong \int_{\widetilde{SL}(2,\mathbb{R})} Wh_{\overline{\chi}}(\pi)\otimes \pi\, d\mu(\pi),
\]
%and
%\[
% L^{2}(N\backslash \widetilde{SL}(2,\mathbb{R});\chi)\cong \int_{\widetilde{SL}(2,\mathbb{R})} Wh_{\overline{\chi}}(\pi)\otimes \Theta(\pi0\, d\mu(\pi).
%\]
where
\[
L^2(N\backslash \widetilde{SL}(2,\mathbb{R});\chi)=\left\lbrace f:\widetilde{SL}(2,\mathbb{R}) \longrightarrow \mathbb{C} \, \left| \,\begin{array}{c}
				\mbox{$f(ng)=\chi_{r,s}(n)f(g)$ and } \\
                                \mbox{$\int_{N\backslash \widetilde{SL}(2,\mathbb{R})} |f(g)|^2 \, dNg < \infty$}
                              \end{array}  \right\rbrace\right. .
\]
In other words the Plancherel measure of $L^{2}(O(p,q-1)\backslash O(p,q))$ can be seen as the pullback, under the $\Theta$-lift, of $L^{2}(N\backslash \widetilde{SL}(2,\mathbb{R});\chi)$.

We will now consider the dual pair $Sp(m,\mathbb{R})\times O(p,q)\subset Sp(m(p+q),\mathbb{R})$, with $m >1$. We will assume that we are in the stable range, that is, $p,q > m$. Let $P=MAN$ be a Siegel parabolic subgroup with given Langlands decomposition. Let $\chi$ be a generic character of $N$. In this case generic means that the orbit of $\chi$ in $\hat{N}$ under the action of $M$ is open. Let
\[
 M_{\chi}=\{m\in M \,| \, \chi(m^{-1}nm)=\chi(n)\}
\]
be the stabilizer of $\chi$ in $M$.
 Then there is a natural action of $M_{\chi}\times G$ on
\[
L^2(N\backslash \widetilde{Sp}(m,\mathbb{R});\chi)=\left\lbrace f:\widetilde{Sp}(m,\mathbb{R}) \longrightarrow \mathbb{C} \, \left| \,\begin{array}{c}
				\mbox{$f(ng)=\chi(n)f(g)$ and } \\
                                \mbox{$\int_{N\backslash \widetilde{Sp}(m,\mathbb{R})} |f(g)|^2 \, dNg < \infty$}
                              \end{array}  \right\rbrace\right. .
\]

\begin{theorem}
As a $M_{\chi}\times G$-module
\[
 L^{2}(N\backslash \widetilde{Sp}(m,\mathbb{R});\chi)\cong \int_{\hat{G}} \int_{\hat{M}_{\chi}} W_{\chi,\tau}(\pi)\otimes \check{\tau}\otimes \pi \, d\nu(\tau)\, d\mu(\pi),
\]
where $\mu$, $\nu$, are the Plancherel measures of $\widetilde{Sp}(m,\mathbb{R})$ and $M_{\chi}$  respectively, and $W_{\chi,\tau}(\pi)$ is some multiplicity space that depends on $\chi$ and $\tau$. Furthermore, if $M_{\chi}$ is compact, then $W_{\chi,\tau}(\pi)\cong Wh_{\chi,\tau}(\pi)$, where
\[
 Wh_{\chi,\tau}(\pi) = \{ \lambda:V_{\pi}\longrightarrow V_{\tau} \,|\,\mbox{$ \lambda(\pi(mn)v)=\chi(n)\tau(m)\lambda(v)$ for all $m\in M_{\chi}$, $n\in N$}\}.
\]
\end{theorem}
 The purpose of this chapter is to use the explicit formulas for the action of $P\times G$ on the oscillator representation and this result to show that
\begin{theorem}
If $r+s=m$, then there exists a generic character, $\chi_{r,s}$, of $N$ such that
\[
 L^{2}(O(p-r,q-s)\backslash O(p,q)) \cong \int_{\widetilde{Sp}(m,\mathbb{R})^{\wedge}} \int_{\hat{M}_{\chi_{r,s}}} W_{\chi_{r,s},\tau}(\pi)\otimes \check{\tau}\otimes \pi \, d\nu(\tau)\, d\mu(\pi).
\]
\end{theorem}

\section{Howe Duality for the Symplectic and the Orthogonal Group}

Consider the dual pair $(Sp(m,\mathbb{R}),O(p,q))\subset Sp(mn,\mathbb{R})$, with $n=p+q$, and $p\geq q\geq 2m$. The last condition asserts that we are in the stable range. Let $P=MN$ be the Siegel parabolic subgroup of $Sp(m,\mathbb{R})$ with given Langlands decomposition. In the theory of the oscillator representation there are very explicit formulas for the action of $P \times O(p,q)$ on $L^2(\mathbb{R}^{mn})$ \cite{A:theta,R:explicit,R:weil,R:witt}. To simplify the exposition we will only consider the case where $n$ is even. The case $n$ odd is very similar, but involves taking a double cover of $Sp(m,\mathbb{R})$. To write down the explicit action on the oscillator representation we will identify $\mathbb{R}^{mn}$ with $Hom(\mathbb{R}^{m},\mathbb{R}^{n})$, and we will fix a unitary character $\psi$ of $\mathbb{R}$. The action is then given by
\begin{eqnarray}
 \left(\left[\begin{array}{cc}I & X \\ & I\end{array}\right]\cdot \varphi\right)(T)& = & \psi(\operatorname{tr} XT^tI_{p,q}T)\varphi(T) \label{eq:Naction}  \\
\left(\left[\begin{array}{cc}A &  \\  & A^{-t}\end{array}\right]\cdot \varphi\right)(T)&=&|\det A|^{\frac{n}{2}}\varphi(TA), \qquad A\in GL(m,\mathbb{R}) \label{eq:Maction} \\
(g\cdot \varphi)(T)&=&\varphi(g^{-1}T), \qquad g\in O(p,q), \label{eq:O(p,q)action}
\end{eqnarray}
where $T\in Hom(\mathbb{R}^{m},\mathbb{R}^{n})$. Using this formulas we will describe $L^2(Hom(\mathbb{R}^{m},\mathbb{R}^{n}))$ as a representation of $P\times O(p,q)$. Let
\[
 U=\left\{T\in Hom(\mathbb{R}^{m},\mathbb{R}^{n})\, \left| \, \begin{array}{c} \mbox{$T$ is of maximal rank and the inner } \\ \mbox{product on $T(\mathbb{R}^{m})$ is non-degenerated} \end{array} \right\}. \right.
\]

Observe that $U\subset Hom(\mathbb{R}^{m},\mathbb{R}^{n})$ is open, dense, and its complement has measure $0$. Let $r,s\geq0$ be a pair of integers such that $r+s=m$, and define
\[
 U_{r,s}=\{T\in U\, | \, \mbox{$T(\mathbb{R}^{m})$ has signature $(r,s)$}\}.
\]
It's then clear that
\[
 U=\bigcup_{r+s=m}U_{r,s},
\]
and hence
\begin{equation}
 L^2(Hom(\mathbb{R}^{m},\mathbb{R}^{n}))\cong\bigoplus_{r+s=m}L^2(U_{r,s}). \label{eq:directsum}
\end{equation}
By looking at the formulas for the action of $P\times O(p,q)$ on $L^2(Hom(\mathbb{R}^{m},\mathbb{R}^{n}))$, it is easy to check that the subspaces $L^2(U_{r,s})$ are $P\times O(p,q)$ invariant. 

We will now describe $L^2(U_{r,s})$ as a $P\times O(p,q)$-module. Let $T_{r,s}\in U_{r,s}$ be given by $T_{r,s}e_{i}=e_{p-r+i}$, and define a character $\chi_{r,s}$ on $N$ by the formula
\[
 \chi_{r,s}\left(\left[\begin{array}{cc}I & X \\ & I\end{array}\right]\right)=\psi(\operatorname{tr} X T_{r,s}^{t}I_{p,q}T_{r,s})=\psi(\operatorname{tr} XI_{r,s}).
\]
Let 
\[
M_{r,s}=\{m\in M \, | \, \chi_{r,s}(mnm^{-1})=\chi_{r,s}(n) \}
\]
be the stabilizer of $\chi_{r,s}$ in $M$. We will now identify $M$ with $GL(m,\mathbb{R})$. Observe that then $M_{r,s}$ gets identified with $O(r,s)$. On the other hand we can define an embedding of $O(r,s)\times O(p-r,q-s)$ into $O(p,q)$ by identifying $O(p-r,q-s)$ with the subgroup of $O(p,q)$ that fixes every element in the image of $T_{r,s}$, and $O(r,s)$ with the subgroup that fixes every element in the orthogonal complement of the image of $T_{r,s}$. With this identifications in mind, let $H_{r,s}$ be the stabilizer of $T_{r,s}$ in $M\times O(p,q)$. Observe that there is a subgroup, that we will denote by $\Delta O(r,s)$, of $H_{r,s}$ such that $\Delta O(r,s)\subset O(r,s)\times O(r,s) \subset M\times O(p,q)$, and $H_{r,s}=(\Delta O(r,s)\times O(p-r,q-s))N$. Then from equations (\ref{eq:Naction}), (\ref{eq:Maction}) and (\ref{eq:O(p,q)action}) we have that
\begin{eqnarray}
 L^2(U_{r,s})&\cong  & \operatorname{Ind}_{H_{r,s}  N}^{P\times O(p,q)} 1\otimes \chi_{r,s} \nonumber \\
             &\cong  & \operatorname{Ind}_{(\Delta O(r,s)\times O(p-r,q-s))N}^{P\times O(p,q)}1\otimes 1 \otimes \chi_{r,s} \nonumber \\
            &\cong  & \operatorname{Ind}_{O(r,s)N}^{P}L^{2}(O(p-r,q-s)\backslash O(p,q))\otimes \chi_{r,s}, \label{eq:IndL2}
\end{eqnarray}
where $O(p,q)$ acts on the right on $L^{2}(O(p-r,q-s)\backslash O(p,q))$ and $O(r,s)$ acts on the left. Then, from equations (\ref{eq:directsum}) and (\ref{eq:IndL2}),
\begin{equation}\label{eq:opqpart}
L^2(Hom(\mathbb{R}^{m},\mathbb{R}^{n}))\cong \bigoplus_{r+s=m} \operatorname{Ind}_{O(r,s)N}^{P}L^{2}(O(p-r,q-s)\backslash O(p,q))\otimes \chi_{r,s}.
\end{equation}

Now we will describe the mixed model of the oscillator representation. Observe that, since $p$, $q\geq 2m$, there exists a polarization $\mathbb{R}^{n}=X\oplus U\oplus Y$ such that $X$ and $Y$ are totally isotropic complementary subspaces, and $\dim X=\dim Y=2m$. Let $B=Stab_{X}=\{g\in O(p,q)\, | \, g X\subset X\}$ be the stabilizer of $X$, and let $B=M_{B}N_{B}$ be its Langlands decomposition.  We will now describe the mixed model of the oscillator representation relative to the polarization $\mathbb{R}^{n}=X\oplus U\oplus Y$. Observe that $\mathbb{R}^{mn}=\mathbb{R}^{2m}\otimes \mathbb{R}^{2m}\oplus \mathbb{R}^{m}\otimes \mathbb{R}^{n-4m}$. Identifying $\mathbb{R}^{2m}\otimes \mathbb{R}^{2m}$ and $\mathbb{R}^{m}\otimes \mathbb{R}^{n-4m}$ with $End(\mathbb{R}^{2m})$ and $Hom(\mathbb{R}^{m},\mathbb{R}^{n-4m})$, respectively, we have that
\[
L^{2}(\mathbb{R}^{mn})\cong L^{2}(End(\mathbb{R}^{2m}))\otimes L^{2}(Hom(\mathbb{R}^{m},\mathbb{R}^{n-4m})),
\]
where we interpretate $L^{2}(End(\mathbb{R}^{2m}))\otimes L^{2}(Hom(\mathbb{R}^{m},\mathbb{R}^{n-4m}))$ as the space of square integrable functions on $End(\mathbb{R}^{2m})$ with values in $L^{2}(Hom(\mathbb{R}^{m},\mathbb{R}^{n-4m}))$. Now identifying $M_{B}$ with $GL(2m,\mathbb{R})\times O(p-2m,q-2m)\cong GL(X)\times O(U)$, we have that the action of $Sp(m,\mathbb{R})\times M_{B}$ is given by
\begin{eqnarray}
(A\cdot \phi)(T)(S) & = & |\det A|^{-m} \phi(A^{-1}T)(S) \qquad A \in GL(2m,\mathbb{R}) \\
(h \cdot \phi)(T)(S) & = & \phi(h^{-1}T)(S). \\
(g\cdot \phi)(T)(S) &= & [\tilde{\omega}(g)\phi (Tg)](S) \qquad g\in Sp(m,\mathbb{R}) ,
\end{eqnarray}
where $T\in End(\mathbb{R}^{2m})$, $S\in Hom(\mathbb{R}^{m},\mathbb{R}^{n-4m})$, and $(\tilde{\omega}, L^{2}(Hom(\mathbb{R}^{m},\mathbb{R}^{n-4m}))$ is the oscillator representation associated to the dual pair $(Sp(m,\mathbb{R}),O(U))$. Now observe that $GL(2m,\mathbb{R})\subset End(\mathbb{R}^{2m})$ is open, dense, and its complement has measure $0$. Therefore, if $I_{2m}\in  End(\mathbb{R}^{2m})$ is the identity map, then as a $Sp(2m,\mathbb{R})\times GL(2m,\mathbb{R}) $-module,
\begin{eqnarray*}
L^{2}(\mathbb{R}^{nm}) &\cong & Ind_{Stab_{I_{2m}}}^{Sp(m,\mathbb{R})\times GL(2m,\mathbb{R})} 1\otimes   L^{2}(Hom(\mathbb{R}^{m},\mathbb{R}^{n-4m})) \\
& \cong & \int_{Sp(m,\mathbb{R})^{\wedge}} \pi \otimes Ind_{Sp(m,\mathbb{R})}^{GL(2m,\mathbb{R})} \pi^{\ast} \otimes L^{2}(Hom(\mathbb{R}^{m},\mathbb{R}^{n-4m})) \, d\mu(\pi),
\end{eqnarray*}
where $\mu$ is the Plancherel measure of $Sp(m,\mathbb{R})$. Hence, from the abstract theory of Howe duality,
\begin{equation}\label{eq:HoweL2}
L^{2}(\mathbb{R}^{mn})|_{P\times O(p,q)} \cong \int_{Sp(m,\mathbb{R})^{\wedge}} \pi|_{P}\otimes \Theta(\pi) \, d\mu(\pi).
\end{equation}

In the stable range the representation $\Theta(\pi)$ has been determined by the work of Jian-Shu Li \cite{jsl:stiefel} among others. We are thus left with the problem of decomposing an irreducible tempered representation of $Sp(m,\mathbb{R})$ when restricted to $P$.

Now let's look at $L^2(Sp(m,\mathbb{R}))$ as a $P\times Sp(m,\mathbb{R})$-module. We claim that we have an isomorphism
\[
 L^2(Sp(m,\mathbb{R}))\cong L^2(\hat{N},E,\lambda),
\]
where $E$ is a measurable bundle over $\hat{N}$ with fibers $E_{\chi}\cong \operatorname{Ind}_{N}^{Sp(m,\mathbb{R})} \chi$ for any given $\chi\in \hat{N}$, and $\lambda$ is a Haar measure on $\hat{N}$. The isomorphism is given in the following way: Given $f\in  L^2(Sp(m,\mathbb{R}))$, define $s_{f}\in L^2(\hat{N},E,\lambda)$ by
\[
 s_{f}(\chi)(g)=\int_{N}\chi(n)^{-1}f(ng) \, dn,
\]
where $dn$ is the usual Lebesgue measure on $N$. The way we should interpret
the above formula is that we have an isomorphism $L^2(Sp(m,\mathbb{R}))\cong \operatorname{Ind}_{N}^{Sp(m,\mathbb{R})} \operatorname{Ind}_{1}^{N} 1 \cong \operatorname{Ind}_{N}^{Sp(m,\mathbb{R})} L^2(N)$, and in the last expression we take the Fourier transform on $L^2(N)$.
With this convention the measure $\lambda$ on $\hat{N}$ is the measure dual to $dn$. Now by definition
\begin{eqnarray*}
 s_{R_{g_{1}}f}(\chi)(g)& = &\int_{N}\chi(n)^{-1}R_{g_{1}}f(ng) \, dn=\int_{N}\chi(n)^{-1}f(ngg_{1}) \, dn\\ & = & s_{f}(\chi)(gg_{1})=(R_{g_{1}}s_{f}(\chi))(g)
\end{eqnarray*}
and
\begin{eqnarray*}
 s_{L_{p}f}(\chi)(g)& = & \int_{N}\chi(n)^{-1}L_{p}f(ng) \, dn=\int_{N}\chi(n)^{-1}f(p^{-1}npp^{-1}g) \, dn \\
                    & = & \int_{N}\chi(pnp^{-1})^{-1}\delta(p)f(np^{-1}g) \, dn\\ & = & \int_{N}(p^{-1}\chi)(n)^{-1}\delta(p)f(np^{-1}g) \, dn  \\
                    & = & \delta(p)s_{f}(p^{-1}\chi)(p^{-1}g)=[\delta(p)L_{p}s_{f}(p^{-1}\chi)](g),
\end{eqnarray*}
where $\delta$ is the modular function of $P$. This means that the action of $P\times Sp(m,\mathbb{R})$ on $L^2(\hat{N},E,\eta)$ is given by a vector bundle action, and hence, if $\Omega$ is the set of open orbits for the action of $M$ on $\hat{N}$, then 
\[
\Omega=\bigcup_{r+s=m} \Omega_{r,s},
\]
 where $\Omega_{r,s}$ is the orbit of the character $\chi_{r,s}$ defined before. Therefore
\begin{eqnarray*}
 L^2(Sp(m,\mathbb{R}))& \cong & L^2(\hat{N},E,\lambda) \\
 & \cong & \bigoplus_{r+s=m} \operatorname{Ind}_{M_{r,s}  N \times Sp(m,\mathbb{R})}^{P\times Sp(m,\mathbb{R})}\operatorname{Ind}_{N}^{Sp(m,\mathbb{R})}\chi_{r,s} \\
 & \cong & \bigoplus_{r+s=m} \operatorname{Ind}_{M_{r,s}  N \times Sp(m,\mathbb{R})}^{P\times Sp(m,\mathbb{R})}\operatorname L^2(N\backslash   Sp(m,\mathbb{R});\chi_{r,s}).
\end{eqnarray*}
But, then according to equation (\ref{Bessel-Plancherel}),
\begin{equation}\label{eq:BesselPlancherelmeasure}
 L^2(Sp(m,\mathbb{R})) \cong  \bigoplus_{r+s=m} \operatorname{Ind}_{M_{r,s}  N}^{P} \int_{Sp(m,\mathbb{R})^{\wedge}} \int_{O(r,s)^{\wedge}} W_{\chi_{r,s},\tau}(\pi)\otimes \tau^{\ast}\otimes \pi \, d\eta(\tau) d \mu_{r,s}(\pi),
\end{equation}
where $\eta$ is the Plancherel measure of $O(r,s)$ and $\mu_{r,s}$ is the Bessel-Plancherel measure.
On the other hand, the Harish-Chandra Plancherel theorem says that
\begin{equation}
 L^2(Sp(m,\mathbb{R}))\cong  \int_{Sp(m,\mathbb{R})^{\wedge}} \pi^{\ast}|_{P}\otimes \pi\, d\mu(\pi). \label{L2GHarish}
\end{equation}
Now from equations (\ref{L2GHarish}) and (\ref{eq:BesselPlancherelmeasure}) we conclude that
\begin{equation}
\pi^{\ast}|_{P} \cong \bigoplus_{r+s=m} \int_{O(r,s)^{\wedge}} W_{\chi_{r,s},\tau}(\pi)\otimes   \operatorname{Ind}_{M_{\chi}  N}^{P} \tau^{\ast} \, d\eta(\tau). \label{eq:restrictiontoparabolic}
\end{equation}
 From this and equation (\ref{eq:HoweL2}) we have that
\begin{eqnarray}
L^2(Hom(\mathbb{R}^{m},\mathbb{R}^{n})) &  \cong & \int_{Sp(m,\mathbb{R})^{\wedge}} \pi^{\ast}|_{P}\otimes \Theta(\pi^{\ast}) \, d\mu(\pi)   \nonumber \\
%& \cong & \int_{Sp(m,\mathbb{R})^{\wedge}} \bigoplus_{r+s=m} \int_{O(r,s)^{\wedge}} W_{\chi_{r,s},\tau}(\pi)\otimes   \operatorname{Ind}_{M_{\chi}  N}^{P} \tau^{\ast} \otimes \Theta(\pi^{\ast})\, d\eta(\tau) \, d\mu(\pi) \nonumber \\
& \cong & \bigoplus_{r+s=m} \int_{Sp(m,\mathbb{R})^{\wedge}}  \int_{O(r,s)^{\wedge}} W_{\chi_{r,s},\tau}(\pi) \otimes   \operatorname{Ind}_{M_{\chi}  N}^{P} \tau^{\ast}\nonumber \\
& &  \otimes \Theta(\pi^{\ast})\, d\eta(\tau) \, d\mu_{r,s}(\pi). 
\end{eqnarray}
But then, from this and equation (\ref{eq:opqpart}), we have that as an $O(r,s)\times O(p,q)$-module
\[
L^{2}(O(p-r,q-s)\backslash O(p,q)) \cong \int_{Sp(m,\mathbb{R})^{\wedge}}\int_{O(r,s)^{\wedge}} W_{\chi_{r,s},\tau}(\pi)\otimes    \tau^{\ast} \, \otimes \Theta(\pi^{\ast})\, d\eta(\tau)\, d\mu_{r,s}(\pi).
\] 

\section{The Dual Pair $(SL(2,\mathbb{R})$, $O(V))$ outside stable range}

The results obtained in the former section can be further refined when we restrict ourselves to the case $n=1$, i.e., to the dual pair $(SL(2,\mathbb{R})$, $O(V))$.

\subsection{The case $O(V)=O(n)$}

From classical invariant theory we know that as an $SL(2,\mathbb{R}) \times O(n)$-module
\[
L^2(\mathbb{R}^{n}) \simeq \bigoplus_{k\geq 0} \mbox{\rmfamily H}^{k}\otimes \C[r^{2}]
\]
where $\mbox{\rmfamily H}^{k}$ are the harmonic polynomials of degree $k$, and ${\mathfrak sl}(2,\mathbb{R})$ acts via the operators
\[
 e=r^2/2 \qquad f=-\Delta/2, \qquad h=E+n/2
\]
where $E$ is the Euler operator.

In this case the $\theta$-correspondence relates the irreducible representations of $O(n)$ on $\mbox{\rmfamily H}^{k}$ with the irreducible representations of $\widetilde{SL}(2,\mathbb{R})$ with lowest weight $k+n/2$.

\subsection{The case $O(V)=O(p,1)$ }

We will now consider the dual pair $(SL(2,\mathbb{R}),O(p,1))$. Once again we will only consider the case where $n=p+1$ is even, and leave to the reader the modifications needed for the case where $n$ is odd. Let $P=MN\subset SL(2,\mathbb{R})$ be the minimal parabolic subgroup consisting of all upper triangular matrices, with given Langlands decomposition. Let $\{e_{1},\ldots, e_{n}\}$ be the canonical basis of $\mathbb{R}^{n}$, and assume that we have an inner product $\langle, \rangle$ such that $\langle e_{i}, e_{j}\rangle =0$ if $i\neq j$, $\langle e_{i},e_{i}\rangle =1$ for $1\leq i\leq p$, and $\langle e_{n}, e_{n}\rangle=-1$. Then the oscillator representation associated to the dual pair $(SL(2,\mathbb{R}),O(p,1))$ can be realized on the space $L^{2}(\mathbb{R}^{n})$, and the action of $P\times O(p,1)$ is given by the following formulas:
\begin{eqnarray}
 \left( \left[ \begin{array}{cc} 1 & x \\ & 1 \end{array}\right] \cdot \varphi \right) (v)&  = & \psi(x\langle v,v\rangle)f(v) \qquad \forall x\in \mathbb{R}, \label{actionN2} \\
\left( \left[ \begin{array}{cc} \lambda &  \\ & \lambda^{-1} \end{array}\right] \cdot \varphi \right) (v)&  = & |\lambda|^{n/2}\varphi(\lambda v) \qquad \forall \lambda \in \mathbb{R}^*, \label{actionlamba2} \\
(g\cdot \varphi)(v) & = & \varphi(g^{-1}v) \qquad \forall g\in O(p,1), \label{actionO2}
\end{eqnarray}
where $ v\in \mathbb{R}^{n}$. Let $U=\{v\in \mathbb{R}^{n}\, | \,\langle v, v\rangle \neq 0 \}$. Then $U$ is open, dense, and its complement has measure 0. Observe that $U=U_{+}\cup U_{-}$, where $U_{+}=\{v\in \mathbb{R}^{n}\, | \,\langle v, v\rangle > 0 \}$ and $U_{-}=\{v\in \mathbb{R}^{n}\, | \,\langle v, v\rangle < 0 \}$ . Hence
\begin{equation}\label{eq:sum2}
L^{2}(\mathbb{R}^{n}) \cong L^{2}(U_{+}) \oplus L^{2}(U_{-}).
\end{equation}
We will identify $O(p,\mathbb{R})$ with the subgroup of $O(p,1)$ that fixes $e_{n}$, $O(p-1,1)$ with the subgroup of $O(p,1)$ fixing $e_{1}$, and $O(1,\mathbb{R})$ with the center of $SL(2,\mathbb{R})$. We will also set $H_{+}=Stab_{e_{1}}$ and $H_{-}=Stab_{e_{n}}$. Then, putting everything together, we get that
\begin{eqnarray}
L^{2}(U_{+})& \cong & Ind_{H_{+}N}^{P\times O(p,1)} 1\otimes \chi \nonumber \\
 & \cong & Ind_{O(1,\mathbb{R})N}^{P}L^{2}(O(p-1,1)\backslash O(p,1))\otimes \chi, \label{eq:U+2}
\end{eqnarray}
and 
\begin{eqnarray}
L^{2}(U_{-})& \cong & Ind_{H_{-}N}^{P\times O(p,1)} 1\otimes \chi \nonumber \\
 & \cong & Ind_{O(1,\mathbb{R})N}^{P}L^{2}(O(p,\mathbb{R})\backslash O(p,1))\otimes \overline{\chi}, \label{eq:U-2}
\end{eqnarray}
where 
\[
\chi \left( \left[ \begin{array}{cc} 1 & x \\ & 1 \end{array} \right]\right)=e^{ix} \qquad \mbox{and} \qquad \overline{\chi} \left( \left[ \begin{array}{cc} 1 & x \\ & 1 \end{array} \right]\right)=e^{-ix} \qquad \forall x\in \mathbb{R}.
\]

Now we will describe the mixed model for the oscillator representation associated to a complete polarization $\mathbb{R}^{n}=X\oplus U\oplus Y$. Observe that $\dim X=\dim Y=1$. Let $B=Stab_{X}=\{g\in O(p,1) \, | \, g X\subset X\}$, and let $B=M_{B}N_{B}$ be a Langlands decomposition. Observe that $M_{B}\cong \mathbb{R}^{\ast}\times O(p-1,\mathbb{R})$, where $\mathbb{R}^{\ast}=\mathbb{R}-\{0\}$. Since $\mathbb{R}^{n}\cong \mathbb{R}^{2}\oplus \mathbb{R}^{p-1}$, we can identify $L^2(\mathbb{R}^{n})$ with $L^2(\mathbb{R}^{2})\otimes L^2(\mathbb{R}^{p-1})$, which we interpretate as the space of square integrable functions on $\mathbb{R}^{2}$ with values on $L^2(\mathbb{R}^{p-1})$. With this conventions the action of $ SL(2,\mathbb{R})\times M_{B}$ is given by the following formulas: 
\begin{eqnarray}
(\lambda  \cdot \phi)((x,y))(v) & = &|\lambda|^{-1}\phi(\lambda^{-1}x,\lambda^{-1}y)(v) \qquad \mbox{for $\lambda \in \mathbb{R}^{\ast}$,} \label{actionN3}\\
(h\cdot \phi)((x,y))(v)           & = & \phi((x,y))(h^{-1}v) \qquad \mbox{for $h\in O(p-1,\mathbb{R})$,} \label{actionO3}\\
(g\cdot \phi)((x,y))(v)          & = & [\omega_{p-1}(g)\phi((x,y)g)](v), \qquad \mbox{for $g\in SL(2,\mathbb{R})$,}  \label{actionSL23} 
\end{eqnarray}
where $x$, $y \in \mathbb{R}$, $v\in \mathbb{R}^{p-1}$, and $(\omega_{p-1}, L^{2}(\mathbb{R}^{p-1}))$ is the oscillator representation associated to the dual pair $(SL(2,\mathbb{R}), O(p-1,\mathbb{R}))$. Now observe that $\mathbb{R}^{2}-\{0\}$ and $\mathbb{R}^{p-1}-\{0\}$ are open dense subsets of $\mathbb{R}^{2}$ and $\mathbb{R}^{p-1}$, respectively, and its complements have measure 0. From this observation, and equations (\ref{actionN3}), (\ref{actionO3}) and (\ref{actionSL23}) we have that as an $ SL(2,\mathbb{R})\times M_{B}$-module
\[
L^2(\mathbb{R}^{n}) \cong Ind_{HN}^{SL(2,\mathbb{R})\times \mathbb{R}^{\ast}} L^{2}(\mathbb{R}^{p-1})\otimes \chi,
\]
where 
\[
H=\left\{\left. \left( \left[\begin{array}{cc} \lambda & \\ & \lambda^{-1}\end{array}\right], \lambda \right) \, \right|  \, \lambda \in \mathbb{R}^{\ast}\right\} \subset SL(2,\mathbb{R})\times \mathbb{R}^{\ast}.
\]
Let $\Delta\{\pm 1\}$ be the subgroup of $H$ whose projection onto $\mathbb{R}^{\ast}$ is precisely $\{\pm 1\}$. Then
\begin{eqnarray*}
L^2(\mathbb{R}^{p+1})&\cong &Ind_{HN}^{SL(2,\mathbb{R})\times \mathbb{R}^{\ast}} L^{2}(\mathbb{R}^{p-1})\otimes \chi \\
      & \cong & Ind_{HN}^{SL(2,\mathbb{R})\times \mathbb{R}^{\ast}} Ind_{\Delta\{\pm 1\}N}^{HN}L^{2}(S^{p-2})\otimes  \chi \\
 & \cong & Ind_{N}^{SL(2,\mathbb{R})}\chi \otimes Ind_{\{\pm 1\}}^{\mathbb{R}^{\ast}} L^{2}(S^{p-2}),
\end{eqnarray*}
where $S^{p-2} \subset \mathbb{R}^{p-1}$ is the $p-2$-dimensional sphere.
From all this formulas it is immediate that
\begin{equation}\label{eq:whittakerplancherel2}
L^2(\mathbb{R}^{n})|_{P\times O(p,1)} = \int_{SL(2,\mathbb{R})^{\wedge}} \pi|_{P} \otimes \Theta(\pi) \, d\mu_{\chi}(\pi),
\end{equation}
where $\mu_{\chi}$ is the Plancherel-Whittaker measure of $L^{2}(N\backslash SL(2,\mathbb{R});\chi)$.

Finally, from the usual Plancherel-Whittaker theorem \cite{w:vol2} we have that if $(\pi,H_{\pi})$ is a tempered representation of $SL(2,\mathbb{R})$, then
\[
\pi|_{P}\cong Ind_{O(1,\mathbb{R}) N}^{P} Wh_{\chi}(\pi)\otimes \chi \bigoplus Ind_{O(1,\mathbb{R}) N}^{P} Wh_{\overline{\chi}}(\pi)\otimes \overline{\chi}.
\]
But then, from equations (\ref{eq:sum2}), (\ref{eq:U+2}), (\ref{eq:U-2}) and (\ref{eq:whittakerplancherel2}), we have that
\begin{equation}
L^{2}(O(p-1,1)\backslash O(p,1)) \cong \int_{SL(2,\mathbb{R})^{\wedge}} Wh_{\chi}(\pi)\otimes \Theta(\pi)\, d\mu_{\chi}(\pi)
\end{equation}
and
\begin{equation}\label{eq:spherical}
L^{2}(O(p,\mathbb{R})\backslash O(p,1)) \cong \int_{SL(2,\mathbb{R})^{\wedge}} Wh_{\overline{\chi}}(\pi)\otimes \Theta(\pi)\, d\mu_{\chi}(\pi).
\end{equation}
Observe that $L^{2}(O(p,\mathbb{R})\backslash O(p,1))$ has no discrete spectrum. Effectively from equation (\ref{eq:spherical}) only the theta lift of a representation with a positive \emph{and} a negative whittaker model can appear in the spectral decomposition of $L^{2}(O(p,\mathbb{R})\backslash O(p,1))$.

\end{document}